\documentclass[10pt]{article}
\usepackage[hmargin=20mm,vmargin=20mm]{geometry}

\usepackage[mathscr]{euscript}

\usepackage{appendix}

\usepackage{bbm}

\usepackage{makecell}
\usepackage{tablefootnote}

\let\oldbibliography\thebibliography
\renewcommand{\thebibliography}[1]{\oldbibliography{#1}
\setlength{\itemsep}{0pt}} 

\usepackage{float}

\usepackage[mathscr]{euscript}
\usepackage{relsize}

\usepackage{pgfplots}
\pgfplotsset{width=10cm,compat=1.9}

\usepackage{multirow, hhline}
\usepackage{verbatim}
\usepackage[shortlabels]{enumitem}
\usepackage{tikz-cd}
\usepackage{xcolor}
\usepackage{hyperref}

\usepackage{amssymb} 
\usepackage{amsbsy,
  amsthm, amsmath,amstext, amsopn} 
  \usepackage[all]{xy}
\usepackage{amsfonts} 
\usepackage{amscd} 
\usepackage{stmaryrd}
\usepackage{mathabx}

\usepackage[affil-it]{authblk}

\usepackage{titlesec}
\titleformat{\subsubsection}[runin]
       {\normalfont\itshape}
       {\thesubsubsection.}
       {0.5em}
       {}
       []

\newtheorem{introthm}{Theorem}

\newtheorem{thm}[subsubsection]{Theorem}

\newtheorem{lemma}[subsubsection]{Lemma}

\newtheorem{cor}[subsubsection]{Corollary}
\newtheorem{prop}[subsubsection]{Proposition}

\theoremstyle{definition}

\newtheorem{defn}[subsubsection]{Definition}

\newtheorem{exmp}[subsubsection]{Example}

\newtheorem{warning}[subsubsection]{{\fontencoding{U}\fontfamily{futs}\selectfont\char 66\relax} Warning {\fontencoding{U}\fontfamily{futs}\selectfont\char 66\relax}}

\theoremstyle{remark}

\newtheorem{rmk}[subsubsection]{Remark}

\DeclareMathOperator{\ind}{ind}

\DeclareMathOperator*{\colim}{colim}

\DeclareMathOperator{\Ch}{Ch}

\DeclareMathOperator{\Dol}{Dol}

\usepackage{etoolbox}
\patchcmd{\subsubsection}{-.5em}{0em}{}{}

\begin{document}

\numberwithin{equation}{section}

\setcounter{tocdepth}{2}

\newcommand{\hs}{\mbox{\hspace{.4em}}}
\newcommand{\ds}{\displaystyle}
\newcommand{\bd}{\begin{displaymath}}
\newcommand{\ed}{\end{displaymath}}
\newcommand{\bcd}{\begin{CD}}
\newcommand{\ecd}{\end{CD}}

\newcommand{\minus}{\scalebox{0.5}[1.0]{$-$}}
\renewcommand{\ng}{\minus}

\newcommand{\shear}{{\mathbin{\mkern-6mu\fatslash}}}
\newcommand{\unshear}{\,{\mathbin{\mkern-6mu\fatbslash}}}

\newcommand{\indcoh}{\operatorname{QC^!}}
\newcommand{\on}{\operatorname}
\newcommand{\proj}{\operatorname{Proj}}
\newcommand{\bproj}{\underline{\operatorname{Proj}}}
\newcommand{\spec}{\operatorname{Spec}}
\newcommand{\Spec}{\operatorname{Spec}}
\newcommand{\bspec}{\underline{\operatorname{Spec}}}
\newcommand{\pline}{{\mathbf P} ^1}
\newcommand{\aline}{{\mathbf A} ^1}
\newcommand{\pplane}{{\mathbf P}^2}
\newcommand{\cone}{\operatorname{cone}}
\newcommand{\aplane}{{\mathbf A}^2}
\newcommand{\coker}{{\operatorname{coker}}}
\newcommand{\ldb}{[[}
\newcommand{\rdb}{]]}

\newcommand{\IndFin}{\operatorname{IndFin}}
\newcommand{\Sym}{\operatorname{Sym}^{\bullet}}
\newcommand{\Symp}{\operatorname{Sym}}
\newcommand{\Pic}{\bf{Pic}}
\newcommand{\Aut}{\operatorname{Aut}}
\newcommand{\codim}{\operatorname{codim}}
\newcommand{\PAut}{\operatorname{PAut}}

\newcommand{\KCoh}{\operatorname{KCoh}}

\newcommand{\Sec}{\operatorname{Sec}}
\newcommand{\Fqbar}{\overline{\mathbb{F}_q}}
\newcommand{\Fq}{{\mathbb{F}_q}}

\newcommand{\too}{\twoheadrightarrow}
\newcommand{\C}{{\mathbb C}}
\newcommand{\Z}{{\mathbb Z}}
\newcommand{\Q}{{\mathbb Q}}
\newcommand{\R}{{\mathbb R}}
\newcommand{\Cx}{{\mathbb C}^{\times}}
\newcommand{\Cbar}{\overline{\C}}
\newcommand{\Cxbar}{\overline{\Cx}}
\newcommand{\cA}{{\mathcal A}}
\newcommand{\fA}{{\mathfrak A}}
\newcommand{\cS}{{\mathcal S}}
\newcommand{\cV}{{\mathcal V}}
\newcommand{\cM}{{\mathcal M}}
\newcommand{\bA}{{\mathbb A}}

\newcommand{\cB}{{\mathcal B}}
\newcommand{\cC}{{\mathcal C}}
\newcommand{\cD}{{\mathcal D}}
\newcommand{\D}{{\mathcal D}}
\newcommand{\cs}{{\mathbf C} ^*}
\newcommand{\boldc}{{\mathbf C}}
\newcommand{\cE}{{\mathcal E}}
\newcommand{\cF}{{\mathcal F}}
\newcommand{\bF}{{\mathbb F}}
\newcommand{\cG}{{\mathcal G}}
\newcommand{\G}{{\mathbb G}}
\newcommand{\cH}{{\mathcal H}}
\newcommand{\bH}{{\mathbf H}}
\newcommand{\CI}{{\mathcal I}}
\newcommand{\cJ}{{\mathcal J}}
\newcommand{\cK}{{\mathcal K}}
\newcommand{\cL}{{\mathcal L}}
\newcommand{\baL}{{\overline{\mathcal L}}}
\newcommand{\M}{{\mathcal M}}
\newcommand{\Mf}{{\mathfrak M}}
\newcommand{\bM}{{\mathbb M}}
\newcommand{\bm}{{\mathbf m}}
\newcommand{\cN}{{\mathcal N}}
\newcommand{\theo}{\mathcal{?}{O}}
\newcommand{\cP}{{\mathcal P}}
\newcommand{\cR}{{\mathcal R}}
\newcommand{\Pp}{{\mathbb P}}
\newcommand{\boldp}{{\mathbf P}}
\newcommand{\boldq}{{\mathbf Q}}
\newcommand{\bbL}{{\mathbf L}}
\newcommand{\cQ}{{\mathcal Q}}
\newcommand{\cO}{{\mathcal O}}
\newcommand{\cT}{{\mathcal T}}
\newcommand{\Oo}{{\mathcal O}}
\newcommand{\cY}{{\mathcal Y}}
\newcommand{\OX}{{\Oo_X}}
\newcommand{\OY}{{\Oo_Y}}
\newcommand{\cZ}{{\mathcal Z}}
\newcommand{\fib}{\operatorname{fib}}
\newcommand{\DMod}{\mathcal{D}}
\newcommand{\cDMod}{\breve{\mathcal{D}}}
\newcommand{\rnD}{\cDMod}
\newcommand{\otY}{{\underset{\OY}{\ot}}}
\newcommand{\otX}{{\underset{\OX}{\ot}}}
\newcommand{\cU}{{\mathcal U}}\newcommand{\cX}{{\mathcal X}}
\newcommand{\cW}{{\mathcal W}}
\newcommand{\boldz}{{\mathbf Z}}
\newcommand{\Rees}{\operatorname{Rees}}
\newcommand{\IC}{\IndCoh}
\newcommand{\ssupp}{\operatorname{SS}}
\newcommand{\qgr}{\operatorname{q-gr}}
\newcommand{\gr}{\operatorname{gr}}
\newcommand{\rk}{\operatorname{rk}}
\newcommand{\Sh}{\operatorname{Sh}}
\newcommand{\Shv}{\operatorname{Sh}}
\newcommand{\SH}{{\underline{\operatorname{Sh}}}}
\newcommand{\End}{\operatorname{End}}
\newcommand{\uEnd}{\underline{\operatorname{End}}}
\newcommand{\Hom}{\operatorname{Hom}}
\newcommand{\uHom}{\underline{\operatorname{Hom}}}
\newcommand{\Sing}{\operatorname{Sing}}
\newcommand{\uHomY}{\uHom_{\OY}}
\newcommand{\uHomX}{\uHom_{\OX}}
\newcommand{\Ext}{\operatorname{Ext}}
\newcommand{\bExt}{\operatorname{\bf{Ext}}}
\newcommand{\Tor}{\operatorname{Tor}}

\newcommand*\brslash{\vcenter{\hbox{\includegraphics[height=3.6mm]{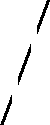}}}}

\newcommand{\inv}{^{-1}}
\newcommand{\airtilde}{\widetilde{\hspace{.5em}}}
\newcommand{\airhat}{\widehat{\hspace{.5em}}}
\newcommand{\nt}{^{\circ}}
\newcommand{\del}{\partial}

\newcommand{\supp}{\operatorname{supp}}
\newcommand{\GK}{\operatorname{GK-dim}}
\newcommand{\hd}{\operatorname{hd}}
\newcommand{\pt}{\operatorname{pt}}
\newcommand{\id}{\operatorname{id}}
\newcommand{\res}{\operatorname{res}}
\newcommand{\lrar}{\leadsto}
\newcommand{\im}{\operatorname{Im}}
\newcommand{\hh}{HH}
\newcommand{\hn}{HN}
\newcommand{\hc}{HC}
\newcommand{\hp}{HP}

\newcommand{\TF}{\operatorname{TF}}
\newcommand{\Bun}{\operatorname{Bun}}

\newcommand{\F}{\mathcal{F}}
\newcommand{\Ff}{\mathbb{F}}
\newcommand{\nthord}{^{(n)}}
\newcommand{\Gr}{{\mathfrak{Gr}}}

\newcommand{\BB}{\mathbb{B}}

\newcommand{\BGA}{B\mathbb{G}_{a}^{\mathrm{gr}}}

\newcommand{\Fr}{\operatorname{Fr}}
\newcommand{\GL}{\operatorname{GL}}
\newcommand{\Perv}{\operatorname{Perv}}
\newcommand{\gl}{\mathfrak{gl}}
\newcommand{\SL}{\operatorname{SL}}
\newcommand{\KPerf}{\operatorname{KPerf}}
\newcommand{\ff}{\footnote}
\newcommand{\ot}{\otimes}
\def\Ext{\operatorname {Ext}}
\def\Hom{\operatorname {Hom}}
\def\Ind{\operatorname {Ind}}
\newcommand{\MHM}{\operatorname{MHM}}

\def\bbZ{{\mathbb Z}}

\newcommand{\Irr}{\mathrm{Irr}}

\newcommand{\nc}{\newcommand}
\nc{\ol}{\overline} \nc{\cont}{\on{cont}} \nc{\rmod}{\on{mod}}
\nc{\Mtil}{\widetilde{M}} \nc{\wb}{\overline} \nc{\wt}{\widetilde}
\nc{\wh}{\widehat} \nc{\sm}{\setminus} \nc{\mc}{\mathcal}

\nc{\wc}{\widecheck}
\nc{\mbb}{\mathbb}  \nc{\K}{{\mc K}} \nc{\Kx}{{\mc K}^{\times}}
\nc{\Ox}{{\mc O}^{\times}} \nc{\unit}{{\bf \on{unit}}}
\nc{\boxt}{\boxtimes} \nc{\xarr}{\stackrel{\rightarrow}{x}}

\newcommand{\kz}{\widecheck}
\newcommand{\kzz}[1]{{#1}^!}

\nc{\Ga}{\G_a}
 \nc{\PGL}{{\on{PGL}}}
 \nc{\PU}{{\on{PU}}}

\nc{\h}{{\mathfrak h}} \nc{\kk}{{\mathfrak k}}
 \nc{\Gm}{\G_m}
\nc{\Gabar}{\wb{\G}_a} \nc{\Gmbar}{\wb{\G}_m} \nc{\Gv}{G^\vee}
\nc{\Tv}{T^\vee} \nc{\Bv}{B^\vee} 
\nc{\g}{{\mathfrak g}}
\nc{\gv}{{\mathfrak g}^\vee} \nc{\RGv}{\on{Rep}\Gv}
\nc{\RTv}{\on{Rep}T^\vee}
 \nc{\Flv}{{\mathcal B}^\vee}
 \nc{\TFlv}{T^*\Flv}
 \nc{\Fl}{{\mathfrak Fl}}
\nc{\RR}{{\mathcal R}} \nc{\Nv}{{\mathcal{N}}^\vee}
\nc{\St}{{\mathcal St}} \nc{\ST}{{\underline{\mathcal St}}}
\nc{\Hec}{{\bf{\mathcal H}}} \nc{\Hecblock}{{\bf{\mathcal
H_{\alpha,\beta}}}} \nc{\dualHec}{{\bf{\mathcal H^\vee}}}
\nc{\dualHecblock}{{\bf{\mathcal H^\vee_{\alpha,\beta}}}}
\newcommand{\ramBun}{{\bf{Bun}}}
\newcommand{\ramBuno}{\ramBun^{\circ}}

\nc{\Buntheta}{{\bf Bun}_{\theta}} \nc{\Bunthetao}{{\bf
Bun}_{\theta}^{\circ}} \nc{\BunGR}{{\bf Bun}_{G_\R}}
\nc{\BunGRo}{{\bf Bun}_{G_\R}^{\circ}}
\nc{\HC}{{\mathcal{HC}}}
\nc{\risom}{\stackrel{\sim}{\to}} \nc{\Hv}{{H^\vee}}
\nc{\bS}{{\mathbf S}}
\def\Rep{\operatorname {Rep}}
\def\Conn{\operatorname {Conn}}

\nc{\Vect}{{\operatorname{Vect}}}
\nc{\Hecke}{{\operatorname{Hecke}}}

\newcommand{\ZZ}{{Z_{\bullet}}}
\nc{\HZ}{{\mc H}\ZZ} \nc{\eps}{\epsilon}

\nc{\CN}{\mathcal N} \nc{\BA}{\mathbb A}
\nc{\XYX}{X\times_Y X}

\nc{\ul}{\underline}

\nc{\bn}{\mathbf n} \nc{\Sets}{{\on{Sets}}} \nc{\Top}{{\on{Top}}}

\nc{\Simp}{{\mathbf \Delta}} \nc{\Simpop}{{\mathbf\Delta^\circ}}

\nc{\Cyc}{{\mathbf \Lambda}} \nc{\Cycop}{{\mathbf\Lambda^\circ}}

\nc{\Mon}{{\mathbf \Lambda^{mon}}}
\nc{\Monop}{{(\mathbf\Lambda^{mon})\circ}}

\nc{\Aff}{{\on{Aff}}} \nc{\Sch}{{\on{Sch}}}
\newcommand{\bE}{\mathbb{E}}

\nc{\bul}{\bullet}
\nc{\module}{{\operatorname{-mod}}}

\nc{\dstack}{{\mathcal D}}

\nc{\BL}{{\mathbb L}}

\nc{\BD}{{\mathbb D}}

\nc{\BR}{{\mathbb R}}

\nc{\BT}{{\mathbb T}}

\nc{\SCA}{{\mc{SCA}}}
\nc{\DGA}{{\mc DGA}}

\nc{\DSt}{{DSt}}

\nc{\lotimes}{{\otimes}^{\mathbf L}}

\nc{\bs}{\backslash}

\nc{\Lhat}{\widehat{\mc L}}

\newcommand{\Coh}{{\on{Coh}}}

\nc{\QC}{\operatorname{QCoh}}
\nc\Perf{\on{Perf}}
\nc{\Cat}{{\on{Cat}}}
\nc{\dgCat}{{\on{dgCat}}}
\nc{\bLa}{{\mathbf \Lambda}}
\nc{\QCoh}{\QC}
\newcommand{\IndCoh}{\operatorname{IndCoh}}

\nc{\RHom}{\mathbf{R}\hspace{-0.15em}\on{Hom}}
\nc{\REnd}{\mathbf{R}\hspace{-0.15em}\on{End}}
\nc{\oo}{\infty}
\nc\Mod{\on{Mod}}
\nc\Comod{\on{Comod}}

\nc\fh{\mathfrak h}
\nc\al{\alpha}
\nc\la{\alpha}
\nc\BGB{B\bs G/B}
\nc\QCb{QC^\flat}
\nc\qc{\cQ}

\nc{\fg}{\mathfrak g}

\nc{\fn}{\mathfrak n}
\nc{\Map}{\on{Map}} \nc{\fX}{\mathfrak X}

\nc{\Tate}{\operatorname{Tate}}
\nc{\grTate}{\operatorname{grTate}}

\nc{\ch}{\check}
\nc{\fb}{\mathfrak b} \nc{\fu}{\mathfrak u} \nc{\st}{{st}}
\nc{\fU}{\mathfrak U}
\nc{\fZ}{\mathfrak Z}

\nc\fk{\mathfrak k} \nc\fp{\mathfrak p}

\newcommand{\opp}{{\mathrm{op}}}

\nc{\RP}{\mathbf{RP}} \nc{\rigid}{\text{rigid}}
\nc{\glob}{\text{glob}}

\nc{\cI}{\mathcal I}

\nc{\La}{\mathcal L}

\nc{\quot}{/\hspace{-.25em}/}

\nc\aff{\it{aff}}
\nc\BS{\mathbb S}

\newcommand{\Wdg}{\mathchoice{{\textstyle\bigwedge}}%
    {{\bigwedge}}%
    {{\textstyle\wedge}}%
    {{\scriptstyle\wedge}}}

\nc\git{/\hspace{-0.2em}/}
\nc{\fc}{\mathfrak c}
\nc\BC{\mathbb C}
\nc\BZ{\mathbb Z}

\nc{\Alg}{\on{Alg}}

\nc\stab{\text{\it st}}
\nc\Stab{\text{\it St}}

\nc\perf{\on{-perf}}

\nc\intHom{\mathcal{H}om}
\nc\shExt{\mathcal{E}xt}

\nc\gtil{\widetilde\fg}

\def\adjquot{/_{\hspace{-0.2em}ad}\hspace{0.1em}}

\nc\mon{\text{\it mon}}
\nc\bimon{\text{\it bimon}}
\nc\uG{{\underline{G}}}
\nc\uB{{\underline{B}}}
\nc\uN{{\underline{\cN}}}
\nc\uNtil{{\underline{\wt{\cN}}}}
\nc\ugtil{{\underline{\wt{\fg}}}}
\nc\uH{{\underline{H}}}
\nc\uX{{\ul{X}}}
\nc\uY{\ul{Y}}
\nc\upi{\ul{\pi}}
\nc{\uZ}{\ul{Z}}
\nc{\ucZ}{\ul{\cZ}}
\nc{\ucH}{\ul{\cH}}
\nc{\ucS}{\ul{\cS}}
\nc{\ahat}{{\wh{a}}}
\nc{\shat}{{\wh{s}}}
\nc{\DCoh}{{\rm DCoh}}

\newcommand{\Lie}{\operatorname{Lie}}

\newcommand{\cind}{\operatorname{c-Ind}}
\newcommand{\LL}{\cL}

\newcommand{\Tot}{\operatorname{Tot}}

\newcommand{\mf}{\mathfrak}

\newcommand{\todo}[1]{\textbf{\textcolor{red}{!!!TODO: #1}}}

\newcommand{\gwt}{\operatorname{wt}_{\mathbb{G}_m}}


\newcommand{\cat}{\mathbf}
\newcommand{\OO}{\cO}
\newcommand{\dmod}{\operatorname{-mod}}
\newcommand{\DMOD}{\textbf{-mod}}
\newcommand{\dperf}{\operatorname{-perf}}
\newcommand{\dcmod}{\operatorname{-cmod}}
\newcommand{\dtors}{\operatorname{-tors}}

\newcommand{\dcomod}{\operatorname{-comod}}
\newcommand{\dcoh}{\operatorname{-coh}}
\newcommand{\ICoh}{\operatorname{ICoh}}
\newcommand{\dR}{\operatorname{dR}}
\newcommand{\Hod}{\operatorname{Hod}}
\newcommand{\LLf}{\widehat{\cL}}
\newcommand{\Fin}{\operatorname{Fin}}
\newcommand{\GG}{\widetilde{G}}

\newcommand{\Fun}{\operatorname{Fun}}

\newcommand{\NN}{\widetilde{\mathcal{N}}}

\newcommand{\TT}[1]{\mathbb{T}_{#1}^{[{\ng}1]}}
\newcommand{\TTf}[1]{\mathbb{T}^{[{\ng}1]}#1}

\newcommand{\hTT}[1]{\widehat{\mathbb{T}}_{#1}^{[{\ng}1]}}
\newcommand{\hTTf}[1]{\widehat{\mathbb{T}}^{[{\ng}1]}#1}

\newcommand{\CT}[1]{\mathbb{T}^*_{#1}[2]}
\newcommand{\EE}[1]{\mathbb{E}_{#1}}
\newcommand{\EEd}[1]{\mathbb{E}_{#1}^*[1]}

\newcommand{\actson}{\circlearrowright}

\newcommand{\Haff}{\mathcal{H}}
\newcommand{\grH}{\bar{\mathcal{H}}}

\newcommand{\Hcat}{\mathbf{H}}

\newcommand{\ubQ}{\underline{\mathbb{Q}}}

\nc\Tr{{\mathbf{Tr}}} 
\newcommand{\tr}{\mathcal{T}r} 
\newcommand{\chern}{\operatorname{ch}}

\newcommand{\tilN}{\widetilde{\mathcal{N}}}

\newcommand{\un}{\operatorname{un}}

\newcommand{\TS}{\mathbf{T}}

\newcommand{\DLnil}{\mathcal{V}}

\newcommand{\bQl}{\overline{\mathbb{Q}}_\ell}

\newcommand{\dgMod}{\on{dgMod}}
\newcommand{\dgQCoh}{\on{dgQCoh}}

\newcommand{\Kz}{{\mathrm{K}}}

\mathchardef\md="2D
\mathchardef\mh="2D

\newcommand{\NH}{\mathbf{H}}

\newcommand{\bfI}{\mathbf{I}}
\newcommand{\bfG}{\mathbf{G}}

\newcommand{\LS}{\mathcal{LS}}
\newcommand{\APerf}{\operatorname{APerf}}

\newcommand{\bT}{\mathbb{T}}

\newcommand{\fil}{\mathrm{fil}}
\newcommand{\bB}{\mathbb{B}}
\newcommand{\bC}{\mathbb{C}}
\newcommand{\bD}{\mathbb{D}}
\newcommand{\bG}{\mathbb{G}}
\newcommand{\bI}{\mathbb{I}}
\newcommand{\bJ}{\mathbb{J}}
\newcommand{\bK}{\mathbb{K}}
\newcommand{\bL}{\mathbb{L}}
\newcommand{\bN}{\mathbb{N}}
\newcommand{\bO}{\mathbb{O}}
\newcommand{\bP}{\mathbb{P}}
\newcommand{\bQ}{\mathbb{Q}}
\newcommand{\bR}{\mathbb{R}}
\newcommand{\bU}{\mathbb{U}}
\newcommand{\bV}{\mathbb{V}}
\newcommand{\bW}{\mathbb{W}}
\newcommand{\bX}{\mathbb{X}}
\newcommand{\bY}{\mathbb{Y}}
\newcommand{\bZ}{\mathbb{Z}}

\newcommand{\wtD}{\wt{\cD}}
\newcommand{\wtO}{\wt{\Omega}^\bullet}
\newcommand{\wtOneg}{\wt{\Omega}^{-\bullet}}
\newcommand{\cdR}{d\cR^\bullet}

\mathchardef\md="2D
\mathchardef\mh="2D

\newcommand{\llb}{\llbracket}
\newcommand{\rrb}{\rrbracket}
\newcommand{\bbb}[1]{\llb #1 \rrb}

\newcommand{\rihgtarrow}{\rightarrow}

\newcommand{\hCoh}{\wh{\Coh}}
\newcommand{\tCoh}{\wt{\Coh}}

%
%


\newcommand{\tens}[1]{%
  \mathbin{\mathop{\otimes}\limits_{#1}}%
}

\newcommand{\utimes}[1]{%
  \mathbin{\mathop{\times}\limits_{#1}}%
}

\newcommand{\footremember}[2]{%
    \footnote{#2}
    \newcounter{#1}
    \setcounter{#1}{\value{footnote}}%
}
\newcommand{\footrecall}[1]{%
    \footnotemark[\value{#1}]%
} 

\title{Categorical cyclic homology and filtered $\cD$-modules on stacks:\\ Koszul duality}
\date{}
\author{Harrison Chen\footnote{Institute of Mathematics, Academia Sinica, Taipei, Taiwan}}

\maketitle


\begin{abstract}
Motivated by applications to the categorical and geometric local Langlands correspondences, we establish an equivalence between the category of filtered $\mathcal{D}$-modules on a smooth stack $X$ and the category of $S^1$-equivariant ind-coherent sheaves on its formal loop space $\widehat{\mathcal{L}} X$, exchanging compact $\mathcal{D}$-modules with coherent sheaves, and coherent $\mathcal{D}$-modules with continuous ind-coherent sheaves.  The equivalence yields a sheaf of categories over $\mathbb{A}^1/\mathbb{G}_m$ whose special fiber is a category of coherent sheaves on stacks appearing in categorical traces, and whose generic fiber is a category of equivariant constructible sheaves.
\end{abstract}

\tableofcontents

\section{Introduction}

In this paper we are interested in a Koszul equivalence of categories between modules for the de Rham complex (the ``$\Omega$-side'') and modules for the algebra of differential operators (the ``$\cD$-side'') on a smooth scheme $X$, or rather a certain deformation of this equivalence, which in a sense interpolates between algebra and topology.
\begin{center}
\begin{tabular}{|c|c|c|} \hline
generic fiber $\bG_m/\bG_m$ & interpolation over $\mathbb{A}^1/\bG_m$ & special fiber $\{0\}/\bG_m$ \\ 
\hspace{10ex}``topology''\hspace{10ex} & $\begin{tikzcd}[column sep=20ex] \arrow[r] & \arrow[l] \end{tikzcd}$ & ``algebra'' \\ 
\Xhline{3\arrayrulewidth}
Tate localization   & $B\bG_a \rtimes \bG_m$-coherent sheaves on $\bT_X[\ng1]$ & $\bG_m$-coherent sheaves on $\bT_X[\ng1]$ \\
de Rham complex $\Omega^{\bullet}_{X, d}$ & ``mixed'' de Rham complex $\dddot{\Omega}_X^{-\bullet}$ & differential forms $\Omega_X^{-\bullet}$ \\ \hline
$\cD$-modules & filtered $\cD$-modules & $\bG_m$-coherent sheaves on $\bT^*_X$ \\
sheaf of differential op. $\cD_X$ & Rees alg. $\wt{\cD}_X = \Rees(\cD_X, F^{\mathrm{ord}})$ & classical limit $\cO_{\mathbb{T}^*_X}$  \\
\Xhline{3\arrayrulewidth}
$\cD-\Omega$ duality  & ``filtered'' Koszul duality & linear Koszul duality \\
\cite{BD, rybakov, toly indcoh} & \cite{kapranov, coderived, loops and conns, foliations 1, foliations 2}  &  \cite{MR, MR IM, MR2, AO} \\ \hline
\end{tabular}
\end{center}
Our goal is to describe a version of this duality where $X$ is not a smooth scheme, but a smooth stack, and to explore new features that appear in that setting.  Our results are motivated by questions in the categorical local Langlands correspondence of Fargues--Scholze and Zhu \cite{FS, xinwen} and the relationship between geometric realizations of affine and graded Hecke algebras inside categories of coherent and constructible sheaves respectively \cite{ihes}; see Section \ref{local lang} for a discussion.  Our results are also a basic ingredient in developing a support theory for $\cD$-modules which serves as an automorphic counterpart to the singular support theory of coherent sheaves on the spectral side of Langlands duality \cite{AG}; see Section \ref{applications unsafe} for a discussion.

\subsubsection{} We now review the equivalences described in the table above, in the case where $X$ is a scheme.  In particular, we provide an interpretation of the categories appearing on the $\Omega$-side in terms of categorical traces, i.e. categorifications of Hochschild, negative, and periodic cyclic homologies \cite{BN:NT}.  The familiar reader may wish to skip directly to Section \ref{kd stack intro}.

\medskip

We first review constructions in the decategorified setting.  The usual Hochschild homology of the category $\QCoh(X)$ has a natural identification as global functions on the (derived) loop space $\cO(\cL X)$ \cite{TV S1, TV chern}.  When $X$ is a smooth scheme, the Hochschild--Kostant--Rosenberg isomorphism identifies this with the global sections of negatively graded differential forms $\Gamma(X, \Omega^{-\bullet}_X)$.  Hochschild homology has a cyclic enhancement in \emph{(negative) cyclic homology}, obtained by taking $S^1$-invariants for the natural $S^1$-action encoded by the Connes $B$-operator on Hochschild homology, which for a smooth scheme $X$ corresponds under the HKR isomorphism with the de Rham differential acting as a ``mixed'' degree $-1$ operator on $\Omega^{-\bullet}_X$.  The negative cyclic homology is naturally a module over $H^\bullet(BS^1; k)$-module; by inverting the degree 2 Chern class $u \in H^2(BS^1; k)$ one obtains the \emph{periodic cyclic homology}, which for a smooth scheme $X$ is isomorphic to a 2-periodicization of the de Rham cohomology of $X$.

\medskip

Likewise, the categorical Hochschild homology or trace of $\QCoh(X)$ may be identified with the category $\QCoh(\cL X)$ of quasi-coherent sheaves\footnote{We are interested in studying convolution categories over $X$ in general, which means studying the renormalization $\IndCoh(X)$ \cite{CD}.  To simplify the discussion we will ignore this in the introduction.} on the loop space \cite{BFN}, and for a smooth scheme $X$ there is an \emph{exponential map} $\exp: \bT_X[\ng1] = \Spec_X \Sym_X \Omega^{-\bullet}_X \rightarrow \cL X$ identifying the loop space with its linearization the \emph{odd tangent bundle}, i.e. the normal bundle to constant loops.  Categorical trace methods have been the subject of great interest lately \cite{GKRV, xinwen, FS, AGKRRV, BCHN} and have been used to apply results from geometric Langlands toward arithmetic Langlands; however, relatively little has been written about its cyclic enhancement, which is necessary for applications discussed in \cite{ihes} and Section \ref{local lang}.

\medskip

That is, it is of interest to study the \emph{categorical cyclic homology} $\QCoh(\cL X)^{S^1}$.  When $X$ is a smooth scheme, the HKR exponential map identifies the $S^1$-action on $\cL X$ with a $B\bG_a$-action on $\bT_X[\ng1]$, thus we may equivalently study the category $\QCoh(\bT_X[\ng1])^{B\bG_a}$ of $B\bG_a$-equivariant sheaves on the odd tangent bundle.  Explicitly, this category may be identified with modules for a certain \emph{mixed de Rham complex} $\dddot\Omega^{-\bullet}_X$, obtained by adjoining a noncommutative degree -1 operator $\delta$ such that $\delta^2 = 0$ and $[\delta, \omega] = d\omega$ \cite{loops and conns}.   The category of $S^1$-representations is equivalent to the category of $H^\bullet(BS^1; k)$-modules,\footnote{This is not strictly true.  One needs to renormalize, and we suppress this and related details in the introduction; see an extended discussion in Section \ref{sec bga}.} thus the categorical cyclic homology $\QCoh(\cL X)^{S^1}$ is linear over $H^\bullet(BS^1; k)$.  By inverting $u$ we obtain the \emph{categorical periodic cyclic homology}, which in the case of a smooth scheme is a 2-periodic version of Beilinson and Drinfeld's $\Omega$-modules \cite{BD}, i.e. a certain coderived category of modules for the de Rham complex $\Omega^{\bullet}_{X, d}$ \cite{coderived}.

\subsubsection{} That is, categorical cyclic homology lives over the ``2-shifted affine line'' $H^\bullet(BS^1; k) \simeq k[u]$ where the special fiber is the categorical Hochschild homology and the generic fiber is the categorical periodic cyclic homology.  To identify the 2-shifted line with the usual affine line $\bA^1$ we introduce the $\bG_m$-equivariance coming from linearization, and apply a Tate shearing commonly present in Koszul duality patterns, which transforms the weight-degree $(k, 0)$-line to the $(k, 2k)$-line \cite{BGS}, in particular identifying $u \in H^\bullet(BS^1; k)$ with a degree 0 variable $\hbar$, which we now interpret as a Rees parameter.  Now in degree 0, objects living over the graded affine line $\bA^1/\bG_m$ may be interpreted as filtered versions of objects living over the generic fiber.

\medskip

On the Koszul dual $\cD$-side, over this generic fiber the (unsheared) categorical periodic cyclic homology identifies with the category of modules for the sheaf of differential operators $\cD_X$.  It's extension to the graded affine line, i.e. the categorical negative cyclic homology, identifies with modules for its Rees algebra $\wt{\cD}_X$ for the order filtration, i.e. the category of filtered $\cD$-modules.  Then, its special fiber, i.e. the categorical Hochschild homology, is identified with the classical limit, sheaves on the cotangent space $\bT^*_X$.  In the case of schemes, these categories may all be described as certain categories of modules for sheaves of algebras, and the Koszul duality equivalences are given by tensoring with certain \emph{Koszul complexes} (or their duals) \cite{kapranov, coderived}.

\subsubsection{}\label{kd stack intro} In the case where $X$ is a stack, the story in the above table takes on a different and richer flavor.  The exponential map no longer identifies the entire loop space with the entire odd tangent bundle, but only near constant loops $X \subset \cL X$ and the zero section $X \subset \bT_X[\ng1]$, e.g. the exponential map between a Lie algebra and its group is not an isomorphism except near the identity.  In particular, linearizations are only able to see the part of categorical (periodic) cyclic and Hochschild homology near the constant loops, i.e. the \emph{formal loop space}, which is the subject of this paper.  In \cite{Ch}, the decategorified formal periodic cyclic homology was computed in terms of the de Rham complex of the stack $X$, and in \cite{ihes} we describe an equivariant localization pattern for studying the derived loop space away from constant loops in using methods from \cite{Ch}.

\medskip

Even near constant loops it is no longer true that the categories involved may be described as modules for sheaves of algebras on $X$.  For example, the category of $\cD$-modules on a global quotient stack $X/G$ is the category of strongly $G$-equivariant $\cD$-modules on $X$ where strong equivariance is imposed homotopically; the formulation of strong equivariance involves differentiating the $G$-equivariant structure on a sheaf.  This phenomenon is visible on the $\Omega$-side as well, as the $S^1$-action can no longer by described purely in terms of an operator on the sheaf $\Omega^{-\bullet}_X$.  For example, the $S^1$-action on $\cL(BG) = G/G$ is given by the universal automorphism which acts on an adjoint-equivariant sheaf on $G$, fiber-by-fiber at $g \in G$ by $g_*$-equivariance.  Thus for a stack, Koszul duality equivalence is necessarily more involved; we refer the reader to Section \ref{dmod BG exmp} for an example of this in the simplest case.

\medskip

Moreover, when $X$ is a stack, the category $\cD(X)$ has two natural small categories attached to it, introduced by Drinfeld and Gaitsgory \cite{QCA}.  The compact objects consist of the \emph{safe} $\cD$-modules $\cD_s(X)$, while one may also consider the larger category of \emph{coherent} $\cD$-modules $\cD_c(X)$.  This difference is fundamental, as the constant $\cD$-module is often coherent but not safe.  On the $\Omega$-side we have two different small categories  as well: the usual category $\Coh(\wh{\bT}_X[\ng1])$ of \emph{coherent sheaves on the formal completion} (i.e. coherent sheaves supported along the zero section), and the larger category of \emph{continuous ind-coherent sheaves} on the formal completion\footnote{We briefly comment on this nomenclature in the case of $R = k[x]$ with respect to the $(x)$-adic topology.  Roughly, there are two kinds of modules which have unique topologies such that the $R$-action is continuous with respect to the adic topology.  One is pro-finite (thus complete), i.e. the coinvariants or $*$-restriction to $x=0$ is coherent, while the other is ind-finite (thus locally discrete), i.e. the invariants or $!$-restriction is coherent.  See \cite[Rmk. 2.2.8]{CD}\cite[App. B]{EDH} for further discussion.} $\wh{\Coh}(\wh{\bT}_X[\ng1])$ introduced in \cite{CD}.  This Koszul duality gives rise to an invariant of $\cD$-modules on stacks measuring this difference, which we discuss in Section \ref{applications unsafe}.

\subsection{Statement of results}

We now state our results; fix $k$ to be a field of characteristic zero throughout.  Related Koszul duality results in the literature (see Section \ref{lit rev}) are currently formulated in a variety of different languages; we choose to formulate ours in the language of derived algebraic geometry, which we review in Section \ref{notation}.

\subsubsection{} We first discuss the $\Omega$-side, and how the stacks there arise naturally.  For $X$ a reasonable stack, the (derived) \emph{loop space} is the derived mapping stack
$$\cL X := \Map(S^1, X) = \Map(\asterisk \coprod_{\asterisk \; \asterisk} \asterisk, X) = X \utimes{X \times X} X$$
i.e. the derived self-intersection of the diagonal, or the derived inertia stack, which has $k$-points given by pairs $(x, \gamma)$ where $x \in X(k)$ and $\gamma \in \Aut(x)$ is an automorphism of the point $x$, modulo conjugation by $\Aut(x)$.  By construction, the derived loop space has an $S^1$-action.  We may linearize the derived loop space and its circle action as follows.  Consider the closed substack $X \subset \cL X$ of constant loops, and consider its normal bundle, i.e. the \emph{odd tangent bundle}
$$\bT_X[\ng1] := \Spec_X \Sym_X \Omega^1_X[1]$$
where $\Omega^1_X$ denotes the cotangent complex of $X$.  We define the \emph{formal odd tangent bundle} $\wh{\bT}_X[\ng1]$ to be the completion along the zero section, which is acted on by the \emph{affinization} $B\bG_a$ of the circle, i.e. we have identifications of dg algebras
$$\cO(S^1) = C^\bullet(S^1; k) \simeq \Ext^\bullet_{\Rep(\bG_a)}(k, k) = \cO(B\bG_a).$$
Furthermore, being linearizations, both $\wh\bT_X[\ng 1]$ and $B\bG_a$ admit compatible scaling $\bG_m$-actions, and we may extend the $B\bG_a$-action on $\wh\bT_X[\ng1]$ to a $B\bG_a \rtimes \bG_m$-action.

\subsubsection{}\label{intro ren}  We now introduce the category of \emph{continuous ind-coherent sheaves} on a formal completion $\wh{X}_Z$, which is a certain enlargement of the usual category of coherent sheaves $\Coh(\wh{X}_Z)$.   Let $\wh{X}_Z$ be a formal completion along a closed substack; we adopt the definitions from \cite{GR} so that the category $\Coh(\wh{X}_Z)$ is identified with the category $\Coh_Z(X)$ of coherent sheaves on $X$ supported along $Z$.  This category, in particular, excludes the structure sheaf $\cO_{\wh{X}_Z}$ and the dualizing sheaf $\omega_{\wh{X}_Z}$.  In \cite[\textsection 2]{CD} we introduced the enlargement $\hCoh(\wh{X}_Z) \subset \IndCoh(\wh{X}_Z)$, which in particular contains the dualizing sheaf\footnote{There is a Grothendieck dual version which contains the structure sheaf.} (see Definition \ref{def hcoh} for a precise definition).  We now enumerate its main properties from Proposition \ref{graded functoriality thm} with special attention to the case of the formal odd tangent bundle $\wh{\bT}_X[\ng1]$ for a smooth Artin stack $X$.
\begin{enumerate}
\item It contains the category of coherent sheaves $\Coh(\wh{X}_Z) \subset \hCoh(\wh{X}_Z)$, which is a proper subcategory if and only if $|Z| \not= |X|$.  In particular, if $X$ is a smooth scheme, then $\hCoh(\wh{\bT}_X[\ng1]) = \Coh(\wh{\bT}_X[\ng1])$.
\item It contains $!$-restrictions of coherent sheaves on the ambient stack $X$ to the formal neighborhood.  That is, for $\wh{\iota}: \wh{X}_Z \hookrightarrow X$ we have a functor $\wh{\iota}^!: \Coh(X) \rightarrow \hCoh(\wh{X}_Z)$.  In particular, $!$-restriction along the exponential map defines a functor $\exp^!: \Coh(\cL X) \rightarrow \hCoh(\wh{\bT}_X[\ng1])$.
\item Membership in the category may be tested by $!$-restriction along a quasi-smooth proper map which is surjective on geometric points.  That is, if $z: Z \hookrightarrow \wh{X}_Z$ is quasi-smooth, we have $z^! \cF \in \Coh(Z)$ if and only if $\cF \in \hCoh(\wh{X}_Z)$. In particular, $\hCoh(\wh{\bT}_X[\ng1])$ may be characterized as the full subcategory of $\IndCoh(\wh{\bT}_X[\ng1])$ consisting of objects whose $!$-restriction to $\wh{\bT}_U[\ng1] = \bT_U[\ng1]$ is coherent for an (or any) atlas $p: U \rightarrow X$.
\end{enumerate}
Part of our study involves Koszul dual descriptions of the category $\hCoh(\wh{X}_Z)$, for example in Theorem \ref{kd graded ren}.  Similar descriptions were established by Raskin in \cite[Thm. 3.2]{raskin}, which treats the case of the completion at a point inside a quasi-smooth scheme.

\subsubsection{} The main theorem is as follows, and appears in the body as Theorem \ref{kd kperf}.  The symbol $\shear$, borrowed 
 from \cite{BG}, denotes the Tate shearing, while $\unshear$ is the inverse unshearing; see Definition \ref{tate shear} for a discussion.  We require a certain \emph{compactly renormalized invariants operation} $(-)^{\omega B\bG_a}$ for groups which are not affine algebraic, and the \emph{graded Tate construction}  $(-)^{\grTate}$ inverting the degree 2 Chern class $u \in \cO(B^2\bG_a) \simeq H^\bullet(BS^1; k)$; see Definition \ref{ren G inv} and Definition \ref{def tate}.  We direct the reader to Section \ref{dmod BG exmp} for an explicitly worked example.
\begin{introthm}[Koszul duality for stacks]\label{intro thm}
Let $X$ be a smooth QCA stack.  We have compatible equivalences:
$$\begin{tikzcd}
\IndCoh(\wh{\bT}_X[\ng1])^{\grTate} \arrow[d, "\simeq"', "\unshear \circ \kappa"] & \IndCoh(\wh{\bT}_X[\ng1])^{\omega B\bG_a \rtimes \bG_m} \arrow[r] \arrow[d, "\simeq"', "\unshear \circ \kappa"] \arrow[l] & \IndCoh(\wh{\bT}_X[\ng1])^{\G_m} \arrow[d, "\unshear \circ \kappa", "\simeq"'] \\
\cD(X) & F\cD(X)  \arrow[r, "\gr"'] \arrow[l, "\un"] & \QCoh(\mathbb{T}_{X}^*)^{\bG_m} 
\end{tikzcd}$$
identifying the full subcategory of coherent (filtered) $\cD$-modules with the full subcategory of continuous ind-coherent sheaves
$$\begin{tikzcd}
\wh{\Coh}(\bT_X[\ng1])^{\grTate} \arrow[d, "\simeq"'] & \hCoh(\bT_X[\ng1])^{ B\bG_a \rtimes \bG_m} \arrow[r] \arrow[d, "\simeq"'] \arrow[l] & \hCoh(\bT_X[\ng1])^{\G_m} \arrow[d, "\simeq"'] \\
\cD_c(X) & F\cD_c(X)  \arrow[r, "\gr"'] \arrow[l, "\un"] & \Coh(\mathbb{T}_{X}^*)^{\bG_m} 
\end{tikzcd}$$
functorially with respect to schematic (i.e. representable by schemes) pullback, induction, restriction, and exchanging the notions of singular support. When $f$ is proper, the equivalences are compatible with pushforward.
\end{introthm}

We note that there is no serious obstruction to functoriality with respect to pushforward along non-proper maps $f$; however, such functors will not preserve the small category of coherent sheaves, thus one must take care to define the functor properly on renormalizations.  We expect this can be done using the approach of \cite{toly indcoh}.  We also expect that the equivalences should hold more generally in the non-smooth setting and can be extended as an equivalence of functors out of a category of correspondences \cite{GR, german}.

\subsubsection{Application to categorical local Langlands.}\label{local lang}

We discuss the application to the categorical local Langlands correspondence \cite{FS, xinwen} which motivated this paper, discussed at length in \cite{ihes}. In \cite{lusztig unipotent}, Lusztig established an equivalence between unipotent representations of a reductive group over a non-Archimedian local field and the module theory for certain \emph{affine Hecke algebras} with unequal parameters.  Lusztig's fruitful strategy \cite{cuspidal1, gradedAHA, cuspidal2} for classifying the simple modules of these affine Hecke algebras was to reduce them to certain completions of twisted \emph{graded Hecke algebras}, obtained by taking $\mf{m}$-adic completions for maximal ideals $\mf{m}$ of the center.  In turn, these completed graded Hecke algebras have, up to a Tate shearing, natural \emph{geometric} realizations as the Borel-Moore homology on certain ``fixed-point'' versions of the Steinberg stack, i.e. the self-$\Ext$s of ``fixed-point'' versions of the equivariant Springer sheaf living on the corresponding Vogan variety.  In other words, the module theory of graded Hecke algebras has a geometric categorical realization as categories of \emph{constructible sheaves}, or by Riemann-Hilbert, \emph{$\cD$-modules}; this geometric realization was studied in recent work of Solleveld \cite{sol1, sol2}.

\medskip

Recently, the affine Hecke algebra itself has also seen a parallel realization as the Hochschild homology of the Steinberg stack \cite{BCHN}, and by the formalism of iterated traces~\cite{BN:NT, CP, GKRV} Hochschild homology has a realization as the self-$\Ext$s of a certain \emph{coherent Springer sheaf} which lives in the \emph{categorical} Hochschild homology of the equivariant nilpotent cone, i.e. the category $\Coh(\cL(\wh{\cN}/G^\vee\!\times\bG_m))$ of coherent sheaves on the derived loop space of the equivariant nilpotent cone, i.e. a certain stack of \emph{unipotent Langlands parameters}.  In other words, the affine Hecke algebra can be directly realized geometrically via \emph{coherent sheaves on derived loop spaces}.

\medskip

It is thus a natural question whether one can realize the decategorified relationship between affine and graded Hecke algebras geometrically and categorically.  More generally, one can ask whether there is a general framework for passing between coherent sheaves on loop spaces and constructible sheaves.  The mechanism is precisely the Koszul duality we develop in this paper, though one must pass through an interpolation of the two, and a Tate shearing defined by a choice of graded lift: $\cD$-modules become filtered $\cD$-modules, and coherent sheaves on formal loop spaces are equipped with $B\bG_a \rtimes \bG_m$-equivariant structures.

\medskip

In a forthcoming work with Ben-Zvi, Helm and Nadler, we prove that when $G = GL_n$, the periodic localization of the category of $S^1$-equivariant sheaves on the stack of unipotent Langlands parameters is generated by the coherent Springer sheaf.  The proof reduces the question to the analogous claim in for constructible Springer sheaves one parameter at a time \cite{ihes}, with the Koszul duality result in this paper and the equivariant localization result in \cite{Ch} providing the essential technical tools for this reduction.

\subsubsection{An automorphic counterpart to singular support.}\label{applications unsafe}  

In \cite{AG}, Arinkin and Gaitsgory introduced the notion of \emph{singular support of coherent sheaves}, which measures the failure of a given coherent sheaf to be perfect, and is defined in terms of the ordinary support of the corresponding sheaf under a (different) Koszul duality.  In particular, the presence of singular support phenomena is related to the singularities of a given scheme $X$, i.e. is measured by the $(-1)$st cohomology of its cotangent complex $\cH^{-1}(\bL_X)$.  This notion was necessary to formulate the correct category on the spectral side of geometric Langlands in \emph{op. cit.} by imposition of a \emph{nilpotent singular support} condition.  This nilpotent singular support condition also appears in spectral affine Hecke categories \cite{roma hecke, CD}, and more generally a 2-categorical version is expected to play a role in the local geometric Langlands correspondence.  A natural question to ask is whether there is a corresponding support theory on the automorphic side, i.e. on categories of $\cD$-modules.

\medskip

When $X$ is a stack, there are two small categories one can associate to the category $\cD(X)$ of $\cD$-modules on $X$: the smaller category $\cD_s(X)$ of \emph{safe} $\cD$-modules (the compact objects) or the larger category $\cD_c(X)$ of \emph{coherent} $\cD$-modules \cite{QCA}.  In a ongoing project with Dhillon, we study analogous intermediate categories between $\cD_s(X)$ and $\cD_c(X)$ via a notion of support which we call \emph{unsafe support}, i.e. an invariant measuring the failure of a coherent $\cD$-module to be safe, which is defined via the Koszul duality in this paper.  Where singular support of coherent sheaves arises in the presence of non-smoothness, unsafe support arises in the presence of stackiness in semisimple directions, i.e. up to nilpotent directions by the $(-1)$st cohomology of its tangent complex $\cH^{-1}(\bL_X^\vee)$.  This notion of unsafe support on (non-filtered) $\cD$-modules cannot measure in unipotent directions, and the smallest allowable unsafe support is given by the nilpotent cone (see Example \ref{BGa exmp}); in particular, safety is defined by a \emph{nilpotent unsafe support} condition, matching the nilpotent singular support condition on the spectral side.

\subsubsection{Literature review.}\label{lit rev}

The earliest appearance (to our knowledge) of a ``filtered'' Koszul duality between the mixed de Rham complex and Rees algebra of differential operators is a paper of Kapranov \cite{kapranov}, which studies the case of smooth schemes.  One difference in their approach is that the notion of quasi-isomorphism on the $\Omega$-side is by passing to analytifications, as is typical when defining the de Rham functor for $\cD$-modules. The same idea was exploited by Beilinson and Drinfeld \cite{BD} to establish smooth descent for (non-filtered) $\cD$-modules, which they use to develop a theory of $\cD$-modules on stacks, using an explicit description of the dg totalization of dg derived categories with t-exact transition functors.  To avoid the analytifications employed in \cite{kapranov}, the authors simply transport the desired notion of quasi-isomorphism from the $\cD$-side to the $\Omega$-side, which were later characterized more intrinsically by Positselski in \cite{coderived} in terms of coderived categories.  Furthermore, both versions of Koszul duality are reinterpreted in this framework in \emph{op. cit.}, and likewise functoriality in \cite{rybakov}.

\medskip

On the other hand, the Koszul duality at the special fiber is a special case of the \emph{linear Koszul duality} of Mirkovi\'{c} and Riche \cite{MR, MR IM, MR2}.  In this series the authors prove functoriality results for linear Koszul duality for quotient stacks, often passing through Grothendieck duality to do so.  In the language of \cite{coderived, rnkd} this means taking the contraderived category rather than the coderived category.  We will prefer our equivalences to be covariant rather than contravariant, since we sometimes need to ind-complete and work with large categories.

\medskip

In the spiritual predecessor to this paper, Ben-Zvi and Nadler \cite{loops and conns} give an algebreo-geometric interpretation of the $\Omega$-side of Koszul duality, namely as a category of $S^1$-equivariant (or more precisely, $B\bG_a \rtimes \bG_m$-equivariant) quasi-coherent sheaves on derived loop spaces of smooth schemes, and prove a filtered Koszul duality in that setting; the case of a quotient stack in the non-$B\bG_a$-equivariant setting was also studied more recently in \cite{AO}. Later, Preygel established the Koszul duality equivalence over the generic point in the case of possibly singular schemes \cite{toly indcoh}, while in recent work To\"{e}n and Vezzosi developed a general framework for $B\bG_a$-equivariant sheaves, Tate constructions and Koszul duality equivalences on singular schemes \cite{foliations 1, foliations 2}.  A relationship between loop spaces and the Hodge stack also appeared in \cite{tasos hodge}.

\medskip

Finally, we note that the recent book by Positselski \cite{rnkd} provides a robust framework for understanding many ideas in this work in the setting of schemes. For a discussion of translating the language in \emph{op. cit.} to the language of derived algebraic geometry in \cite{GR}, see Section H.3 of \cite{AG}.  We also note the close connection between $S^1$-equivariant dg categories and curved dg categories discussed in \cite{MF}.

\subsection{Background and set-up}\label{notation}\label{sec infcat}

We review the language of $\infty$-categories and derived algebraic geometry (as developed in \cite{topos, HA, GR}) which we use to formulate and prove the results in this paper.  Throughout the paper we let $k$ denote an algebraically closed field of characteristic 0.  Our grading conventions will be cohomological, i.e. $[1]$ indicates a cohomological shift in the negative direction.

\subsubsection{\texorpdfstring{$\infty$}{Infinity}-categories.} We will sometimes require distinguishing between large and small categories in the usual way, which we recall briefly.  Following \cite[\textsection 1.2.15]{topos}, throughout the discussion, we fix a Grothendieck universe; we say a set is \emph{small} if it is an element of the Grothendieck universe, and otherwise we say it is \emph{large}.  A category is \emph{small} if its set of objects is small and (total) set of morphisms is small, and \emph{large} if otherwise (i.e. it is small in some other, larger, Grothendieck universe).  Throughout this paper we fix a Grothendieck universe from which we take small sets.

\medskip

By \emph{$\infty$-category}, we will mean a quasi-category (i.e. weak Kan complex) unless otherwise noted (though this choice will not be particularly important for us).  We denote by $\cat{Cat}_{\infty}$ the $\infty$-category of small $\infty$-categories, and $\wh{\cat{Cat}}_\infty$ the category of not necessarily small $\infty$-categories~\cite[Def. 3.0.0.1]{topos}.  
We will also consider monoidal $\infty$-categories, and refer the reader to the appendix of \cite{CD} for a discussion.

\subsubsection{} We are interested in $k$-linear (stable) $\infty$-categories.  We refer the reader to \cite{HA} for the notion of a stable $\infty$-category, but the key feature is that the homotopy (1-)category of stable $\infty$-categories are canonically triangulated, and that the vast amount of structure in $\infty$-categories means stability is not an additional structure but a property of an $\infty$-category.  We denote by $\cat{St} \subset \cat{Cat}_\infty$ the subcategory of stable $\infty$-categories, with morphisms exact functors (i.e. functors which commute with finite limits and colimits).

\medskip

As a starting point, we define the category of chain complexes over a field $k$ as a stable $\infty$-category via the dg nerve of the dg category of $k$-chain complexes $\cat{Vect}_k := N_{dg}(\Ch_{dg}(k)).$
Every quasi-isomorphism of $k$-complexes is a homotopy equivalence, so there is no need to pass to the derived category.  This category is symmetric monoidal, since the underlying dg category is a strict symmetric monoidal dg category \cite[\textsection A.1.9]{CD}.  It is equipped with the usual t-structure, for which the symmetric monoidal structure maps are t-exact.

\medskip

A $k$-linear $\infty$-category is most easily defined to be a $\cat{Vect}_k$-module category in the $\infty$-categorical sense \cite[\textsection 3.3]{HA}.  Note that $\cat{Vect}_k$ is rigid, thus by typical arguments~\cite[Lem. 2.11]{BCHN} we may view a $k$-linear $\infty$-category as an $\infty$-category whose $\Hom$-spaces are enriched in $\cat{Vect}_k$.  Such categories are automatically stable \cite[Rem. 6.5]{dag7}.  We denote the category of $k$-linear $\infty$-categories by $\cat{St}_k$.  We denote by $\cat{Pr}^L_k \subset \wh{\cat{Cat}}_\infty$ the category of (large) presentable $\Perf(k)$-module (in particular, stable and possessing small colimits) $\infty$-categories, with morphisms colimit-preserving functors (equivalently by the $\infty$-adjoint functor theorem, functors which are left adjoints; see Section 5.5 of \cite{topos}).

\medskip

There is an ind-completion functor $\Ind: \cat{St}_k \rightarrow \cat{Pr}^L_k$ \cite[Def. 5.3.5.1]{topos}\cite[Prop. 1.1.3.6]{HA}.  
Note that the functor  $\Ind: \cat{St}_k \rightarrow \cat{Pr}^L_k$ commutes with small colimits \cite[Prop. 5.5.7.6, Not. 5.5.7.7]{topos}, but it does \emph{not} commute with limits.  In particular, if $G$ is a topological group action on $\cat{C} \in \cat{Cat}_\infty$, then the functor $\Ind(\cat{C}^G) \rightarrow \Ind(\cat{C})^G$ is not an equivalence in general; see the introduction to \cite{toly indcoh} and Definition \ref{ren G inv} for a discussion.

\subsubsection{Limits of $\infty$-categories and subcategories.}\label{truncated} We now recall standard results regarding limits of $\infty$-categories.  The following is \cite[Prop. 1.1.4.4, 5.5.3.13]{topos}, the point being that we do not have to be fussy about the context in which we take limits of $\infty$-categories.
\begin{prop}\label{limit commute}
The category $\cat{Pr}^L_k$ has all small limits, and the inclusion $\cat{Pr}^L_k \hookrightarrow \wh{\cat{Cat}}_\infty$ commutes with limits.  The category $\cat{St}_k$ has all small limits, and the inclusion $\cat{St}_k \hookrightarrow \cat{Cat}_\infty$ commutes with limits.  
\end{prop}

Let us recall the notion of truncatedness \cite[Def. 5.5.6.1]{topos} in $\infty$-categories.  Namely, an object $X$ of an $\infty$-category $\cat{C}$ is \emph{$k$-truncated} if $\pi_i(\Map_{\cat{C}}(-, X)) = 0$ for $i > k$.  A morphism is \emph{$k$-truncated} if all of its homotopy fibers are $k$-truncated.  We say an $\infty$-category is $k$-truncated if it is $k$-truncated viewed as an object in the $\cat{Cat}_\infty$.   The category of $k$-truncated objects is stable under limits which exist \cite[Prop. 5.5.6.5]{topos}, and the inclusion of the full subcategory of $k$-truncated objects in presentable $\infty$-categories commutes with limits \cite[Prop. 5.5.6.18]{topos}.

\medskip

We will often use the fact that limits of $\infty$-categories preserve fully faithful functors.  This appears to be a well-known fact to experts, though we could not find an argument in the literature.  
\begin{prop}
Let $\cat{C}$ be an $\infty$-category with limits, and let $K$ be an indexing simplicial set, and $X_\bullet \rightarrow Y_\bullet$ be a map between two $K$-diagrams in $\cat{C}$ which is pointwise $k$-truncated.  Then, $\lim_K X_\bullet \rightarrow \lim_K Y_\bullet$ is $k$-truncated.
\end{prop}
\begin{proof}
By \cite[Lem. 5.5.6.15]{topos} it suffices to show that 
$$\lim_K X_\bullet \rightarrow \lim_K X_\bullet \times_{\lim_K Y_\bullet} \lim_K X_\bullet \simeq \lim_K (X_\bullet \times_{Y_\bullet} X_\bullet)$$
is $(k-1)$-truncated.  This would follow by the statement of the proposition for $k-1$, thus we may induct.  The case where $k=-2$, i.e. where the arrows are equivalences, is evident.
\end{proof}
\begin{cor}\label{ff limit}
Let $K$ be an indexing simplicial set, and $\cat{C}_\bullet \hookrightarrow \cat{D}_\bullet$ a map of two $K$-diagrams of $\infty$-categories which is pointwise fully faithful.  Then, the induced functor on limits $F: \lim_K \cat{C}_\bullet \rightarrow \lim_K \cat{D}_\bullet$ is fully faithful with essential image the full subcategory of objects  which evaluate to the subcategories $\cat{C}_\kappa$ for every 0-simplex $\kappa \in K$.
\end{cor}

\subsubsection{Derived categories.}  As we will see in Section \ref{dag sec}, there is a way to define the stable $\infty$-category of quasi-coherent sheaves on any stack $X$ intrinsically in the $\infty$-categorical setting, entirely without reference to abelian categories.  However, it will sometimes be useful to produce, as a stable $\infty$-category, the derived category of an abelian or dg category.  This is done by model structures.
\begin{enumerate}
\item If $\cat{C}$ is a Grothendieck abelian category, one may consider its injective model structure and define $D(\cat{C}) := N_{dg}(\Ch(\cat{C})^f)$ to be the dg nerve of the full subcategory of fibrant (automatically cofibrant) objects.  By \cite[Prop 1.3.5.13]{HA} it is a localization of $N_{dg}(\Ch(\cat{C}))$ with respect to quasi-isomorphisms, and by \cite[Prop 1.3.5.21]{HA} this category is presentable.  See \cite[\textsection 1.3.5]{HA} for a discussion.
\item If $A$ is a dg algebra, then by \cite[Thm. 9.10]{BMR}, the dg category $\Ch(A)$ of unbounded complexes of $A$-modules has a projective model structure where the cofibrant objects are retracts of semi-free complexes.  We define $D(A) := N_{dg}(\Ch(A)^{c})$ to be the dg nerve of the full subcategory of cofibrant (automatically fibrant) objects, i.e. of semi-free complexes.
\end{enumerate}

The construction (1) is useful in treating classical non-affine situations, since categories of sheaves are often Grothendieck abelian.  The construction (2) is useful in treating derived affine situations.  We are not aware of a result ``joining'' the two.  One can argue using the universal property of localizations~\cite[Prop. 5.5.4.20]{HTT} that the two constructions give canonically equivalent categories in their ``meet''.

\subsubsection{Derived algebraic geometry.}\label{dag sec}

We take the approach to derived algebraic geometry developed in \cite{GR}.  For any commutative $\bQ$-algebra $R$, viewed as an $E_\infty$-algebra, one can define the category of $E_\infty$ $R$-algebras intrinsically to $\infty$-categories \cite[Def. 7.1.0.1, Rmk. 7.1.0.3]{HA}.  Taking $R=k$ to be our field of characteristic 0, one can define
$$\cat{DRng}_k := \mathrm{Alg}_{E_\infty}(\cat{Vect}_k^{\leq 0}).$$
By \cite[Prop. 7.1.4.11]{HA}, there is a model structure on the category of commutative connective dg algebras and an equivalence
$$N_{dg}(\mathrm{CAlg}^{\leq 0}_{dg}(k)^c)[W^{-1}] \simeq \mathrm{Alg}_{E_\infty}(\cat{Vect}_k)$$
between the dg nerve of cofibrant commutative dg $k$-algebras localized with respect to weak equivalences, and the abstractly defined category $\cat{DRng}_k$ as defined above.  In particular, by the discussion preceding Proposition 7.1.4.20, semi-free dg algebras are cofibrant, thus for any dg algebra $A$ we may (up to contractible homotopy) define an object of $\cat{DRng}$ by choosing a semi-free resolution of $A$ and passing it through this equivalence.

\medskip

The category of \emph{affine derived schemes} is simply defined $\cat{Aff}_k := \cat{DRng}_k^{\opp}$ and an \emph{affine derived scheme} is a formal symbol $\Spec A$ for $A$ a derived ring.  A \emph{prestack} is an $\infty$-functor $X: \cat{Aff}_k^{\opp} \rightarrow \cat{Grpd}_\infty$ from the $\infty$-category of dg $k$-algebras to the $\infty$-category of $\infty$-groupoids (i.e. spaces).  We refer the reader to \emph{op. cit.} for definitions of these notions, and other standard ones such as Artin $n$-stacks.  We also use the notion of a \emph{QCA stack} from \cite{QCA}, which have the important property that $\IndCoh(X) = \Ind(\Coh(X))$.

\subsubsection{} The approach of \cite{GR} uses the following tautological definition of the derived category of quasicoherent of sheaves for any prestack.  Namely, let $S = \Spec(A)$ be an affine derived scheme, i.e. $A$ is an algebra object in $\cat{Vect}^{\leq 0}_k$.  The symmetric monoidal structure on $\cat{Vect}_k$ makes $\cat{Vect}_k$ a module category for the symmetric monoidal category $\cat{Vect}_k^{\leq 0}$, and we define
$$\QCoh(S) := \Mod(\cat{Vect}_k, A)$$
to be the category of $A$-module objects in $\cat{Vect}_k$ \cite[\textsection 4]{HA}\cite[\textsection A.2]{CD}.\footnote{One can argue that $\QCoh(S) \simeq D(A)$ since both categories are generated by the free object $A$ under colimits, shifts and retracts.}  We let $\Perf(S) \subset \QCoh(S)$ denote the smallest full idempotent-complete  stable subcategory containing the free object $A$ (see \cite[I.3 3.6]{GR} for equivalent characterizations).  Now, for any prestack $X$ we define $\Perf(X)$ and $\QCoh(X)$ via right Kan extension
$$\left(  \Perf(X) = \lim_{S = \Spec(A) \rightarrow X} \Perf(S) \right) \subset \left( \QCoh(X) := \lim_{S = \Spec(A) \rightarrow X} \QCoh(S)\right).$$
When $X$ is an Artin stack with atlas $p: U \rightarrow X$, we may compute $\QCoh(X)$ via the $\infty$-totalization of the usual Cech diagram.  We say $\cE \in \QCoh(X)$ is \emph{perfect in degrees} $[a, b]$ if for every $x \in X(S)$, we have $H^i(x^*\cE) = 0$ outside of the range $[a, b]$.  

\subsubsection{} The category $\QCoh(X)$ is equipped with a canonical t-struture, and one can identify the above abstract definition of $\QCoh(X)$ with a more classical presentations of the above categories via this t-structure.  See \cite{BFN} for the notion of a perfect stack.
\begin{prop}\label{classical is derived qcoh}
Suppose that $X$ is a classical perfect stack (e.g. the quotient of a quasi-projective scheme by an affine algebraic group).  Then, $\QCoh(X)^\heartsuit$ is the usual 1-category of quasicoherent sheaves on the stack $X$, and the canonical t-exact map $D(\QCoh(X))^\heartsuit \rightarrow \QCoh(X)$ is an equivalence.
\end{prop}
\begin{proof}
The category $\QCoh(X)$ is both left and right t-complete by \cite[Prop. 3.5.14]{DAG8} then, by \cite[Rem. 1.3.5.23]{HA}, we have a t-exact functor $D(\QCoh(X)^\heartsuit) \rightarrow \QCoh(X)$.  It is an equivalence follows by \cite[I.3 Lem. 2.4.5]{GR} and \cite[Rem. 1.2.10]{QCA}, i.e. since $X$ is a classical algebraic stack with affine diagonal, the functor $D(\QCoh(X))^+ \rightarrow \QCoh(X)^+$ is an equivalence, and if $X$ is perfect then the functor is an equivalence.  Furthermore, we may identify $\QCoh(X)^\heartsuit$ with the usual 1-category of quasicoherent sheaves on $X$: choose a smooth cover $p: U \rightarrow X$; since $p^*$ is t-exact, it preserves the heart, and by Corollary \ref{ff limit} we have that $\QCoh(X)^\heartsuit$ is the $\infty$-totalization of the Cech diagram with terms $\QCoh(U \times_X \cdots \times_X U)^{\heartsuit}$.  Since these are discrete $\infty$-categories, i.e. 1-categories, the limit may be taken in 1-categories (see the discussion in Section \ref{truncated}).
\end{proof}

\subsection{Acknowledgements}

I would like to thank David Ben-Zvi for extensive discussions on this subject in preparation of the paper \cite{BCHN}, especially surrounding shearing and (de-)equivariantization.  I would like to thank Gurbir Dhillon for many conversations surrounding equivariant $\cD$-modules and continuous ind-coherent sheaves in preparation of the work \cite{CD} and for numerous helpful comments on a draft of this paper.  Finally, I would like to thank Bertrand To\"{e}n for discussion surrounding filtered $\cD$-modules, Aron Heleodoro for stimulating discussions on this paper, David Nadler for discussions on circle actions, Sam Raskin for suggestions regarding renormalizations of Koszul duality from \cite{raskin}, and Anatoly Preygel for discussions on his work \cite{toly indcoh}.

\section{Renormalized linear Koszul duality and odd tangent bundles}\label{graded kd sec}

In this section we introduce a certain renormalization of the category of ind-coherent sheaves which appears in our version of Koszul duality.  We first review and recast the linear Koszul duality of Mirkovi\'c and Riche \cite{MR, MR IM, MR2} in the language of derived algebraic geometry in Section \ref{linear koszul}, treating both the graded and non-graded settings by $\bG_m$-equivariantization and de-equivariantization.   In Section \ref{ren linear kd} we enlarge the categories of compact objects and extended linear Koszul duality to this enlargement in Theorem \ref{kd graded ren}. In Section \ref{ren loops} we apply the discussion to the case of odd tangent bundles.

\subsection{Linear Koszul duality}\label{linear koszul}

In this section we review the linear Koszul duality of Mirkovi\'{c} and Riche \cite{MR, MR IM, MR2}.  We will reformulate their results in our notation and on the foundations of derived algebraic geometry and stable $\infty$-categories.  In addition, we also consider non-$\bG_m$-equivariant versions of their statements.

\subsubsection{}

Let $X$ be a prestack.  We define a \emph{perfect derived vector bundle}  $\bE_X$ on $X$ to be a pair $(X, \cE)$ where $\mathcal{E}$ is a perfect complex on $X$ of \emph{linear sections}.  We call its dual $\cE^\vee$ the \emph{sheaf of linear functionals}, and the algebra object of $\QCoh(X)$ denoted $\cO_{\bE_X} := \Sym_X \cE^\vee$ is its  \emph{structure sheaf}.  We say that $\bE_X$ is \emph{perfect in degrees $[a, b]$} if its sheaf of linear sections $\cE$ is perfect in degrees $[a, b]$.  We say a derived vector bundle is \emph{connective} if it is perfect in degrees $[0, \infty)$, and \emph{coconnective} if it is perfect in degrees $(-\infty, 0]$.  When $\bE_X$ is connective, we may view the a priori formal symbol $\bE_X$ as a prestack which is affine over $X$
$$\bE_X = \Spec_X \Sym_X \mathcal{E}^\vee.$$

If $\pi: X \rightarrow Y$ is a map of prestacks and $\bE_Y = (Y, \cE)$ is a derived vector bundle, then its base change $\bE_X := \bE_Y \times_Y X = (X, \pi^*\cE)$ to $X$ is a derived vector bundle.  We also denote by $X$ the trivial derived vector bundle corresponding to $(X, 0)$.  A derived vector bundle $\bE_X = (X, \cE)$ may be cohomologically shifted $\bE_X[n] = (X, \cE[n])$ with structure sheaf $\cO_{\bE_X[n]} = \Sym_X(\cE^\vee[-n])$.

\medskip

We may introduce a weight grading, or $\bG_m$-action, by assigning $\cE$ any integral weight.  Most commonly, we will give it weight $\pm 1$.  When the weight is $+1$, we say the the action is \emph{contracting}, and when the weight is $-1$, we say it is \emph{expanding}.  Furthermore, when $\bE_X$ is connective, we may form the affine over $X \times B\bG_m$ prestack
$$\bE_X/\bG_m = (\Spec_X \Sym_X \mathcal{E}^\vee)/\bG_m.$$
With the standard contracting action, shifting $\bE_X$ to $\bE_X[n]$ corresponds to a weight-degree shearing on the structure sheaf which shifts the weight $k$ in the negative (or left) direction by $nk$.

\medskip

\begin{warning}\label{prestack warning}
When $\bE_X$ is connective, we will treat it as a derived stack.  In the non-connective case, it is possible to do this as well by defining its functor of points; in the coconnective case, this gives rise to To\"{e}n's affine stacks~\cite{toen affine} or Lurie's coaffine stacks~\cite{lurie coaffine}.  However, this is \emph{not what we do}, and the category $\QCoh(\bE_X)$ that we will define in Definition \ref{cat def} will not be compatible with this interpretation, and other categories we define can exhibit some strange behavior (see Example \ref{non graded kd exmp}).  That is, in the non-connective case, we view $\bE_X$ as a formal symbol.  
\end{warning}

\begin{exmp}
Take $X = \Spec k$ and $\bE_X$ the derived vector bundle attached to $k[2]$, i.e. so that $\cO_{\bE_X} = k[u]$ where $|u| = 2$.  We will soon define $\QCoh(\bE_X)$ to be the derived category of $k[u]$-modules.  However, the corresponding coaffine stack is $B^2\bG_a$, and by Proposition \ref{bga} the category $\QCoh(B^2\bG_a)$ is equivalent to the full subcategory of $k[u]$-modules consisting of locally $u$-nilpotent modules.  One can verify that the same discrepancy arises in the graded situation.
\end{exmp}

\subsubsection{} We make a brief digression on shearing.  The category $\QCoh(B\bG_m)$, i.e. internally $\bZ$-graded chain complexes, has a family of monoidal automorphisms indexed by $n \in \bZ$:
$$\llb n\rrb: \QCoh(B\bG_m) \rightarrow \QCoh(B\bG_m), \;\;\;\;\;\;\;\; M \mapsto M{\llb n \rrb} := \bigoplus_{n \in \bZ} M_n[n].$$
Thus, for any $\QCoh(B\bG_m)$-module category $\cat{C}$, we may act on $\cat{C}$ via twist by this automorphism.
\begin{defn}\label{def shear}
Let $\cat{C}$ be a $\QCoh(B\bG_m)$-module category.  We define the \emph{$n$-shearing} $\cat{C}^{\llb n\rrb}$ to be the category with the same underlying category $\cat{C}$, but the $\QCoh(B\bG_m)$-action is twisted by the automorphism $\bbb{-n}$.  We note that the identity functor on the underlying category $\cat{C} = \cat{C}^{\bbb n}$ is not $\QCoh(B\bG_m)$-equivariant.
\end{defn}

\medskip

We heuristically explain the sign discrepancy as follows. Let $V \in \QCoh(B\bG_m)^{\bbb{n}}$; by definition, it is isomorphic to the action of $V^{\bbb{n}}$ on $k \in \QCoh(B\bG_m)^{\bbb{n}}$.  In particular, the image of $V$ in the de-equivariantization is exactly $V\bbb{n}$.  Alternatively, thinking of a $\QCoh(B\bG_m)$-module category as a category enriched in $\QCoh(B\bG_m)$ \cite[\textsection 3.3.3]{ho li}, by our definition above we have the formula
$$\uHom_{\cat{C}^{\llb n\rrb}}(X, Y) = \uHom_{\cat{C}}(X, Y)^{\llb n\rrb}.$$
If $A$ is an algebra object in $\QCoh(B\bG_m)$, then by our convention
$$\Mod_{\QCoh(B\bG_m)}(A)^{\bbb{n}} \simeq \Mod_{\QCoh(B\bG_m)}(A{\bbb{n}}).$$
Note that if $n$ is even, then shearing preserves (super-)commutativity, but otherwise it may not.  

\medskip

By the (de-)equivariantization correspondence (see Section \ref{sec bga}), we may use this to define a shearing on $\QCoh(\bG_m)$-module categories or $\QCoh(\bG_m)$-comodule categories.  Unlike above this operation changes the underlying category.
\begin{defn}\label{deeq shear}
Let $\cat{C}$ be a $\QCoh(\bG_m)$-module category.  We define the \emph{de-equivariantized shearing}
$$\cat{C}^{\bbb{n}} := (\cat{C}^{\QCoh(\bG_m)})^{\bbb{n}} \otimes_{\QCoh(B\bG_m)} \cat{Vect}_k$$
i.e. by equivariantizing, changing the $\QCoh(B\bG_m)$-action, then de-equivariantizing.
\end{defn}

There is a particular shearing that features prominently in Koszul duality, and we introduce special notation for it from \cite{BG}.
\begin{defn}\label{tate shear}
We use the notation $\shear = \bbb{-2}$ to denote the \emph{Tate shearing}, which puts complexes concentrated in degree 0 on the ``Tate line'' in weight-degree $(k, 2k)$, and $\unshear = \bbb{2}$ for the \emph{Tate unshearing}, which reverses it.
\end{defn}

\subsubsection{} We are interested in defining certain categories of sheaves attached to the formal symbol $\bE_X$.  When $\bE_X$ is connective, we define the categories by interpreting $\bE_X$ and $\bE_X/\bG_m$ as derived stacks.  When $\bE_X$ is not connective, we define the category via shearing.

\begin{defn}\label{cat def}
Let $X$ be any prestack and $\bE_X$ a derived vector bundle with $\bG_m$ acting by nonzero weight.  Choose an $n$ such that $\bE_X[n]$ is connective and consider the $\bG_m$-action on $\bE_X[n]$.  We define the categories via shearing (see Definitions \ref{def shear} and \ref{deeq shear}):
$$\QCoh(\bE_X/\bG_m) := \QCoh(\bE_X[n]/\bG_m)^{\bbb{-n}}, \;\;\;\;\;\;\;\;\;\; \IndCoh_X(\bE_X/\bG_m) := \IndCoh_X(\bE_X[n]/\bG_m)^{\bbb{-n}},$$
$$\QCoh(\bE_X) := \QCoh(\bE_X[n])^{\bbb{-n}}, \;\;\;\;\;\;\;\;\;\; \IndCoh_X(\bE_X) := \IndCoh_X(\bE_X[n])^{\bbb{-n}}.$$  
We often consider the \emph{formal derived vector bundle} $\wh{\bE}_X$ (completed along the zero section), and define
$$\IndCoh(\wh{\bE}_X/\bG_m) := \IndCoh_X(\bE_X/\bG_m), \;\;\;\;\;\;\;\;\;\; \IndCoh(\wh{\bE}_X) := \IndCoh_X(\bE_X).$$
\end{defn}

\begin{rmk}
Morally, for the graded categories, the idea is that the category of graded $\cO_{\bE_X}$-modules on $X$ is equivalent via an explicit and analogous weight shearing operation to the category of graded $\cO_{\bE_X[n]}$-modules on $X$.  This identification is canonical.  In the above we simply define the two categories to be the same, and detect the difference via the $\QCoh(B\bG_m)$-module structure.  On the other hand, the sheaves $\cO_{\bE_X}$ and $\cO_{\bE_X[n]}$ are distinguishable in the de-equivariantization without regard to module structures.
\end{rmk}

\medskip

We establish the following basic property, which tells us that even when $\bE_X$ is not connective, the category $\QCoh(\bE_X)$, while differing from the interpretation as a derived stack, is just a category of modules for an algebra object in $\QCoh(X)$.  On the other hand, we will see, such as in Example \ref{non graded kd exmp}, that its various renormalizations such as $\IndCoh_X(\bE_X)$ can behave in perhaps unintuitive (but interesting) ways.
\begin{prop}\label{shear cat prop}
The category $\QCoh(\bE_X)$ may be canonically identified with the category of $\cO_{\bE_X}$-modules in $\QCoh(X)$.  Likewise, $\QCoh(\bE_X/\bG_m)$ is the category of weight-graded $\cO_{\bE_X}$-modules in $\QCoh(X)$.  In particular, the above definitions are independent of the choice of $n$.
\end{prop}
\begin{proof}
We apply Barr-Beck to the adjoint pair $(p^*, p_*)$ where $p: \bE_X \rightarrow X$.  When $\bE_X$ is connective, the monad $p_*p^*(-) \simeq \cO_{\bE_X} \otimes -$, and the claim follows by the usual arguments.  When $\bE_X$ is not connective, choose $[n]$ such that it is connective.  Consider the adjoint pair $(\wt{p}^*, \wt{p}_*)$ where $\wt{p}: \bE_X[n]/\bG_m \rightarrow X \times B\bG_m$.  The monad $\wt{p}^*\wt{p}_*$ is given by tensoring with $\cO_{\bE_X[n]}$, where we view it with its internal weight grading, inside the category $\QCoh(X \times B\bG_m)^{\bbb{-n}}$.  On de-equivariantization $\cO_{\bE_X[n]}$ becomes $\cO_{\bE_X[n]}{\bbb{-n}} \simeq \cO_{\bE_X}$.
\end{proof}

\subsubsection{}  We are also interested in studying certain small subcategories.
\begin{defn}
We first consider the non-graded setting.
\begin{enumerate}
\item The category $\Perf(\bE_X) \subset \QCoh(\bE_X)$ is the full subcategory consisting of objects $\cE$ such that for any $\eta: \Spec R \rightarrow X$, the pullback $\eta^*\cE$ is in the idempotent completion of the pretriangulated closure of the free object $\eta^*\cO_{\bE_X}$.
\item The category $\Coh_X(\bE_X) = \Coh(\wh{\bE}_X) \subset \IndCoh(\wh{\bE}_X)$ is the smallest stable idemopotent-complete category containing the essential image of $\Coh(X)$ under $z_*$, where $z: X \hookrightarrow \bE_X$.
\end{enumerate}
In the graded setting, we define them to be the full subcategories consisting of objects which de-equivariantize to objects in the categories defined below.  When $\bE_X$ is connective, these definitions are evidently compatible with the usual definitions.
\end{defn}

We are also interested in defining various t-structures on these categories corresponding to the standard t-structures as well as sheared versions in the equivariant setting.
\begin{defn}\label{shear tstructure}
The non-graded category $\QCoh(\bE_X)$ is equipped with a standard t-structure characterized by the property that functor $p_*: \QCoh(\bE_X) \rightarrow \QCoh(X)$ is left $t$-exact; this agrees with the usual t-structure in the connective case.   In the graded case we may define a family of \emph{$n$-sheared t-structure} for each $n \in \bZ$ as follows.  On $\QCoh(X \times B\bG_m)$, we define the $n$-sheared t-structure to be the unique t-structure such that the $\bbb n$-twisted $\QCoh(B\bG_m)$-action and $q^*: \QCoh(X) \rightarrow \QCoh(X \times B\bG_m)$ are t-exact, where $q: X \times B\bG_m \rightarrow X$ is the projection.  We define the $n$-sheared t-structure on $\QCoh(\bE_X/\bG_m)$ to be the unique t-structure such that the functor $p_*: \QCoh(\bE_X/\bG_m) \rightarrow \QCoh(X \times B\bG_m)$ is left t-exact\footnote{It is a general fact for any right adjoint functor $G: \cat{D} \rightarrow \cat{C}$ between cocomplete categories and t-structure on $\cat{C}$, one can define a unique t-structure on $\cat{D}$ characterized by the property that $G$ is left t-exact.} for the $n$-sheared t-structure on $\QCoh(B\bG_m)$; this agrees with the usual t-structure when $n=0$ in the connective case.
\end{defn}

\begin{rmk}
One may characterize the connective and coconnective objects under these sheared t-structures in the following way.   Suppose we are given a presentable category $\cat{C}$ with a t-structure compatible with filtered colimits, and a pair of colimit-preserving adjoint functors $\begin{tikzcd} F: \cat{C} \arrow[r, shift left] & \cat{D} :G. \arrow[l, shift left] \end{tikzcd}$  Then there is a unique t-structure on $\cat{D}$ subject to the condition that $G$ is left t-exact.  Since $G$ is conservative this completely determines the coconnective objects, and the connective objects are generated by the image of $\cat{C}^{\leq 0}$ under $F$.
\end{rmk}

\subsubsection{}\label{koszul resolution} We now define the Koszul resolution, and the notion of Koszul dual vector bundles.  Let $\bE_X$ be a derived vector bundle.  There is a natural dg locally semifree resolution of the augmentation module $\cO_X$ given by
$$\cK_{\bE_X} := \Sym_{\cO_X}(\mathrm{cone}(\mathrm{id}_{\cE^\vee})) \rightarrow \cO_X.$$
By construction, i.e. since $\cone(\mathrm{id}_{\cE^\vee}) \simeq 0$, the map is a quasi-isomorphism, and a semi-free\footnote{Let us recall how to formulate the notion of semi-free without reference to dg categories or model structures.  By semi-free, we mean that $\cK_{\bE_X}$ is a colimit of locally free $\cO_{\bE_X}$-modules $\cO_{\bE_X} \otimes_{\cO_X} \cP$.  A semi-free presentation of a module $\cM = \colim \cO_{\bE_X} \otimes_{\cO_X} \cP_\alpha$ allows for a calculation of $\intHom_{\cO_{\bE_X}}(\cM, -) \simeq \intHom_{\cO_{\bE_X}}(\colim \cO_{\bE_X} \otimes_{\cO_X} \cP, -) \simeq \lim \cP \otimes_{\cO_X} -$, i.e. using the tautological properties that $\Hom$ commutes colimits in the source into limits, and that the relative tensor product on free modules is the forgetful functor.  The Koszul complex is semi-free since $\cone(\mathrm{id}_{\cE^\vee})$ may be endowed with a two-step filtration $0 \subset \cE \subset \cone(\mathrm{id}_{\cE^\vee})$, inducing a filtered algebra structure on $\cK_{\bE_X}$; in particular, a $\cO_{\bE_X}$-algebra structure since the filtration specifies $\Symp \cE^\vee \simeq \cO_{\bE_X}$ as a subalgebra.} complex of $\cO_{\bE_X}$-modules.  Alternatively, or more explicitly,
$$\cK_{\bE_X} \simeq \Sym_{\bE_X} (\cE^\vee[1] \otimes_{\cO_X} \cO_{\bE_X}), \;\;\;\;\;\;\; d(s) = s \in \Symp_{\bE_X}^0 \cE^\vee[1] = \cO_{\bE_X} \text{ for } s \in \cE^\vee[1]$$
with the rest of the differentials defined via the Leibniz rule.  Since $\bE_X$ was assumed to be connective, $\cK_{\bE_X}$ is finite rank in each fixed cohomological degree and internal weight, and the graded $\cO_{\bE_X}$-linear dual of $\cK_{\bE_X}$ is the interpolating complex:
$$\cR_{\bE_X} := \Sym_{\cO_X} (\cE^\vee \oplus \cE[\ng1]), \;\;\;\;\;\;\; d(s) = \eta(1)s$$
where $\eta: \cO_X \rightarrow \cE^\vee \otimes_{\cO_X} \cE[1]$ is the coevaluation map (see the definition of the complex $\cA(\cT)$ in \cite[\textsection 1.2]{MR2}).  In particular, we may compute the graded endomorphisms:
$$R\underline{\intHom}_{\bE_X}(\cO_X, \cO_X) \simeq \cR_{\bE_X} \otimes_{\cO_{\bE_X}} \cO_X \simeq \Sym_X \cE[\ng1] = \cO_{\bE^*_X[1]}.$$
Furthermore, performing the same calculation for $\bE^*_X[1]$ gives back $\bE_X$ (in particular $\cR_{\bE_X} = \cR_{\bE^*_X[1]}$); this motivates the following definition, introduced primarily for notational convenience. We also define a Tate-unsheared version which appears in $\bG_m$-equivariant contexts.

\begin{defn}
Let $\bE_X$ be a derived vector bundle with the standard (weight 1) contracting $\bG_m$-action.  We define the \emph{Koszul dual} vector bundle to $\bE_X = (X, \cE)$ to be $\kz{\bE}_X := \bE_X^*[1] = (X, \cE^\vee[1])$, with $\cO_{\bE_X^*[1]} = \Sym_X \cE[\ng1]$, which has the standard (weight $-1$) expanding $\bG_m$-action.  We define the \emph{sheared} or \emph{regraded Koszul dual} \cite[\textsection 1.7]{MR2} to be $\kzz{\bE}_X := \bE_X^*[\ng1] = (X, \cE^\vee[\ng1])$.
\end{defn}

\subsubsection{} The following result appears as \cite[Thm. 1.6.1, 1.7.1]{MR2}; we provide a proof in our set-up for convenience.

\begin{thm}[Koszul duality for formal vector bundles]\label{kd graded no functoriality}
Let $X$ be a smooth QCA stack and $\bE_X$ a derived vector bundle with the standard contracting $\bG_m$-action, and $z: X \hookrightarrow \bE_X$.  The functor $z^!$ induces an equivalence on graded categories:
$$\begin{tikzcd}[column sep=20ex] \IndCoh_X(\mathbb{E}_X/\bG_m) \arrow[r, "\simeq"', "{\kappa = z^!}", shift left=1ex] & \arrow[l, shift left=1ex, "{- \tens{\cO_{\kz{\bE}_X}} \cO_X}"]  \QCoh(\kz{\bE}_X/\bG_m) \arrow[r, "\simeq"', "{\bbb{\ng2} = \unshear}", shift left=1ex] & \QCoh(\kzz{\bE}_X/\bG_m) \arrow[l, shift left=1ex, "{\bbb{2} = \shear}"] \end{tikzcd}$$
such that $\kappa(z_*\cO_X) \simeq \cO_{\kz{\bE}_X}$, and likewise for non-graded categories.
\end{thm}
\begin{proof}
We assume that $\bE_X$ is connective; we deduce the non-connective case by shearing (see Definition \ref{cat def}), and the non-graded statements by de-equivariantization.  Let $z: X \hookrightarrow \bE_X$ denote the inclusion of the zero section, and consider the adjoint pair 
$$z_*: \begin{tikzcd} \QCoh(X) = \IndCoh(X) \arrow[r, shift left] & \arrow[l, shift left] \IndCoh(\wh{\bE}_X) \end{tikzcd}: z^!.$$  
By construction, $z^!$ is conservative since the essential image of $z_*$ generates $\IndCoh_X(\bE_X)$.  Thus $z^!$ is monadic.  Since $X$ is QCA, $\cO_X$ is compact in $\IndCoh(X)$, thus $z_*\cO_X \in \IndCoh(\bE_X/\bG_m)$ is compact and we may identify the monad using the Koszul resolution
$$z^!z_*(-) \simeq \mathcal{H}om_{\mathbb{E}_X}(\OO_X, -) \simeq  \mathcal{E}nd_{\bE_X}(\cO_X) \otimes_{\cO_X} - \simeq \cO_{\kz{\bE}_X} \otimes_{\cO_X} -.$$
By Barr-Beck-Lurie and Proposition \ref{shear cat prop} we have an equivalence $\IndCoh_X(\bE_X/\bG_m) \simeq \QCoh(\kz{\bE}_X/\bG_m)$ as desired.
\end{proof}


\begin{exmp}\label{non graded kd exmp}
Take $X = \Spec k$ and $\bE_X = \bA^1[1]$, i.e. the \emph{coconnective} derived vector bundle attached to $k[1]$ so that $\cO_{\bE_X} = k[\eta]$ with $|\eta| = 1$.  The Koszul dual bundle is $\kz{\bE}_X = \bA^1$.  Here, the Koszul complex has infinite rank in degree 0, so that in the non-graded setting we have $R\intHom_{\bE_X}(k, k) \simeq k[[x]]$ with $|x| = 0$, while in the graded setting we have $R\intHom_{\bE_X}(k, k) \simeq k[x]$.  Theorem \ref{kd graded no functoriality} tells us that
$$\IndCoh(\wh{\bA}^1[1]) \simeq \QCoh(\bA^1)$$
which may set off some alarm bells: the left-hand side appears to only have one simple object, while the right-hand side has many skyscraper objects.  In fact, this is not a mistake: the left-hand side \emph{does} have all these skyscraper objects, a feature of our definition by de-equivariantization. 
\end{exmp}

\subsubsection{}\label{sec kd graded functorial} We are interested in functorial properties of Koszul duality, as introduced in \cite[\textsection 2, 3]{MR}; we will summarize their results for convenience. There are two kinds of maps between derived vector bundles we wish to consider.
\begin{enumerate}
\item A \emph{linear morphism} $\delta: \bE_X \rightarrow \bF_X$ between two derived vector bundles on $X$ is defined by a map of perfect complexes $\delta^\flat: \cF^\vee \rightarrow \cE^\vee$, inducing a (graded) map on structure sheaves $\cO_{\bF_X} \rightarrow \cO_{\bE_X}$.  In the connective case, this morphism of derived stacks is affine.
\item A \emph{base-change morphism} $\pi: \bE_X = \bE_Y \times_Y X \rightarrow \bE_Y$ is defined by a map of derived stacks $\pi_0: X \rightarrow Y$.  In the connective case, this morphism of derived stacks is the base change along $\pi_0$.
\end{enumerate}

\begin{exmp}
The inclusion of the zero section $z: X \hookrightarrow \bE_X$ can be realized as a linear morphism given by the map of perfect complexes $\cE^\vee \rightarrow 0$, while the projection $p: \bE_X \rightarrow X$ can be realized via the map $0 \rightarrow \cE^\vee$.
\end{exmp}

There are two ways to assemble these two types of morphisms.  The basic example is that morphisms between tangent bundles are covariant, while morphisms between cotangent bundles are contravariant.
\begin{defn}\label{map of dvect}
We define a \emph{covariant morphism of derived vector bundles} $\phi: \bE_X \rightarrow \bF_Y$ to be the composition linear morphism followed by a base-change morphism, i.e. 
$$\begin{tikzcd} \bE_X \arrow[r, "\delta"] & \bF_X = \bF_Y \times_Y X \arrow[r, "\pi"] & \bF_Y\end{tikzcd}$$
i.e. explicitly, a pair $(\pi_0, \delta^\flat)$ where $\pi_0: X \rightarrow Y$ and $\delta^\flat: \pi_0^* \cF^\vee \rightarrow \cE^\vee$ is a map in $\QCoh(X)$.  We define a \emph{contravariant morphism of derived vector bundles} $\phi: \bE_X \rightarrow \bF_Y$ to be a correspondence
$$\begin{tikzcd}\bE_X & \arrow[l, "\delta"'] \arrow[r, "\pi"] \bF_X  = \bF_Y \times_Y X& \bF_Y \end{tikzcd}$$
where $\pi$ is base-changed from $\pi_0: X \rightarrow Y$ and $\delta$ is linear.  These morphisms may be composed
$$\begin{tikzcd}
\bD_W \arrow[r] & \bE_W \arrow[r, dotted] \arrow[d] & \bF_W \arrow[d, dotted] & & \bF_W \arrow[r, dotted] \arrow[d, dotted] & \bE_W \arrow[r] \arrow[d] & \bD_W \\
& \bE_X \arrow[r] & \bF_X \arrow[d] &  &\bF_X \arrow[r] \arrow[d] & \bE_X \\
& & \bF_Y &  &\bF_Y & & 
\end{tikzcd}$$
since the middle square commutes in the covariant case, and by base change since the middle square is Cartesian in the contravariant case.  In particular, the category of derived vector bundles assembles into a two different $\infty$-categories: the covariant $\cat{DVect}_k$ and the contravariant $\cat{coDVect}_k$.
\end{defn}

\subsubsection{} There is a correspondence between maps of derived vector bundles under Koszul duality, where base-change morphisms correspond covariantly, while linear morphisms correspond contravariantly, i.e. a linear morphism $\delta: \bE_X \rightarrow \bF_X$ becomes a linear morphism $\kz{\delta}: \kz{\bF}_X \rightarrow \kz{\bE}_X$, while a base-change morphism $\pi: \bE_X \rightarrow \bE_Y$ induced by $\pi_0: X \rightarrow Y$ becomes a base-change morphism $\kz{\pi}: \kz{\bE}_X \rightarrow \kz{\bE}_Y$ induced from the same morphism.  The following observation is immediate.
\begin{prop}\label{kd functor}
The Koszul duality functor $\bE_X \mapsto \kz{\bE}_X$ (and likewise the sheared Koszul duality functor $\bE_X \mapsto \kzz{\bE}_X$) define functors:
$$\kz{(-)}: \cat{DVect}_k \rightarrow \cat{coDVect}_k, \;\;\;\;\;\;\;\;\;\; \kzz{(-)}: \cat{DVect}_k \rightarrow \cat{coDVect}_k.$$
\end{prop}

\subsubsection{} Our definitions of $\QCoh(\bE_X), \IndCoh_X(\bE_X)$ in Definition \ref{cat def} allow for easy definitions of various functors.
\begin{defn}
Let $\bE_X, \bE_Y, \bF_X$ be connective derived vector bundles with compatible weight gradings.  We will consider the following functors, for $\pi: \bE_X/\bG_m \rightarrow \bE_Y/\bG_m$ and $\delta: \bE_X/\bG_m \rightarrow \bF_X/\bG_m$:
$$\pi^*: \QCoh(\bE_Y/\bG_m) \rightarrow \QCoh(\bE_X/\bG_m), \;\;\;\;\;\;\;\;\;\;\;\;\; \pi_*: \QCoh(\bE_X/\bG_m) \rightarrow \QCoh(\bE_Y/\bG_m),$$
$$\delta^*: \QCoh(\bF_X/\bG_m) \rightarrow \QCoh(\bE_X/\bG_m), \;\;\;\;\;\;\;\;\;\; \delta_*: \QCoh(\bE_X/\bG_m) \rightarrow \QCoh(\bF_X/\bG_m),$$
$$\pi_*: \IndCoh_X(\bE_X/\bG_m) \rightarrow \IndCoh_X(\bE_Y/\bG_m), \;\;\;\;\;\;\;\;\;\; \pi^!: \IndCoh_Y(\bE_Y/\bG_m) \rightarrow \IndCoh_X(\bE_X/\bG_m),$$
$$\delta_*: \IndCoh_X(\bE_X/\bG_m) \rightarrow \IndCoh_X(\bF_X/\bG_m), \;\;\;\;\;\;\;\;\;\; \delta^!: \IndCoh_X(\bF_X/\bG_m) \rightarrow \IndCoh_X(\bE_X/\bG_m)$$
which are defined in the usual way \cite{GR}.  In the non-connective case we define them by shearing, and in the non-graded setting, we define them by de-equivariantization.  When $\pi_0: X \rightarrow Y$ is finite Tor dimension, we may define $\pi^!$ on $\QCoh$ and $\pi^*$ on $\IndCoh_X$ similarly.
\end{defn}

For $\delta: \bE_X \rightarrow \bF_X$, the following relative version of the Koszul resolution (defined in Section \ref{koszul resolution}) is useful for computing the functors $\delta^*: \QCoh(\bF_X) \rightarrow \QCoh(\bE_X)$ and $\delta^!: \IndCoh(\wh{\bF}_X) \rightarrow \IndCoh(\wh{\bE}_X)$ explicitly since
$$\delta^*(-) = - \otimes^L_{\cO_{\bF_X}} \cO_{\bE_X}, \;\;\;\;\;\;\;\;\;\; \delta^!(-) = R\intHom_{\cO_{\bF_X}}(\cO_{\bE_X}, -).$$
Note that since we work with formal completions in the ind-coherent setting, the functors $(\delta_*, \delta^!)$ are always adjoint (leading to the formula for $\delta^!$).
\begin{defn}[Generalized Koszul resolution]\label{generalized koszul resolutions} 
Let $\delta: \bE_X \rightarrow \bF_X$ be a linear morphism given by a map of perfect complexes $\delta^\flat: \cF^\vee \rightarrow \cE^\vee$; we will define a $\cO_{\bF_X}$-free resolution of $\cO_{\bE_X}$.  Define $\wt{\cE} ^\vee = \cone(\fib(\delta^\flat) \rightarrow \cF^\vee)$ equipped with its canonical equivalence $\begin{tikzcd}[column sep=small] \wt{\cE}^\vee \arrow[r, "\simeq"]& \cE^\vee\end{tikzcd}$.   
Define the \emph{generalized Koszul resolution} attached to $\delta$ by
$$\cK_\delta := \Sym_{\cO_X} \wt{\cE}^\vee \rightarrow \cO_{\bE_X}.$$
By construction, this map is a quasi-isomorphism, and $\cK_\delta$ is a semi-free complex of $\cO_{\bF_X}$-modules.  
\end{defn}

\begin{exmp}
The Koszul resolution from Section \ref{koszul resolution} is the special case of the generalized Koszul resolutions in Definition \ref{generalized koszul resolutions} corresponding to the inclusion of the zero section $z: X \hookrightarrow \bE_X$, which corresponds on the Koszul dual side to the projection $p: \kz{\bE}_X \rightarrow X$.  
\end{exmp}

\begin{exmp}
Consider the map $p: \wh{\bE}_X \rightarrow X$ where $\bE_X$ is classical of rank $r$.  Viewing $p$ as a map of derived vector bundles, the functor $p^!$ is computed by the graded dual
$$p^!\cM \simeq \intHom_X(\Sym_X \cE^\vee, \cM) \simeq \Sym_X \cE \otimes_{\cO_X} \cM.$$
Viewing $p$ as a map of prestacks, the functor $p^!$ is defined via the composition $p^!: \IndCoh(X) \rightarrow \IndCoh(\bE_X) \rightarrow \IndCoh(\wh{\bE}_X)$, i.e. the $!$-pullback along $\bE_X \rightarrow X$ followed by the local cohomology functor.  The $!$-pullback takes $\cM \in \IndCoh(X)$ to $\Sym_X \cE^\vee \otimes_{\cO_X} \det(\cE^\vee)[r] \otimes_{\cO_X} \cM$ (i.e. $p^*$ twisted by the relative dualizing bundle), then computing the local cohomology 
gives the same result as above.  In the non-graded setting, we cannot use the relative Koszul complex to compute the functor directly due to convergence issues; the correct functor is computed by de-equivariantization, and likewise agrees with the functor coming from prestacks.
\end{exmp}

\subsubsection{} We now establish functoriality properties of Koszul duality with respect to smooth morphisms.  Let $\cat{QCA}_k$ denote the category of QCA stacks with morphisms given by representable functors, $\cat{smQCA}_k$ the full subcategory of smooth QCA stacks.  We use the superscript $\mathrm{sm}$ to indicate we only take smooth morphisms, and the superscript $+$ to indicate we only take eventually coconnective morphisms \cite[\textsection I.4 Def. 3.1.2]{GR}.
\begin{prop}
We have functors of $\infty$-categories
$$(\IndCoh_X, \delta^!\pi^!): \cat{DVect}(\cat{QCA}_k)^{\opp} \longrightarrow \cat{Pr}^L_k,$$
$$(\QCoh, \delta_*\pi^!): \cat{DVect}(\cat{QCA}_k^+)^{\opp} \longrightarrow \cat{Pr}^L_k,$$
in both the graded and non-graded settings.
\end{prop}
\begin{proof}
By shearing and $\QCoh(B\bG_m)$-equivariance, it suffices to prove the claims under the assumption that the derived vector bundles are connective, and in the graded setting.  To deduce the non-connective case, we write $\cat{DVect}$ as a colimit of categories $\cat{DVect}^{\geq -n}$ (and likewise for $\cat{coDVect}$), which we construct functors out of by twisting by shearing automorphisms.  These functors compatible with the inclusions, so we can take the colimit.  To deduce the non-graded version we compose with the de-equivariantization functor.  Now, in the connective case, functoriality for $\IndCoh_X$ then follows by functoriality for $\IndCoh$, which is by construction.  For $\QCoh$, the claim follows by base change in \cite[Prop. 7.2.9]{indcoh} and a standard descent argument.
\end{proof}

\medskip

We now show that Koszul duality is functorially compatible with pullbacks.  We expect a more general functoriality, i.e. out of a category of correspondences \cite[\textsection I.7]{GR}, but postpone its discussion to the sequel, as we will not need it in this paper.  The following results also appear in  \cite[Prop. 2.1.1, 3.2.1]{MR2}.
\begin{thm}\label{kd graded functorial}\label{big functoriality prop}
There is a natural isomorphism commuting the diagram
$$\begin{tikzcd}[column sep=large]
\cat{DVect}(\cat{smQCA}_k) \arrow[r, "\IndCoh_X"] \arrow[d, "\kz{(-)}"']  & \cat{Pr}_k^L \arrow[d, "\kappa"] \\
\cat{coDVect}(\cat{smQCA}_k) \arrow[r, "\QCoh"] & \cat{Pr}_k^L
\end{tikzcd}$$
intertwining the Koszul duality functors of Proposition \ref{kd functor} and Theorem \ref{kd graded no functoriality}.  Furthermore, for any individual linear morphism $\delta$ or base change morphism $\pi$, $\kappa$ exchanges the following functors, which preserve the small subcategories appearing on their respective sides under the indicated conditions.
\begin{center}
\begin{tabular}{c|c|c}
$\Coh_X(\bE_X)$ & $\Perf(\kz{\bE}_X)$ & condition \\ \hline
$\pi^*$ & $\kz{\pi}^*$ & none \\
$\pi_*$ & $\kz{\pi}_*$ & $\pi_0$ proper \\
$\pi^!$ & $\kz{\pi}^!$ & none \\
$\delta_*$ & $\kz{\delta}^*$ & none \\
$\delta^!$ & $\kz{\delta}_*$ & $\cone(\delta^\flat)$ perfect in odd degrees
\end{tabular}
\end{center}
Furthermore, we have graded versions of all statements.
\end{thm}
\begin{proof}
The claim for linear morphisms $\delta: \bE_X \rightarrow \bF_X$ follows since the map of monads $z^!z_* \rightarrow z^!\delta^!\delta_*z_*$ defining the functor $\kappa$ is exactly the map of algebras $\cO_{\kz{\bE}_X} \rightarrow \cO_{\kz{\bF}_X}$, identifying $\delta^!$ with $\kz{\delta}_*$.  The claim for $\delta_*$ follows by passing to left adjoints.  For a base-change morphism $\pi: \bE_X \rightarrow \bE_Y$ defined by $\pi_0: X \rightarrow Y$, the claim for $\pi^*$ follows from the observation that $\pi^*: \IndCoh_X(\bF_X) \rightarrow \IndCoh_X(\bE_X)$ is characterized by the property that it commutes with colimits and that $\pi^*(z_* \cM) \simeq z_* \pi_0^* \cM$.  Applying $\kappa$, this means that the corresponding Koszul dual functor should send $\cO_{\kz{\bE}_Y} \otimes_{\cO_Y} \cM$ to $\cO_{\bE_X} \otimes_{\cO_X} \pi_0^* \cM$, which characterizes the functor $\kz{\pi}^*$.  The claim for $\pi_*$ follows by passing to right adjoints.  The claim for $\pi^!$ follows similarly to that for $\pi^*$ by introducing a twist by the relative dualizing bundle $\omega_{X/Y}$.

We now address the claims regarding small categories.  It suffices to prove them on the covariant side.  Regarding $\pi$, note that we have Cartesian squares
$$\begin{tikzcd}
X \arrow[d, "\pi_0"] \arrow[r, "z"] & \bE_X \arrow[d, "\pi"]\\
Y \arrow[r, "z"] & \bE_Y.
\end{tikzcd}$$
For the pullback, we note that $\Coh_Y(\bE_Y)$ is generated by the essential image of $\Coh(Y)$ under $z_*$, and use base change on the left square.  The claim for pushforward follows by commutativity of the left square.  Regarding linear morphisms, the claim for pushforwards is evident.  The claim for $!$-pullbacks follows using the generalized Koszul resolution (Definition \ref{generalized koszul resolutions}), i.e. the generalized Koszul resolution is a semi-free $\cO_{\bF_X}$-resolution whose underlying sheaf is $\Symp_{\cO_{\bF_X}}(\cone(\delta^\flat))$, thus is perfect and bounded when $\cone(\delta^\flat)$ is perfect in odd degrees.
\end{proof}

\begin{rmk}
The condition that $\cone(\delta^\flat)$ is perfect in odd degrees is satisfied in the following important example.  Consider the case of a representable map $\pi: X \rightarrow Y$, and take $\bE_X = \bT_X[\ng1] \rightarrow \bE_Y = \bT_Y[\ng1]$.  This map is quasi-smooth if and only if $X \rightarrow Y$ is smooth, where $\cone(\delta^\flat) = \Omega^1_{X/Y}[1]$ is perfect in degree $-1$.  We also note that the conditions on $\cone(\delta^\flat)$ are fixed under Koszul duality.  That is, $\cone(\delta^\flat)$ and $\cone(\kz{\delta}^\flat)$ have the same Tor-amplitude parity, since $\cone(\kz{\delta}^\flat) = \cone(\delta^{\flat, \vee}[\ng1]) = \fib(\delta^\flat)^\vee[\ng1] = \cone(\delta^\flat)^\vee.$
\end{rmk}

\subsection{Renormalized linear Koszul duality}\label{ren linear kd}

We now arrive at our main point of departure from the linear Koszul duality of \cite{MR2}: we wish to avoid invocations of Grothendieck duality in our statements, which are employed in \emph{loc. cit.} to produce Koszul duality statements between categories of coherent sheaves on both sides.  The introduction of contravariance is inconvenient when working with large categories (i.e. one must consider pro-objects rather than ind-objects on one side), so our approach will be to \emph{renormalize} the category on the covariant side of Koszul duality so that, roughly, it contains injective objects instead of projective ones.  In the language of \cite{rnkd} we work with the coderived category rather than the contraderived category.

\subsubsection{} We are primarily interested in specializing to the case $\bE_X = \bT_X[\ng1] = \Spec_X \Sym_X \Omega^1_X[1]$ where $\Omega^1_X$ denotes the cotangent complex of $X$, i.e. the odd tangent bundle of \cite{loops and conns} (see Definition \ref{odd tangent defn}).  The Tor-amplitude of $\bT_X[\ng1]$ depends on how singular $X$ is and the degree of stackiness of $X$; we consider the following cases.
\begin{enumerate}
\item The derived vector bundle $\bE_X$ is perfect in degree $1$, i.e. the fibers are entirely in derived directions.  Then, the Koszul dual $\kz{\bE}_X$ is perfect in degree $-2$, and the sheared Koszul dual $\kzz{\bE}_X$ is perfect in degrees $0$.  The bundle $\bE_X$ is quasi-smooth over $X$, and $\kzz{\bE}_X$ is smooth over $X$  If we take $\bE_X = \bT_X[\ng1]$, this corresponds to the case where $X$ is a smooth scheme.
\item The derived vector bundle $\bE_X$ is perfect in degree $0$, i.e. is a vector bundle in the classical sense.  Then, the Koszul dual $\kz{\bE}_X$ is perfect in degree $-1$, and the sheared Koszul dual $\kzz{\bE}_X$ is perfect in degree $1$.  The bundle $\bE_X$ is smooth over $X$ and $\kzz{\bE}_X$ is quasi-smooth over $X$.  If we take $\bE_X = \bT_X[\ng1]$, this corresponds to the case where $X$ is a classifying stack. 
\item The derived vector bundle $\bE_X$ is perfect in degrees $[0,1]$, i.e. the common generalization of the above two.  Then, the Koszul dual $\kz{\bE}_X$ is perfect in degrees $[-2, -1]$, and the sheared Koszul dual $\kzz{\bE}_X$ is perfect in degrees $[0,1]$.  The bundle $\bE_X$ is quasi-smooth over $X$, and likewise for $\kzz{\bE}_X$.  If we take $\bE_X = \bT_X[\ng1]$, this corresponds to the case where $X$ is a smooth Artin 1-stack.
\end{enumerate}
We will soon impose the assumption that the derived vector bundle $\bE_X$ is perfect in degrees $[0, 1]$, i.e. \emph{relatively quasi-smooth} over $X$, which in the case of the odd tangent bundle corresponds to the case where $X$ is an Artin 1-stack.

\subsubsection{} We first discuss a certain enlargement of the category $\Coh(\wh{\bE}_X) \subset \IndCoh(\wh\bE_X)$, or more generally for any formal completion $\wh{X}_Z$ of an Artin stack along a closed substack, introduced in \cite[Def. 2.2.7]{CD}.  The main properties of this category were discussed in Section \ref{intro ren}.  
\begin{defn}\label{def hcoh}
For a closed substack $i: Z \hookrightarrow X$ of an Artin stack $X$, the category $\hCoh_Z(X) = \hCoh(\wh{X}_Z)$ of \emph{continuous ind-coherent sheaves} on the formal completion is defined to be the full subcategory of $\IndCoh(\wh{X}_Z)$ of t-bounded and $!$-almost perfect objects (i.e. objects such that the $!$-restriction along any closed substack of $\wh{X}_Z$ is left t-bounded with coherent cohomology).  
\end{defn}

The categories $\Coh(\wh{\bE}_X)$ and $\hCoh(\wh{\bE}_X)$ only differ when $\bE_X$ has a degree 0 (i.e. classical) component.  In particular, in the case of odd tangent bundles, this situation arises only when $X$ is a stack, and not when $X$ is a scheme.
\begin{prop}\label{prop coh equals}
Let $X$ be an Artin stack and $\bE_X$ a connective derived vector bundle with a weight 1 contracting $\bG_m$-action.
If $\bE_X$ is perfect in strictly positive degrees (i.e. strictly connective), then 
$$\Coh_X(\bE_X) = \hCoh_X(\bE_X) =\Coh(\bE_X).$$
\end{prop}
\begin{proof}
We first show that $\Coh_X(\bE_X) = \Coh(\bE_X)$.  By definition, $\Coh_X(\bE_X) \subset \Coh(\bE_X)$.  For the opposite inclusion, for any bounded complex in $\Coh(\bE_X)$ by devissage it suffices to consider to argue for objects in the heart, i.e. for $H^0(\Sym_X \cE^\vee)$-modules; this is a module-finite algebra over $\cO_X$, thus is generated by $\cO_X$-modules.  Next, we show that $\hCoh_X(\bE_X) = \Coh(\bE_X)$.  The inclusion $\Coh_X(\bE_X) \subset \hCoh_X(\bE_X)$ is clear.  For the opposite inclusion, it suffices to show that the complexes in $\hCoh_X(\bE_X)$ have coherent cohomology.  By assumption the complexes are t-bounded; assume that the lowest nonvanishing cohomology of $\cF$ is in degree 0, so that $\cH^0(\cF) \subset \cF$ is a sheaf of $\cO_{\bE_X}$-submodules.  Since $\shExt^0_X(\cO_X, \cF) = \cH^0(\cF)$, we have that $\cH^0(\cF)$ is coherent.  We may deduce coherence of the remaining cohomology groups by devissage.
\end{proof}

\subsubsection{}  When $\bE_X$ is perfect in degree 0, i.e. $\bE_X$ is a classical vector bundle over $X$, these categories no longer coincide.  Here the Koszul duality equivalence of Theorem \ref{kd graded no functoriality} gives an equivalence
$$\begin{tikzcd}[column sep=huge]\Coh_X(\bE_X/\bG_m) \arrow[r, "\unshear \circ \kappa", "\simeq"'] & \Perf(\kzz{\bE}_X/\bG_m).\end{tikzcd}$$
Our goal is to enlarge the right-hand side to $\Coh(\kzz{\bE}_X/\bG_m)$ and identify the corresponding category on the left.  Consider the example where $X = \Spec k$, and $\bE_X = \bA^1 = \Spec k[x]$ with Koszul dual $\kz{\bE}_X = \Spec k[\eta]$ with $|\eta| = 1$.  For free, we have $\kappa(z_*\cO_X) \simeq \cO_{\kz{\bE}_X}$.  Koszul dually, we may write $z_* \cO_X \in \Coh(\kz{\bE}_X/\bG_m)$ in terms a ``free resolution'', i.e. as the colimit (where $\langle - \rangle$ denotes a weight shift):
$$z_* \cO_X = \left( \begin{tikzcd} \cdots \arrow[r, "\eta"] & \cO_{\kz{\bE}_X}[-2]\langle 2 \rangle \arrow[r, "\eta"] & \cO_{\kz{\bE}_X}[-1]\langle 1 \rangle \arrow[r, "\eta"] & \cO_{\kz{\bE}_X}\end{tikzcd} \right)$$
$$= \colim_k \left(\cone(\eta: \cO_{\kz{\bE}_X}[-k]\langle k \rangle \rightarrow \cone(\cdots \cone(\eta: \cO_{\kz{\bE}_X}[-2]\langle 2 \rangle \rightarrow \cone(\eta: \cO_{\kz{\bE}_X}[-1]\langle 1 \rangle \rightarrow \cO_{\kz{\bE}_X})))\right).$$
Since $\kappa^{-1}$ sends $\eta$ to the unique $\Ext^1(k\langle i \rangle, k[x]/x^{i} \langle i-1\rangle)$, we have
$$\kappa^{-1}(z_*\cO_X) \simeq \colim k[x]/x^{i+1} \langle i \rangle \simeq k[x,x^{-1}]/x k[x] \simeq \omega_{\wh{\bA}^1/\bG_m}$$
i.e. the injective hull of the augmentation module $k$, which lives in $\hCoh(\wh{\bA}^1/\bG_m)$.  This example suggests that when $\bE_X$ is perfect in degree 0, Koszul duality should give an equivalence between $\hCoh(\wh{\bE}_X)$ and $\Coh(\kz{\bE}_X)$.

\subsubsection{} With shearing $\kzz{\bE}_X$ is perfect in degrees $[0, 1]$ and we have  $\Perf(\kzz{\bE}_X) \subset \Coh(\kzz{\bE}_X)$, but without shearing $\kz{\bE}_X$ is perfect in degrees $[-2, -1]$ and the categories $\Perf(\kz{\bE}_X)$ and $\Coh(\kz{\bE}_X)$ do not have a containment relation.  We introduce the following category for use in statements in the non-sheared setting.
\begin{defn}\label{kcoh def}
We define the category $\KCoh(\bE_X) \subset \QCoh(\bE_X)$ of \emph{Koszul-sheared coherent sheaves} to be the full subcategory of complexes of $\cO_{\bE_X}$-modules whose cohomology is finitely generated over the $\cO_X$-subalgebra of $\cH^\bullet(\cO_{\bE_X}) = \cH^\bullet(\Sym_X \cE^\vee)$ generated by $\cH^2(\cE^\vee)$.  A standard argument by long exact sequences shows that this subcategory is pretriangulated.
\end{defn}


\begin{rmk}\label{kcoh rmk}
The category $\KCoh$ will only appear for us on the Koszul dual side (i.e. for the bundle $\kz{\bE}_X$).  Its name is motivated by the fact that  $\KCoh(\kz{\bE}_X/\bG_m) \simeq \Coh(\kzz{\bE}_X/\bG_m)^{\shear}$ under the Koszul shearing.  
When $\bE_X$ is perfect in degrees $[-2, -1]$ (i.e. dual to a bundle perfect in degrees $[0, 1]$), the subcategory $\KCoh(\bE_X) \subset \QCoh(\bE_X)$ has an equivalent description as the full category of sheaves whose cohomology is finitely generated over $\cH^\bullet(\cO_{\bE_X})$.
\end{rmk}

\subsubsection{} Having introduced the requisite definitions we now prove our renormalization of Koszul duality.
\begin{thm}[Renormalized Koszul duality for formal vector bundles]\label{kd graded ren}
Assume that $X$ is smooth QCA Artin stack, and assume that $\bE_X$ is perfect in degrees $[0, 1]$.  The Koszul duality functor $\kappa: \IndCoh_X(\bE_X) \rightarrow \QCoh(\kz{\bE}_X)$ of Theorem \ref{kd graded no functoriality} restricts to an equivalence of categories
$$\begin{tikzcd}
\hCoh_X(\bE_X) \arrow[r, "{\kappa}", "\simeq"'] & \KCoh(\kz{\bE}_X).\end{tikzcd}$$
Furthermore, assuming $X$ and $Y$ are smooth, let $\pi: \bE_X \rightarrow \bE_Y$ be a base-change morphism induced from $\pi_0: X \rightarrow Y$, and $\delta: \bE_X \rightarrow \bF_X$ a linear morphism.  Then, Koszul duality exchanges the following functors under the indicated conditions:
\begin{center}
\begin{tabular}{c|c|c}
$\wh{\Coh}_X(\bE_X)$ & $\KCoh(\kz{\bE}_X)$ & condition \\ \hline
$\pi_*$ & $\kz{\pi}_*$ & $\pi_0$ proper \\
$\pi^!$ & $\kz{\pi}^!$ & none \\
$\delta_*$ & $\kz{\delta}^*$ & $\cH^0(\mathrm{fib}(\cE^\bullet \rightarrow \cF^\bullet)) = 0$ \\
$\delta^!$ & $\kz{\delta}_*$ & the map $\cH^1(\cE^\bullet) \rightarrow \cH^1(\cF^\bullet)$ is surjective
\end{tabular}
\end{center}
We have similar statements in the graded setting.
\end{thm}
\begin{proof}
We will show that the Koszul duality functor of Theorem \ref{kd graded no functoriality} preserves the indicated subcategories.  The claim is smooth local on $X$, so we may assume that $\cE = (\cE^0 \rightarrow \cE^{1})$ is a complex of locally free sheaves in degrees $[0, 1]$.  There are closed embeddings of bundles sitting in a fiber square
$$\begin{tikzcd}
\bE_X^1 \arrow[r, "\delta", hook] \arrow[d, hook] & \bE_X \arrow[d, hook, "\iota"] \\ 
X \arrow[r, hook] & \bE_X^0
\end{tikzcd}$$
where $\bE_X^1$ is the derived vector bundle defined by $(X, \cE^1[-1])$ and $\bE_X^0$ is defined by $(X, \cE^0)$, and the maps are given by the hard truncations (which are not canonical and depend on our presentation of $\cE$).  Since $\cone(\delta^\flat) = (\cE^{0})^\vee[1]$, by Theorem \ref{kd graded functorial} we have a commuting diagram
$$\begin{tikzcd}
\IndCoh_X(\bE^{1}_X) \arrow[r, "\kappa", "\simeq"'] & \QCoh(\kz{\bE}_X^{1}) \\
\IndCoh_X(\bE_X) \arrow[u, "\delta^!"] \arrow[r, "\kappa", "\simeq"'] & \QCoh(\kz{\bE}_X). \arrow[u, "\kz{\delta}_*"]
\end{tikzcd}$$ 
Furthermore, note that the closed embedding $\bE^{1}_X \hookrightarrow \bE_X$ is quasi-smooth  since it is based changed from $X \hookrightarrow \bE_X^0$; thus we may apply \cite[Prop. 2.2.13]{CD}, and we have that $\hCoh_X(\bE_X) \subset \IndCoh(\bE_X)$ may be characterized as the full subcategory consisting of objects $\cF$ such that $\delta^! \cF \in \Coh(\bE_X^1) = \Coh_X(\bE_X^1)$, the latter identification since $\bE^{1}_X$ is perfect in degree 1.  The category $\Coh_X(\bE_X^1)$ is Koszul dual to $\Perf(\kz{\bE}^{1}_X)$, and by functoriality we have that $\hCoh_X(\bE_X) \subset \IndCoh_X(\bE_X)$ is Koszul dual to the full subcategory of objects $\cG \in \QCoh(\kz{\bE}_X)$ such that $\kz{\delta}_*\cG \in \Perf(\kz{\bE}^{1}_X)$.  Since $\kz{\bE}_X$ is perfect in degrees $[-2, -1]$ and $\kz{\bE}_X^1$ is perfect in degree -2, this Koszul dual subcategory is exactly $\KCoh(\kz{\bE}_X)$.

The functoriality claims are automatic, following Theorem \ref{kd graded functorial}, if we show the functors preserve the given subcategories; it suffices to prove this on the dual side.  We first address $\kz{\delta}^*$.  Let $\kz{\delta}^\flat: \cE[-1] \rightarrow \cF[-1]$ be the structure map for $\kz{\delta}: \kz{\bF}_X \rightarrow \kz{\bE}_X$.  We have that $\cone(\kz{\delta}^\flat)$ has cohomological amplitude $[0, 2]$, and the degree 2 term is the degree 1 part of $\cF$.  Thus, if the degree 0 part vanishes, then the relative Koszul resolution (see Definition \ref{generalized koszul resolutions}) $\Sym_{\cO_{\kz{\bF}_X}}(\cone(\kz{\delta}^\flat))$ is in $\KCoh$, which establishes the claim.  For $\kz{\delta}_*$, we need the map $\cO_{\kz{\bE}_X} \rightarrow \cO_{\kz{\bF}_X}$ to be surjective on $\cH^2$, which is exactly the stated condition.
\end{proof}

\begin{rmk}
The above result is false beyond the relatively quasi-smooth case, i.e. without the assumption that $\bE_X$ is perfect in degrees $[0, 1]$.  For example, if $\bE_X$ is perfect in degree 3, then $\kz{\bE}_X$ is perfect in degree 4, but $\kappa(\cO_X) \simeq \cO_{\kz{\bE}_X}$, which is freely generated in degree 4 and is not in $\KCoh(\bE_X)$.  This issue may be easily avoided by modifying the definition of $\KCoh(\kz{\bE}_X)$.  More seriously, if $\bE_X$ is perfect in degrees $[0, 2]$, then $\kappa^{-1}(\cO_X) \simeq \cO_{\bE_X}^\sharp$ is the graded dual, which is not t-bounded.  Roughly, the issue is that in this case, $\kz{\bE}_X$ is perfect in degrees $[-1, -3]$, and $\cO_X$ is not ``perfect for the degree -3 part''.  To allow for this object requires a modification of the definition of $\hCoh_X(\bE_X)$. 
\end{rmk}

\subsection{Odd tangent bundles}\label{ren loops}

We now introduce our primary derived vector bundle of interest, the odd tangent bundle, and its dual bundle.
\begin{defn}\label{odd tangent defn}
Let $X$ be an Artin 1-stack with cotangent complex $\Omega^1_X$ and tangent complex $\cT_X$.  
\begin{enumerate}
\item We define the \emph{$(-1)$-shifted} or \emph{odd tangent bundle} to be the derived vector bundle
$$\bT_X[\ng1] := (X, \cT_X[\ng1]), \;\;\;\;\;\;\;\;\;\;\;\;\;\; \cO_{\bT_X[\ng1]} = \Sym_X \Omega^1_X[1].$$
Since $X$ is an Artin 1-stack, this is a connective vector bundle on $X$, and we view it as an Artin 1-stack.  We also introduce the \emph{formal odd tangent bundle} $\wh{\bT}_X[\ng1]$ to be its completion along the zero section $X \subset \bT_X[\ng1]$.
\item Its Koszul dual bundle is the \emph{2-shifted cotangent bundle} $\bT^*_X[2] := (X, \Omega^1_X[2])$, and its regraded Koszul dual is the \emph{(classical) cotangent bundle} $\bT^*_X := (X, \Omega^1_X)$.  We always view $\bT^*_X[2]$ as a derived vector bundle and not as a stack (see Warning \ref{prestack warning}).  When $X$ is smooth, $\bT^*_X$ is connective, and we view it as an Artin 1-stack.
\end{enumerate}
To a map $f: X \rightarrow Y$ we can associate the following functorialities. 
\begin{enumerate}
\item We have a map of $(\ng1)$-shifted tangent bundles\footnote{It might be better to use the notation $d[\ng1]f$, but this is cumbersome.}
$$df: \begin{tikzcd}[column sep=huge] \bT_X[\ng1] \arrow[r, "\delta_f"] & \bT_Y[\ng1] \times_Y X \arrow[r, "\pi_f"] & \bT_Y[\ng1] \end{tikzcd}$$ 
where $\pi_f$ is the base change of $f$ and $\delta_f$ is defined by the canonical map on cotangent complexes $f^* \Omega^1_Y \rightarrow \Omega^1_X$.  This map induces a morphism on formal odd tangent bundles  $\wh{d}f: \wh{\bT}_X[\ng1] \longrightarrow \wh{\bT}_Y[\ng1].$  We are interested the functors 
$$\wh{d}f^!: \IndCoh(\wh{\bT}_Y[\ng1]) \rightarrow \IndCoh(\wh{\bT}_X[\ng1]), \;\;\;\;\;\;\; \wh{d}f_*: \IndCoh(\wh{\bT}_X[\ng1]) \rightarrow \IndCoh(\wh{\bT}_Y[\ng1]).$$
\item Functoriality for 2-shifted or classical cotangent bundles are given by the usual correspondence on the Koszul dual side
$$\begin{tikzcd}
\bT_X^*[2] & \bT^*_Y[2] \times_Y X \arrow[l, "\kz{\delta}_f"'] \arrow[r, "\kz{\pi}_f"] & \bT^*_Y[2].
\end{tikzcd}$$
We are interested in the functors 
$$\kz{d}f^! := \kz{\delta}_* \kz{\pi}^!: \QCoh(\bT^*_Y[2]) \rightarrow \QCoh(\bT^*_X[2]), \;\;\;\;\;\;\;\;\;\; \kz{d}f_* := \kz{\pi}_*\kz{\delta}^*: \QCoh(\bT^*_X[2]) \rightarrow \QCoh(\bT^*_Y[2]).$$ 
In order for the functor $\kz{\pi}_f^!$ to be well-defined on quasi-coherent sheaves, we require $f$ to be finite Tor-dimension, e.g. $X, Y$ smooth \cite[\textsection 7.2]{indcoh}.
\end{enumerate}
\end{defn}

\subsubsection{}\label{graded functoriality} We establish a few basic properties of these functors.
\begin{prop}\label{odd tangent functorial prop}
Let $f: X \rightarrow Y$ be a map of Artin 1-stacks representable by algebraic spaces.  We let $\langle r \rangle$ denote a weight shift in the positive $r$ direction (which may be ignored in the non-graded setting).
\begin{enumerate}[(a)]
\item Letting $\iota_X: \wh{\bT}_X[\ng1] \hookrightarrow \bT_X[\ng1]$ denote the inclusion of the formal neighborhood (and likewise $\iota_Y$), we have $\iota_X \circ \wh{d}f^! \simeq df^! \circ \iota_Y$, i.e the following diagram commutes.
$$\begin{tikzcd}
\IndCoh(\wh{\bT}_X[\ng1]) \arrow[r, hook, "\iota_{X*}"] & \IndCoh(\bT_X[\ng1]) \\
\IndCoh(\wh{\bT}_Y[\ng1]) \arrow[r, hook, "\iota_{Y*}"] \arrow[u, "\wh{d}f^!"] & \IndCoh(\bT_Y[\ng1]).\arrow[u, "df^!"'] 
\end{tikzcd}$$
In particular, the functor $\wh{d}f^!$ may be computed by $df^! = \delta_f^! \pi_f^!$.
\item When $f$ is smooth of relative dimension $r$, $df$ is finite Tor-dimension, thus we may define a functor 
$$df^*: \IndCoh(\bT_Y[\ng1]) \rightarrow \IndCoh(\bT_X[\ng1]).$$  
Furthermore, we have a canonical equivalence $df^* \langle -r \rangle \simeq df^!$, i.e. $df$ is Calabi-Yau of relative dimension 0.
\item The functors $(\wh{d}f_*, \wh{d}f^!)$ are adjoint when $f$ is proper, and $(\wh{d}f^!\langle r \rangle, \wh{d}f_*)$ are adjoint when $f$ is smooth.
\item Assume that $f$ has finite Tor-dimension.  The Koszul dual functors $(\kz{d}f_*, \kz{d}f^!)$ are adjoint when $f$ is proper.  When $f$ is smooth of relative dimension $r$, the functors $(\kz{d}f^!\langle r \rangle, \kz{d}f_*)$ are adjoint; in particular, if $f$ is \'{e}tale or an open embedding, $(\kz{d}f^!, \kz{d}f_*)$ are adjoint.
\end{enumerate}
\end{prop}
\begin{proof}
The claim (a) follows by base change if $df^{-1}(\{0\}_Y) \subset \{0\}_X$, i.e. if the (classical reduced) inverse image of the zero section lives in the zero section.  We factor the map $\bT_X[-1] \rightarrow \bT_Y[-1] \times_Y X \rightarrow \bT_Y[-1]$; the claim is clearly true for the latter map.  For the former map, it is equivalent to the claim that $\Omega^1_{X/Y}[1]$ has vanishing cohomology in positive degrees, which follows by representability.  Claim (b) is standard; see for example \cite[Lem. 3.12]{BCHN}.  For claim (c), the argument in (a) tells us that $\delta: \bT_X[-1] \rightarrow \bT_Y[-1]$ is a closed embedding when $f$ is representable, while $\pi$ is proper, thus $df$ is proper, and the claim follows from the usual adjunction for ind-coherent sheaves \cite[Prop. 7.2.9]{indcoh} and (a).  The adjunction $(\wh{d}f^!, \wh{d}f_*)$ follows from (b).

For (d), the adjunction when $f$ is proper follows from the adjunctions $(\kz{\delta}^*, \kz{\delta}_*)$ and $(\kz{\pi}_*, \kz{\pi}^!)$.  When $f$ is smooth, we may compute the right adjoint $(\kz{d}f^!)^R = (\kz{\pi}^!)^R(\kz{\delta}_*)^R = (\kz{\pi}^*)^R (- \otimes \omega_{X/Y}[r])^R \kz{\delta}^! = (\kz{\pi}_* \otimes \omega^{-1}_{X/Y}[-r])(\kz{\delta}^* \otimes \omega_{\bT^*_X[2]/\bT^*_Y[2] \times_Y X}[-r]) = \kz{d}f_*\langle r\rangle$.  We need to show the last identification, i.e. $\omega_{\bT^*_X[2]/\bT^*_Y[2] \times_Y X} \simeq \omega_{X/Y}[2r]\langle r\rangle$.  The map $\bT^*_X[2] \rightarrow \bT^*_Y[2] \times_Y X$ is not quasi-smooth; however, its Koszul shear $\bT^*_X \rightarrow \bT^*_Y \times_Y X$ is, and its dualizing complex is readily computed to be $\omega_{X/Y}\langle r \rangle$.  The claim then follows by unshearing.
\end{proof} 

\begin{rmk}
Note that after shearing, i.e. when we consider $\bT^*_X$ instead of $\bT^*_X[2]$, for $f: X \rightarrow Y$ a smooth morphism of constant relative dimension, the functors $(\kz{d}f^![-2 \dim f]\langle -\dim f \rangle, \kz{d}f_*)$ on the Koszul dual side are adjoint.  This may be viewed as the classical limit of the category of $\cD$-modules, and explains why the second left-adjoint of $f^!$ has a shift by $2 \dim f$ for $\cD$-modules, yet for odd tangent bundles the second left-adjoint of $f^!$ is $f^!$ itself.
\end{rmk}

\subsubsection{} We now relate the two Koszul dual sides by specializing the results of Sections \ref{linear koszul} and \ref{ren linear kd} to this setting.  We now restrict our attention to the smooth case; recall that $\cat{smQCA}_k$ is the category of QCA stacks with morphisms given by representable functors.
\begin{thm}\label{graded functoriality thm}
Consider the functors
$$\IndCoh(\wh{\bT}_\bullet[\ng1]), \QCoh(\bT^*_\bullet[2]): \cat{smQCA}_k^{\opp} \longrightarrow \cat{Pr}^L_k.$$ 
The Koszul duality functor $\kappa$ from Theorem \ref{kd graded no functoriality} defines a natural isomorphism of functors
$$\begin{tikzcd}[column sep=large] \IndCoh(\wh{\bT}_\bullet[\ng1]) \arrow[r, "\simeq"', "\kappa"] & \QCoh(\bT^*_\bullet[2]) \end{tikzcd}$$
satisfying the following (with all statements admitting graded versions).
\begin{enumerate}
\item The functor $\kappa$ restricts to equivalences (see \ref{prop coh equals} for the shearing functor):
$$\begin{tikzcd} 
\IndCoh(\wh{\bT}_X[\ng1]) \arrow[r, "\simeq"', "\kappa"] & \QCoh(\bT^*_X[2]) & \IndCoh(\wh{\bT}_X[\ng1]/\bG_m) \arrow[r, "\simeq"', "\unshear \circ \kappa"] & \QCoh(\bT^*_X/\bG_m) \\ 
\hCoh(\wh{\bT}_X[\ng1]) \arrow[r, "\simeq"', "\kappa"] \arrow[u, hook] & \KCoh(\bT^*_X[2]) \arrow[u, hook]  & \hCoh(\wh{\bT}_X[\ng1]/\bG_m) \arrow[r, "\simeq"', "\unshear \circ \kappa"] \arrow[u, hook]  & \Coh(\bT^*_X/\bG_m) \arrow[u, hook] \\
\Coh(\wh{\bT}_X[\ng1]) \arrow[r, "\simeq"', "\kappa"] \arrow[u, hook] & \Perf(\bT^*_X[2]) \arrow[u, hook] & \Coh(\wh{\bT}_X[\ng1]/\bG_m) \arrow[r, "\simeq"', "\unshear \circ \kappa"] \arrow[u, hook] & \Perf(\bT^*_X/\bG_m). \arrow[u, hook]
\end{tikzcd}$$
\item Pullback along smooth maps and pushforward along proper maps preserves the small categories in the above diagram.
\item If $X$ is a smooth scheme, we have equalities
$$\Coh(\wh{\bT}_X[\ng1]) = \hCoh(\wh{\bT}_X[\ng1]) = \Coh(\bT_X[\ng1]),$$
$$\Perf(\bT^*_X[2]) = \KCoh(\bT^*_X[2]), \;\;\;\;\;\;\;\;\;\; \Perf(\bT^*_X) = \Coh(\bT^*_X).$$
\item Let $U$ be a smooth stack and $p: U \rightarrow X$ a smooth atlas.  The full subcategory $\hCoh(\wh{\bT}_X[\ng1]) \subset \IndCoh(\wh{\bT}_X[\ng1])$ consists of objects $\cF \in \IndCoh(\wh{\bT}_X[\ng1])$ such that $dp^!\cF \in \hCoh(\wh{\bT}_U[\ng1]) = \Coh(\bT_U[\ng1])$.
\item Letting $z: X \hookrightarrow \bT_X[\ng1]$ and $\kz{z} = p: \bT^*_X[2] \rightarrow X$, we have compatibility with induction and restriction to the base $X$, i.e. commuting diagrams
$$\begin{tikzcd}
\Coh(X) \arrow[d, "z_*"'] \arrow[r, equals] & \Perf(X) \arrow[d, "p^*"'] & \Coh(X) \arrow[r, equals] & \Perf(X) \\
\Coh_X(\bT_X[\ng1]) \arrow[r, "\simeq"', "\kappa"] & \Perf(\bT^*_X[2]) & \Coh_X(\bT_X[\ng1]). \arrow[r, "\simeq"', "\kappa"] \arrow[u, "z^!"] & \Perf(\bT^*_X[2]). \arrow[u, "p_*"]  
\end{tikzcd}$$
Furthermore, the induction functors are compatible with pushforward, while the restriction functors are compatible with pullback.
\end{enumerate}
\end{thm}
\begin{proof}
The main statement follows from Theorem \ref{kd graded no functoriality}.  Statements (1)-(3) and (5) follow from Theorems \ref{big functoriality prop} and \ref{kd graded ren}, i.e. if $f: X \rightarrow Y$ is smooth, then $df^\flat: f^* \Omega^1_Y[1] \rightarrow \Omega^1_X[1]$ is injective with cone perfect in degree -1, and the dual map is surjective.  If $f$ is proper, then it is schematic and in particular $\cH^{-1}(\mathrm{cone}(\cE^\bullet \rightarrow \cF^\bullet)) = 0$.  For (4), note that the map $dp$ factors $\bT_U[-1] \rightarrow \bT_X[-1] \times_X U \rightarrow \bT_X[-1]$.  The first map is lci, and the completion of $\bT_X[-1] \times_X U$ along $U$ is a nilthickening of $\bT_U[-1]$, and we apply \cite[Prop. 2.2.13]{CD}; the second map is smooth, so the analogous claim is true for it as well.
\end{proof}

\subsubsection{} In \cite[Thm. 6.6]{loops and conns}, descent for quasi-coherent sheaves on formal odd tangent bundles was established, i.e. for an Artin 1-stack with affine diagonal $X$ and atlas $p: U \rightarrow X$ with Cech diagram $U_\bullet$, the canonical pullback functor
$$\begin{tikzcd} \QCoh(\wh{\bT}_X[\ng1]) \arrow[r, "\simeq"] & \Tot( \QCoh(\wh{\bT}_{U_\bullet}[\ng1])) = \Tot(\QCoh(\bT_{U_\bullet}[\ng1])) \end{tikzcd}$$
is an equivalence.  We have a similar result for ind-coherent sheaves.
\begin{prop}\label{odd tangent indcoh descent}
In the set-up above, the canonical $!$-pullback functor induces an equivalence
$$\begin{tikzcd} \IndCoh(\wh{\bT}_X[\ng1]) \arrow[r, "\simeq"] & \Tot( \IndCoh(\wh{\bT}_{U_\bullet}[\ng1])) = \Tot(\IndCoh(\bT_{U_\bullet}[\ng1])) \end{tikzcd}$$
and likewise for $\bG_m$, $B\bG_a$, $B\bG_a^{\gr}$-equivariant categories.
\end{prop}
\begin{proof}
First, note that the formation of formal odd tangent bundles commutes with pullbacks.  Next, note that if $U \rightarrow X$ is a smooth surjective map, then the composition $\wh{\bT}_U[\ng1] \rightarrow \wh{\bT}_X \times_X U \rightarrow \wh{\bT}_X[\ng1]$ is the composition of a proper map surjective on geometric points with a smooth map, and in particular is an h-cover.  The claim then follows from \cite[Thm. 8.2.2]{indcoh}.  The equivariant claim follows since equivariance is imposed by a limit.
\end{proof}

\section{Filtered Koszul duality and \texorpdfstring{$B\bG_a$}{BGa}-equivariant sheaves}

In this section we extended linear Koszul duality, which lives over $B\bG_m$, to a ``filtered'' Koszul duality living over $\bA^1/\bG_m$, matching $\QCoh(B\bG_a)$-equivariance structures with the deformation quantization of $\QCoh(\bT^*_X)$ to $\cD$-modules.  In Section \ref{sec bga} we review some basic notions regarding group actions on categories, and address an important technical point: when we equivariantize the category $\IndCoh(\wh{\bT}_X[\ng1])$ with respect to $B\bG_a$, we must do so at the level of small categories (or equivalently, we consider a renormalized large category) in order to have an interesting Tate construction.  In Section \ref{sec omega}, we review an explicit model for the category of $B\bG_a$-equivariant sheaves on $\bT_X[\ng1]$ in the case where $X$ is a scheme in terms of ``mixed'' $\Omega$-complexes.  In Section \ref{sec dmod} we then introduce the category of filtered $\cD$-modules on a scheme.  In Section \ref{kd sec} we prove the desired Koszul duality results.  Finally, we work out the example of $BG$ explicitly in Section \ref{dmod BG exmp}.

\subsection{\texorpdfstring{$B\bG_a$}{BGa}-actions on derived vector bundles}\label{sec bga}

We now introduce the notion of $B\bG_a$-actions and the associated \emph{Tate construction}.  Because there are technical issues with formulating this construction at the level of large categories, we prefer to work in the setting of small categories (though one may always ind-complete these small categories after equivariantization, i.e. renormalize).  We discuss these technical issues informally before introducing explicit constructions; for a more detailed discussion see \cite{ihes}.

\subsubsection{} The stacks $B\bG_a$ and $B^2\bG_a$ are examples of coaffine stacks \cite{lurie coaffine}.\footnote{We will prefer the term coaffine stacks to the term affine stacks used in \cite{toen affine} to avoid confusion with the notion of 1-affineness below.}  We will, once and for all, choose coordinates for their rings of global functions
$$\cO(B\bG_a) \simeq k[\eta], \;\;\; |\eta| = 1 \;\;\;\;\;\;\;\;\;\;\;\;\; \cO(B^2\bG_a) \simeq k[u], \;\;\; |u|=2.$$
As discussed in Section 3 of \cite{loops and conns} (see also the introduction of \cite{MRT}), the stack $B\bG_a$ may be realized as the affinization of the circle $S^1$.  Furthermore, we may assign $\eta$ and $u$ a weight-grading of 1, defining $\bG_m$-actions on $B\bG_a$ and $B^2\bG_a$.  
\begin{defn}
We define the \emph{graded 1-shifted} and \emph{2-shifted affine lines} by
$$\BGA := B\bG_a \rtimes \bG_m, \;\;\;\;\;\;\;\;\;\; B^2\bG_a^{\gr} := B^2\bG_a \rtimes \bG_m = B(\BGA).$$
\end{defn}

\subsubsection{}  We make some general remarks on group actions on categories; see \cite{1affine, beraldo} for precise statements and details.  Let $G$ be a group object in prestacks.   A priori, there are a few notions of what one might mean by categories with a $G$-action.
\begin{enumerate}
\item The category $\QCoh(G)$ is a comonoidal category under pullback along the group multiplication and pullback along the inclusion of the unit.\footnote{It is also a monoidal category under the pullback along the diagonal and projection to a point, making it a bialgebra (in fact, Hopf) category.  This bialgebra structure is used to define the invariants operation, but we will suppress it throughout the discussion.}  Then, we may consider $\QCoh(G)$-comodule categories.  If $\QCoh(G)$ is dualizable, we may equivalently consider $\QCoh(G)^\vee$-module categories; if $\QCoh(G)^\vee$ is semi-rigid, then $\QCoh(G)$ is automatically self-dual and the monoidal category $\QCoh(G)$ arises from the comonoidal category by passing to right adjoints.  For details, see \cite[\textsection 3.2]{ctcg}\cite[\textsection D.3]{1affine}.
\item The category $\QCoh(BG)$ is a monoidal category under tensor product, and we may consider $\QCoh(BG)$-module categories.
\end{enumerate}
When the stack $BG$ is 1-affine, the two notions are equivalent and related by the equivariantization and de-equivariantization functors:
$$\begin{tikzcd}[column sep=30ex]
 \cat{Mod}(\QCoh(BG)) \arrow[r ,shift left=0.9ex, "{\cat{D} \mapsto \cat{D} \otimes_{\QCoh(BG)} \cat{Vect}_k}"] & \cat{Comod}(\QCoh(G)). \arrow[l, shift left=0.9ex, "\cat{C} \mapsto \cat{C}^{\QCoh(G)}", "\simeq"']
\end{tikzcd}$$
Typical examples of group prestacks $G$ for which $BG$ is 1-affine are when $G$ is an affine algebraic group.  Taking our example of interested $G = B\bG_a$ (or $G = B\bG_a^{\gr}$), we note that by \cite[Thm. 2.5.7]{1affine}, the stacks $B^2\bG_a$ and $B^2\bG_a^{\mathrm{gr}}$ are 1-affine, thus we have a (de-)equivariantization correspondence.

\subsubsection{}\label{def B} In light of this we now give an explicit description of the monoidal categories $\QCoh(B^2\bG_a)$ and $\QCoh(B^2\bG_a^{\gr})$.  We let $V^* := \cO(\bG_a)$, noting that $\cO(B^2\bG_a) = \Sym_k V^*[-2] \simeq k[u]$, and let $\mathbb{B} := \{0\} \times_V \{0\}$.  The category $\Coh(\mathbb{B})$ is monoidal category under convolution; explicitly, the category $\QCoh(\mathbb{B})$ is equivalent to the category of \emph{mixed complexes}, i.e. chain complexes $W$ with a degree $-1$ chain map $\epsilon: W \rightarrow W[-1]$ such that $\epsilon^2 = 0$, and the convolution of two mixed complexes is given by
$$(W_1, \epsilon_1) \star (W_2, \epsilon_2) := (W_1 \otimes_k W_2, \epsilon_1 \otimes \mathrm{id} + \mathrm{id} \otimes \epsilon_2).$$
The following proposition is a combination of the results in \cite[\textsection 3.4]{loops and conns} and \cite[\textsection 3]{MF}.
\begin{prop}\label{bga}
Let $q: \mathbb{B} \rightarrow \Spec k$ denote the projection and $p: \Spec k \rightarrow B^2\bG_a$ the standard atlas.  We let $\Mod_{u\ng\mathrm{nil}}(k[u])$ be the category of locally $u$-nilpotent complexes.  We have commuting monoidal equivalences and functors
$$\begin{tikzcd}[column sep=15ex]
& \cat{Vect}_k &   \\
\QCoh(B^2\bG_a) \arrow[ur, "p^*"] \arrow[r, "\simeq"] & \QCoh(\mathbb{B}) \arrow[r, "\simeq", "- \otimes_{\cO(\bB)} k"'] \arrow[u, "q_*"] & \Mod_{u\ng\mathrm{nil}}(\cO(B^2\bG_a)).  \arrow[ul, "{R\Hom_{\cO(B^2\bG_a)}(k, -)}"'] 
\end{tikzcd}$$
Furthermore, $\Perf(B^2\bG_a)$ is identified with the full subcategory $\Coh(\bB) \subset \QCoh(\bB)$, and the full subcategory of $\Mod_{u\ng\mathrm{nil}}(\cO(B^2\bG_a))$ consisting of modules $M$ such that $R\Hom_{k[u]}(k, M)$ has finite cohomology.\footnote{For example, the module $M = k[u, u^{-1}]/k[u]$ is such a module which is not finitely generated.}  

In the graded setting, let $i: \{0\}/\bG_m \hookrightarrow \wh{\bA}^1/\bG_m$ be the inclusion; we have commuting monoidal equivalences and functors
$$\begin{tikzcd}[column sep=15ex]
& \Rep(\bG_m)& \\
\QCoh(B^2\bG_a^{\gr}) \arrow[ur, "p^*"] \arrow[r, "\simeq"]& \QCoh(\mathbb{B}/\bG_m) \arrow[r, "\simeq", "{(- \otimes_{\cO(\bB)} k)^{\unshear}}"'] \arrow[u, "q_*"] &  \IndCoh(\wh{\bA}^1/\bG_m).\arrow[ul, "{\shear \circ i^!}"']
\end{tikzcd}$$
We have similar identifications of the full subcategory of perfect complexes.  In particular, perfect complexes on $B^2\bG_a$ and  $B^2\bG_a^{\gr}$ may fail to be compact.
\end{prop}

The relationship between $\cO(\bB)$-modules and $\cO(B^2\bG_a)$-modules is essentially given by the standard Koszul duality calculation $\End_{\cO(B^2\bG_a)}(k, k) \simeq \cO(\bB)$.  Dually, we have $\cO(B^2\bG_a) \simeq \End_{\cO(\bB)}(k, k).$  Note that we have not defined the monoidal structure on $\Mod_{u-\mathrm{nil}}(\cO(B^2\bG_a))$ nor $\IndCoh(\bA^1/\bG_m)$; it is \emph{not} the usual tensor product but a $!$-analogue.  We will not address this since, as we will soon see, these are not the categories that we wish to work with.

\subsubsection{} An important construction in the setting of mixed complexes is the \emph{Tate construction}.  For example, the Hochschild homology of a category is naturally a mixed complex whose negative cyclic homology is obtained by taking $S^1$-invariants (or, in our setting, $\cO(\bB)$-invariants), and its periodic cyclic homology is obtained inverting $u$ (see \cite{Ch} for details).  The Tate construction is the composition of these two processes.

However, in Proposition \ref{bga} all $k[u]$-modules are $u$-torsion, so inverting $u$ for $\QCoh(B^2\bG_a)$-module categories always yields the zero category.  Roughly speaking, this is because we passed from $\QCoh(\bB)$ to $\Mod(\cO(B^2\bG_a))$ via the coinvariants functor $- \otimes_{\cO_{\bB}} k$, which is right t-exact and produces $u$-nilpotent modules.  The invariants functor $\Hom_{\cO_{\bB}}(k, -)$ is left t-exact, but it is not continuous since the augmentation module $k \in \QCoh(\bB)$ is not compact; to make the invariants functor continuous we require a renormalization, i.e. we replace $\QCoh(\bB)$ with $\IndCoh(\bB)$, or equivalently replace the compact objects $\Perf(\bB)$ with $\Coh(\bB)$.

\medskip

This small category $\Coh(\bB)$ appears naturally as the $B\bG_a$-invariants of the trivial action on the small category $\Perf(\Spec k)$.  In particular, if we take $B\bG_a$-invariants of the small category $\Perf(\Spec k)$, then ind-complete, we obtain $\IndCoh(\bB)$; if we reverse the order, we obtain $\QCoh(\bB)$.  This suggests to ``fix'' the Tate construction by taking invariants of small categories.  An alternative ``fix'' was explored in \cite{toly indcoh} by renormalizing using t-structures.

\subsubsection{} Given a group prestack $G$ acting on $X$, it is often unclear whether $\QCoh(G)$ is compactly generated or what the compact objects are, and the monoidal structure often does not restrict to small subcategories.  However, it is tautological that $\Perf(G) \subset \QCoh(G)$ is a small comonoidal subcategory coacting on $\Perf(X)$.  Let us suppose for notational simplicity that the group structure on $G$ is given by multiplication and unit maps $m: G \times G \rightarrow G$ and $e: \Spec k \rightarrow G$, and that the action is given by a map $a: G \times X \rightarrow X$, and projection $p: G \times X \rightarrow X$.  Then, may realize the invariants as a limit, and by descent obtain identifications of the $\Perf(G)$-invariants of $\Perf(X)$:
$$\begin{tikzcd}
\QCoh(X/G) \simeq \QCoh(X)^{\QCoh(G)} \arrow[r] & \QCoh(X) \arrow[r, shift right=1.2ex, "a^*"] \arrow[r, shift left=1.2ex, "p^*"] & \QCoh(G \times X) \arrow[r, "m^*"] \arrow[r, shift left=2.4ex, "p^*"] \arrow[r, shift right=2.4ex, "a^*"] & \QCoh(G \times G \times X) \cdots \\
\Perf(X/G) \simeq \Perf(X)^{\Perf(G)} \arrow[u, hook] \arrow[r] & \Perf(X) \arrow[r, shift right=1.2ex, "a^*"] \arrow[r, shift left=1.2ex, "p^*"] \arrow[u, hook]& \Perf(G \times X) \arrow[r, "m^*"] \arrow[r, shift left=2.4ex, "p^*"] \arrow[r, shift right=2.4ex, "a^*"] \arrow[u, hook]& \Perf(G \times G \times X) \cdots
\end{tikzcd}$$
In particular, if $X/G$ is a \emph{perfect stack} (i.e. compactly generated by perfect complexes), then we may commute the operations of ind-completion and $\Perf(G)$-invariants.  Since $X/G$ is a perfect stack when $G$ is affine algebraic and $X$ is quasiprojective \cite{BFN}, in the traditional setting of affine algebraic group actions on reasonable schemes, we can ignore the issue of large vs. small categories when taking invariants.  However, as we observed in Proposition \ref{bga}, $B^2\bG_a$ is not perfect, which manifests in the difference between $\IndCoh(\bB)$ and $\QCoh(\bB)$.

\subsubsection{} We now give a precise definition of categorical invariants.
\begin{defn}\label{ren G inv}
Let $G$ be a group object in prestacks, and $\cat{C}_0 \in \cat{St}_k$ a small comodule category for $\Perf(G)$.  We define the $\Perf(G)$-invariants to be the limit (i.e. totalization)
$$\cat{C}_0^G := \lim \left( \begin{tikzcd}[column sep=huge] \cat{C}_0 \arrow[r, shift right=1.2ex, "\gamma"] \arrow[r, shift left=1.2ex, "p^* \otimes \mathrm{id}_{\cat{C}_0}"] & \Perf(G) \otimes \cat{C}_0 \arrow[r, "m^*"] \arrow[r, shift left=2.4ex, "p^* \otimes \mathrm{id}_{\Perf(G) \otimes \cat{C}_0}"] \arrow[r, shift right=2.4ex, "\mathrm{id}_{\Perf(G)} \otimes \gamma"] & \Perf(G \times G) \otimes \cat{C}_0  \arrow[r, shift right=1.5ex] \arrow[r, shift right=.5ex] \arrow[r, shift left=.5ex] \arrow[r, shift left=1.5ex] & \cdots \end{tikzcd} \right).$$
This category is naturally a module category for $\Perf(k)^G = \Perf(BG)$.  Taking $\cat{C} := \Ind(\cat{C}_0)$, we define the \emph{category of compactly renormalized $G$-invariants} by 
$$\cat{C}^{\omega G} := \Ind(\cat{C}_0^G) = \Ind((\cat{C}^\omega)^G)$$
which is naturally a module category for $\cat{Vect}_k^{\omega G} = \Ind(\Perf(BG))$.  This category comes equipped with a \emph{forgetful functor} $\cat{C}^{\omega G} \rightarrow \cat{C}$ which is the unique colimit-preserving functor which restricts to the usual forgetful functor $\cat{C}_0^G \rightarrow \cat{C}$ on compact objects.  These constructions are evidently functorial for $\Perf(G)$-equivariant functors on small categories.
\end{defn}

\subsubsection{} 
We now specialize to the case $G = B\bG_a$, again fixing a degree 2 polynomial generator $u \in \cO(B^2\bG_a)$.  Here, $\Perf(k)^{B\G_a} \simeq \Perf(\cO(B^2\bG_a))$, and therefore $\Ind(\Perf(B\bG_a)) \simeq \Mod(\cO(B^2\bG_a))$.  In particular, $\cat{C}^{\omega B\bG_a}$ is a $\Mod(\cO(B^2\bG_a))$-module category.  We will also consider the graded setting $G = B\bG_a^{\gr}$, where $\cat{C}^{\omega B\bG_a^{\gr}, \unshear}$ is a $\Mod(\cO(B^2\bG_a))^{\bG_m, \unshear} \simeq \QCoh(\bA^1/\bG_m)$-module category after unshearing.

\begin{defn}\label{def tate}
We let $k(u) := k[u, u^{-1}]$, i.e. the graded fraction field of $k[u] = \cO(B^2\bG_a)$, and define $\cat{Vect}^{\Tate} := \Mod(k(u))$.  Let $\cat{C}_0$ be a $\Perf(B\bG_a)$-comodule category, and take $\cat{C} = \Ind(\cat{C}_0)$.  We define the \emph{Tate construction} by
$$\cat{C}^{\Tate} := \cat{C}^{\omega B\bG_a} \otimes_{\cat{Vect}^{\omega B\bG_a}} \cat{Vect}^{\Tate} \simeq \cat{C}^{\omega B\bG_a} \otimes_{\Mod(k[u])} \Mod(k(u)).$$
The resulting category is a compactly generated $\QCoh(k(u))$-module category whose monoidal structure restricts to compact objects, i.e. a compactly generated 2-periodic category.  We define the \emph{graded Tate construction} by
$$\cat{C}^{\grTate} := \cat{C}^{\omega B\bG_a^{\gr}} \otimes_{\QCoh(\bA^1/\bG_m)} \QCoh(\bG_m/\bG_m)$$
which is a compactly generated $k$-linear category (i.e. without periodicity).  These constructions are evidently functorial.
\end{defn}

We have the following renormalized version of Proposition \ref{bga} which we state for comparison, keeping the same notation, and recalling that the monoidal structure on $\IndCoh(\bB)$ is by convolution and on $\QCoh(\cO(B^2\bG_a))$ is now by the usual tensor product.
\begin{prop}\label{ren bga}
We have commuting monoidal equivalences
$$\begin{tikzcd}[column sep=15ex]
& \cat{Vect}_k &   \\
\cat{Vect}^{\omega B\bG_a} \arrow[ur] \arrow[r, "\simeq"] & \IndCoh(\mathbb{B}) \arrow[r, "\simeq", "{R\Hom_{\bB}(k, -)}"'] \arrow[u, "q_*"] & \Mod(\cO(B^2\bG_a)),  \arrow[ul, "{- \otimes_{\cO(B^2\bG_a)} k}"'] 
\end{tikzcd}$$
$$\begin{tikzcd}[column sep=15ex]
& \Rep(\bG_m)& \\
\cat{Vect}^{\omega B\bG_a^{\gr}} \arrow[ur] \arrow[r, "\simeq"]& \IndCoh(\mathbb{B}/\bG_m) \arrow[r, "\simeq", "{R\Hom_{\bB}(k, -)^{\unshear}}"'] \arrow[u, "q_*"] &  \QCoh(\bA^1/\bG_m).\arrow[ul, "{\shear \circ i^*}"']
\end{tikzcd}$$
\end{prop}

\subsection{Mixed \texorpdfstring{$\Omega$}{Omega}-complexes}\label{sec omega}

In this section we introduce one of the categories on the two sides of Koszul duality: the category of $B\bG_a$-equivariant sheaves on the odd tangent bundle $\bT_X[\ng1]$.  We begin with some generalities.

\subsubsection{}\label{bimod alg} To formulate the notion of modules for $\cD_X$ and other sheaves of non-commutative algebras, we will use the following $\infty$-categorical framework.  We consider the category $\QCoh(X \times X)$ as a monoidal $\infty$-category under $*$-convolution (denoted $\star$), acting on the module category $\QCoh(X)$; see \cite[\textsection 3]{CD} for a discussion.  More generally, we take as shorthand notation for the graded and filtered variants
$$\QCoh(X \times X)^{\gr} := \QCoh((X \times B\bG_m) \times_{B\bG_m} (X \times B\bG_m)),$$
$$\QCoh(X \times X)^{\fil} := \QCoh((X \times \bA^1/\bG_m) \times_{\bA^1/\bG_m} (X \times \bA^1/\bG_m)),$$
to which all the following discussion also applies.

\medskip

We will consider algebra objects $\cA \in \QCoh(X \times X)$ which may be viewed roughly as objects $\cA_X \in \QCoh(X \times X)$ with a multiplication $\cA_X \star \cA_X \rightarrow \cA_X$ and unit $\Delta_*\cO_X \rightarrow \cA_X$ maps; for precise definitions, we refer the reader to the discussion in \cite[\textsection A.2]{CD} (which itself is a summary of results in \cite{HA}).  If $\cA_X$ is supported along the diagonal of $X$, we say $\cA_X \in \QCoh_\Delta(X \times X)$ is a \emph{$\cO_X$-bimodule algebra}, since convolution in this case is just the usual tensor product of bimodules.

\begin{defn}
Let $X$ be a perfect stack, and $\cA_X$ an algebra object in $(\QCoh(X \times X), \star)$.  The \emph{category of $\cA_X$-modules} $\Mod(\cA_X)$ is the $\infty$-category of $\cA_X$-module objects in the $\QCoh(X \times X)$-module category $\QCoh(X)$.  Given another $\cB_Y$, the \emph{category of $(\cA_X, \cB_Y)$-bimodules} $\Mod(\cA_X, \cB_Y)$ is the category of $\cA_X \boxtimes \cB_Y$-module objects in the $\QCoh(X^2 \times Y^2)$-module category $\QCoh(X \times Y)$.
\end{defn}

There is a canonical functor sending $(\cA_X, \cB_Y)$-bimodules to functors:
\begin{equation}\label{bimod eqn}
\Mod(\cA_X, \cB_Y) \longrightarrow \cat{Fun}_k^L(\Mod(\cA_X), \Mod(\cB_X)),
\end{equation}
$$\cK \longmapsto - \otimes_{\cA_X} \cK := \colim\left(\begin{tikzcd}[column sep=5ex] \cdots \arrow[r, shift left=.75ex] \arrow[r, shift left=.25ex] \arrow[r, shift right=.25ex] \arrow[r, shift right=.75ex] & - \star \cA_X \star \cA_X \star \cK \arrow[r, shift left=.5ex] \arrow[r] \arrow[r, shift right=.5ex] & - \star \cA_X \star \cK \arrow[r, shift left = .25ex] \arrow[r, shift right=.25ex] & - \star \cK \end{tikzcd}\right)$$
where $\star$ indicates convolution of objects and the maps are given by the respective module structures.  This assignment is functorial, i.e. there is a canonical equivalence $(- \otimes_{\cA_X} \cK) \otimes_{\cB_Y} \cL \simeq - \otimes_{\cA_X} (\cK \otimes_{\cB_Y} \cL).$  The following Proposition is an consequence of \cite[Prop. 4.1]{BFN}.
\begin{prop}\label{bimod prop}
The $\infty$-category $\Mod(\cA_X)$ is dualizable with dual $\Mod(\cA_X^{\opp})$, and $\Mod(\cA_X) \otimes \Mod(\cB_X) \simeq \Mod(\cA_X \boxtimes \cB_Y)$.  In particular, the functor (\ref{bimod eqn}) is a monoidal equivalence.
\end{prop}

\medskip

In the special case where we have a map of algebra objects $\cA_X \rightarrow \cB_X$ in $\QCoh(X \times X)$, we may view $\cB_X$ as an $(\cA_X, \cB_X)$-bimodule object and as a $(\cB_X, \cA_X)$-bimodule object.  In this case, it defines two functors, which we call the \emph{induction} and \emph{restriction} functors respectively:
$$\ind_{\cA_X}^{\cB_X}: \Mod(\cA_X) \longrightarrow \Mod(\cB_X), \;\;\;\;\;\;\;\;\; \res_{\cA_X}^{\cB_X}: \Mod(\cB_X) \longrightarrow \Mod(\cA_X).$$
By standard arguments splitting simplicial objects, one can check that the underlying functor $\res_{\cA_X}^{\cB_X}: \QCoh(X) \rightarrow \QCoh(X)$ is equivalent to the identity functor.  Furthermore, the functors $(\ind_{\cA_X}^{\cB_X}, \res_{\cA_X}^{\cB_X})$ are adjoint,  since $\cB_X$ is self right-dual and by abstract nonsense the right adjoint to the functor defined by a bimodule is given by its right dual (if it exists).

\medskip

In particular, the map $\Delta_* \cO_X \rightarrow \cA_X$ defined by the identity element is a map of $\cO_X$-bimodule algebras, and gives rise to an adjoint pair of induction and restriction functors
$$\ind^{\cA_X}_{\cO_X}: \begin{tikzcd} \QCoh(X) \arrow[r, shift left] & \Mod(\cA_X) \arrow[l, shift left] \end{tikzcd} :\res^{\cA_X}_{\cO_X}.$$
Since it is a right adjoint, there is a t-structure on $\Mod(\cA_X)$ uniquely defined by the property that $\res^{\cA}_{\cO_X}$ is left t-exact; furthermore, it follows that $\ind^{\cA_X}_{\cO_X}$ is right t-exact.  In general, the inductions $\ind_{\cA_X}^{\cB_X}$ are right t-exact and restrictions $\res_{\cA_X}^{\cB_X}$ left t-exact for these t-structures.

\subsubsection{}  We now discuss $\Omega$-complexes.  Let $X$ be a smooth stack.  We define a $B\bG_a \rtimes \bG_m$-action on $\bT_X[\ng1]$ as follows.  First, we note that the unipotent loop space $\cL^u X := \Map(B\bG_a, X)$ has a canonical $B\bG_a$-action, and $S^1$-equivariant maps $X \rightarrow \cL^u X \rightarrow \cL X$.  We claim the following (c.f. \cite[Prop. 2.1.24]{Ch}).
\begin{prop}
Let $X$ be a prestack admitting a cotangent complex, and let $i: \cL^uX \rightarrow \cL X$ denote the map induced by the affinization of $S^1$.  Then, $\cL^uX$ and $\cL X$ admit cotangent complexes and $i^* \Omega^1_{\cL X} \rightarrow \Omega^1_{\cL^u X}$ is an equivalence.  In particular $\Omega^1_{X/\cL X} \rightarrow \Omega^1_{X/\cL^u X}$ is an equivalence, and $\wh{\bT}_X[\ng1]$ is both the completed normal bundle for both $X \subset \cL^u X$ and $X \subset \cL X$, thus admits a canonical $B\bG_a^{\gr} = B\bG_a \rtimes \bG_m$-action compatible with the $B\bG_a$-action on $\cL^u X$ and the $S^1$-action on $\cL X$.
\end{prop}
\begin{proof}
For an affine derived scheme $S$, and let $a_S: S^1 \times S \rightarrow B\bG_a \times S$ be the affinization map \cite[Lem. 3.13]{loops and conns}.  We will apply \cite[Prop. 5.1.10]{HP}; consider the diagram
$$\begin{tikzcd}
S \arrow[d, equals] & S^1 \times S \arrow[l, "p_S"'] \arrow[d, "a_S"] \arrow[r, "f \circ a_S"] & X \arrow[d, equals] \\
S & \arrow[l, "q_S"']  B\bG_a \times S \arrow[r, "f"] & X
\end{tikzcd}$$
where $f$ is the map defining a point $\eta \in \cL^uX(S)$ and $f \circ a_S$ is the map defining its image in $\cL{X}(S)$.  Letting $i: \cL^u X \rightarrow \cL X$ denote the map induced by the affinization, we want to show that the map $\eta^* i^* \Omega^1_{\cL X} \rightarrow \eta^* \Omega^1_{\cL^u X}$ is an equivalence.  By \emph{loc. cit.} we have an identification of this map with the counit for the adjunction
$$p_{S+} a_S^* f^* \Omega^1_X \simeq q_{S+} a_{S+} a_S^* f^* \Omega^1_X \rightarrow q_{S+}f^* \Omega^1_X.$$
To see that this counit is an equivalence we pass through Cartier duality, i.e. the equivalences $\QCoh(S^1) = \QCoh(B\bZ) \simeq \QCoh(\bG_m)$ and $\QCoh(B\bG_a) \simeq \QCoh(\wh{\bG}_a)$, under which the adjoint pair $(a_+, a^*)$ becomes the adjunction $\begin{tikzcd} \wh{\iota}^*: \QCoh(\bG_m) \arrow[r, shift left] & \QCoh(\wh{\bG}_a): \wh{\iota}_* \arrow[l, shift left] \end{tikzcd}$, where it is standard that $\wh{\iota}^*\wh{\iota}_* \simeq \mathrm{id}_{\QCoh(\wh{\bG}_a)}$ for the inclusion of a formal neighborhood.
\end{proof}

\subsubsection{} Having defined the $B\bG_a$-action, we can define our category of interest: 
$$\IndCoh(\wh{\bT}_X[\ng1])^{\omega B\bG_a^{\gr}} := \Ind(\Coh(\wh{\bT}_X)^{B\bG_a^{\gr}}).$$
While it is possible to give an explicit description of this renormalization via coderived categories (see \cite{coderived} and \cite[\textsection H.3]{AG}), we will view it as the ind-completion of the full subcategory
$$\Coh(\wh{\bT}_X)^{B\bG_a^{\gr}} \subset \QCoh(\wh{\bT}_X)^{B\bG_a^{\gr}}.$$

\subsubsection{}  We now turn our attention to the case when $X$ is a smooth scheme, and in particular $\wh{\bT}_X[\ng1] = \bT_X[\ng1]$.  In this case, we may identify the $B\bG_a$-action on $\bT_X[\ng1]$ explicitly.  We begin by recalling from \cite{loops and conns} a general method by which one can characterize $B\bG_a^{\gr}$-actions on certain derived schemes satisfying a certain degree-weight parity using purely 1-categorical structures.
The following proposition is established in \cite[Prop. 4.7, Lem. 4.8]{loops and conns}; we summarize the argument for convenience.
\begin{prop}\label{S1 action on shift}
Let $X$ be a derived scheme and $p: \bE_X[\ng1] \rightarrow X$ a derived vector bundle which is perfect in degree $-1$ with shifted locally free sheaf of sections $\cE^\vee[1]$.  The space of linear $\BGA$-actions on $\bE_X[\ng1]$ compatible with the contracting $\bG_m$-action on the fibers is discrete and equivalent to the space of $k$-linear $\cE^\vee$-valued derivations $d: \cO_X \rightarrow \cE^\vee$.
\end{prop}
\begin{proof}
We identify $\cO(B\bG_a) \simeq k[\eta]$.  First, note that a $B\bG_a$-action on a derived scheme is Zariski local, so we may assume that $X$ is affine.  Such an action gives rise to an action map which may be given explicitly by a $k$-linear map
$$\cO_{\bE_X[\ng1]} = \Symp_X \cE^\vee[1] \rightarrow \cO_{\bE_X} \otimes_k k[\eta] = (\Symp_X \cE^\vee[1]) \oplus (\Symp_X \cE^\vee[1])[-1]\langle 1 \rangle$$
satisfying the usual counit and comultiplication identities, i.e. a $k$-linear derivation $\cO_X \rightarrow \cE^\vee[\ng1]$, and higher coherences.  These objects all live in the heart of the 1-sheared t-structure (see Definition \ref{shear tstructure}), thus by \cite[App. A]{CD} $\infty$-module structures are determined by 1-categorical structures.
\end{proof}

By the characterization above, and since $\bT_X[\ng1]$ is the universal derived vector bundle in degree 1 with a linear $\BGA$-action, the following is established in \cite[Prop. 4.4]{loops and conns}.
\begin{cor}
When $X$ is a smooth scheme, the $\BGA$-action on the odd tangent bundle $\bT_X[\ng1]$ encoded by the de Rham differential agrees with the canonical $\BGA$-action under the identification $\Map(B\bG_a, X) \simeq \bT_X[\ng1]$.
\end{cor}


\begin{rmk}
A more general statement without restrictions on the amplitude of $\bE_X$ is likely, following the theory of derived foliations in \cite[\textsection 1.2]{foliations 1}.  We only need this weaker statement.
\end{rmk}

\subsubsection{} Our next step is to give a concrete description of the category $\QCoh(\bT_X[\ng1])^{B\bG_a}$.  In the following we will define algebra objects as a strict 1-categorical algebras, but because they will be algebra objects in the heart of the 1-sheared t-structure of Definition \ref{shear tstructure}, by \cite[Prop. A.2.8]{CD} this defines an algebra object in the $\infty$-categorical sense uniquely up to contractible homotopy. 
\begin{defn}\label{omega dga}
Let $X$ be a smooth scheme and recall that by convention, $\Omega^1_X$ has internal weight grading $-1$.    The (coconnective) \emph{mixed de Rham algebra} $\dddot{\Omega}^{\bullet}_X$ is the $\mathbb{Z}$-graded sheaf of graded $\cO_X$-bimodule algebras on $X$ defined by
$$\dddot{\Omega}^{\bullet}_{X} = \Omega^{\bullet}_X\langle \delta\rangle  \big/ \langle \delta^2, [\delta, \omega] - d_{\dR}\omega \mid \omega \in  \Omega^\bullet_X \rangle \in \Alg(\QCoh_{\Delta}(X \times X)^{\gr}) $$
where the cohomological-weight bidegrees are $|\Omega^1_X| = (1, -1) = |\delta| = (1, -1)$.\footnote{Note that though we a priori take the $k$-linear span of the indicated relations, the relations are actually $\cO_{X \times X}$-linear, i.e. $y\delta(fx) - yf\delta(x) = yx\, df$.  Note also that the bracket is supercommutative, i.e. $[\delta, \omega] = \delta \omega - (-1)^{-|\omega|}\omega \delta$.}  There is a map of graded $\cO_X$-bimodule algebra objects $\Omega^{\bullet}_X \rightarrow \dddot{\Omega}^{\bullet}_X$.  We define a connective version $\dddot{\Omega}^{-\bullet}_X$ similarly by taking $|\Omega^1_X| = (-1, -1) = |\delta| = (-1, -1)$.
\end{defn}

\begin{rmk}
A quick calculation shows that $\dddot\Omega^\bullet_X$ is scheme-theoretically supported on the first infinitesimal neighborhood of the diagonal, i.e. affine locally $(1 \otimes f - f \otimes 1) \cdot \delta = df$ and $(1 \otimes f - f \otimes 1) \cdot \omega = 0$, while $\Omega^\bullet_X$ is scheme-theoretically supported on the diagonal itself (thus is an algebra object in $\QCoh(X)$ under the tensor product, i.e. an $\cO_X$-algebra).
\end{rmk}

\subsubsection{} Our starting point for studying $\BGA$-equivariant sheaves on the odd tangent bundle is the following identification \cite[Thm. 4.9]{loops and conns} of the non-renormalized category.
\begin{prop}\label{identify omega}
Let $X$ be a smooth scheme.  There is an equivalence of of $\QCoh(B\bG_m)$-module categories
$$\QCoh(\bT_X[\ng 1])^{B\bG_a^{\gr}} \simeq \Mod_{\QCoh(X)^{\gr}}(\dddot{\Omega}^{-\bullet}_X)$$
such that the restriction functor is compatible with the forgetful functor, i.e. the following diagram commutes:
$$\begin{tikzcd}
\QCoh(\bT_X[\ng 1])^{B\bG_a^{\gr}} \arrow[r, "\simeq"]  \arrow[d] & \Mod_{\QCoh(X)^{\gr}}(\dddot{\Omega}^{-\bullet}_X) \arrow[d, "\res_{\Omega^{-\bullet}_X}^{\dddot\Omega^{-\bullet}_X}"] \\
\QCoh(\bT_X[\ng 1])^{\bG_m} \arrow[r, "\simeq"] & \Mod_{\QCoh(X)^{\gr}}(\Omega^{-\bullet}_X).
\end{tikzcd}$$
Furthermore, for a map $f: X \rightarrow Y$ of smooth schemes the $B\bG_a^{\gr}$-equivariant pullback functor 
$$df^!: \QCoh(\bT_Y[\ng1])^{B\bG_a^{\gr}} \rightarrow \QCoh(\bT_X[\ng1])^{B\bG_a^{\gr}}$$ 
is identified with the functor defined by the $(\dddot\Omega_X^{-\bullet}, \dddot\Omega_Y^{-\bullet})$-bimodule $(\mathrm{id}_X \times f)_* \dddot\Omega^{-\bullet}_X$.
\end{prop}
\begin{proof}
The main claim is by the argument in \cite[Thm. 4.9]{loops and conns}, which only treats the affine case, but may be directly adapted to our formulation.  The formality argument using weights in \emph{loc. cit.} may be reformulated as the claim that the algebra $\dddot{\Omega}^{-\bullet}_X$ underlying the monad lives in the heart of the 1-sheared t-structure (see Definition \ref{shear tstructure} and the proof of Proposition \ref{S1 action on shift}), and thus its structure as an algebra object is determined by 1-categorical data.  Compatibility with forgetful functors is by construction.  The claim regarding pullbacks follows by Proposition \ref{bimod prop}; that is, evaluating at $\dddot\Omega_Y$ we see that the bimodule computing the pullback must be in the heart of the 1-sheared t-structure, and we may compute that it is the $(\dddot\Omega_X^{-\bullet}, \dddot\Omega_Y^{-\bullet})$-bimodule $(\mathrm{id}_X \times f)_* \dddot\Omega^{-\bullet}_X$ Zariski locally on $Y$ and using that $df^! \simeq df^*$.
\end{proof}

\begin{rmk} There is a third sheaf of algebras and corresponding category from \cite{BG, rybakov} that one can consider, which corresponds to the category $\Coh(\bT_X[\ng1])^{\grTate}$.  The \emph{de Rham algebra} is the dg sheaf of dg-commutative $k$-algebras $\Omega^\bullet_{X, d} = (\Omega^\bullet_X, d)$ underlying the de Rham complex.  While the sheaves themselves are $\cO_X$-quasicoherent, the differentials are only $k$-linear, i.e. $\Omega^\bullet_{X, d}$ is not an object of $\QCoh(X \times X)$.  In \emph{loc. cit.} it is understood as an algebra object in the category of $k$-linear (non-quasicoherent) sheaves on $X$, and its category of modules consists of complexes with quasi-coherent terms whose differentials are not required to be $\cO_X$-linear (and may fail to have quasi-coherent cohomology).   
We will not consider this category explicitly, instead appealing to its Koszul dual realization.
\end{rmk}

\begin{rmk}
When $X$ is not smooth, one may interpret $\Omega^{-\bullet}_X$ as a mixed complex via the explicit dg model in the category of \emph{quasi-coherent crystals over a derived foliation} \cite[Def. 2.1.2]{foliations 2}.
\end{rmk}

\subsection{Filtered \texorpdfstring{$\cD$}{D}-modules}\label{sec dmod}

We now turn our attention to filtered $\cD$-modules on the Koszul dual side, which we realize explicitly via the Rees construction.  The following notions are well-known; detailed discussion may be found in \cite{MHM}.  We reproduce the main definitions and constructions for the reader's convenience, and to establish our conventions.  We let $t$ denote the Rees parameter with weight 1.
\begin{defn}\label{filt dmod schemes}
Let $X$ be a smooth scheme.  The sheaf of differential operators $\cD_X$ is an algebra object of $\QCoh(X \times X)$ supported along the diagonal,\footnote{But unlike $\dddot\Omega^\bullet_X$, not any finite infinitesimal neighborhood.} and its Rees algebra for the order filtration is likewise an algebra object of $\QCoh(X \times X)^{\fil}$ supported along the diagonal, i.e. a $k[t]$-linear $\cO_X$-bimodule algebra
$$\wt{\cD}_X := \Rees(\cD_X) = \bigoplus_{k \geq 0} \cD_X^{\leq k} t^k.$$ 
The category of \emph{filtered right $\cD$-modules} 
$$F{\cD}^r(X) = \Mod_{\QCoh(X)^{\fil}}^r(\wt{\cD}_X)$$
is the category of right $\wt{\cD}_X$-module objects in the $\QCoh(X \times X)^{\fil}$-module category $\QCoh(X)^\fil$.  We may define the category of left modules  $F\cD^\ell(X)$ similarly.  
\end{defn}

\subsubsection{} The map of $\cO_X$-bimodule algebras $\wt{\cO}_X \rihgtarrow \wt{\cD}_X$ gives rise to t-exact and adjoint \emph{induction} and \emph{restriction functors}
$$\ind: \begin{tikzcd} \QCoh(X)^{\fil} \arrow[r, shift left] & \arrow[l, shift left]  F\cD(X)\end{tikzcd} :\res.$$
By definition, the category $F\cD(X)$ is a $\QCoh(\bA^1/\bG_m)$-module category, where $\bA^1 = \Spec k[t]$.  We have identifications of the insertion functors for the generic and special fibers:
$$\un: F\cD(X) \longrightarrow F{\cD}(X) \otimes_{\QCoh(\bA^1/\bG_m)} \QCoh(\bG_m/\bG_m) \simeq \cD(X),$$
$$\gr: F\cD(X) \longrightarrow  F{\cD}(X) \otimes_{\QCoh(\bA^1/\bG_m)} \QCoh(\{0\}/\bG_m) \simeq \QCoh(\bT^*_X/\bG_m)$$
compatible with the induction and forgetful functors.  This allows us to define the following small subcategory.

\begin{defn}
We define the \emph{coherent subcategory} $F\cD_c(X) \subset F\cD(X)$ to be the smallest stable idempotent-completed subcategory containing the image of $\Coh(X)^{\fil} = \Perf(X)^{\fil} \subset \QCoh(X)^\fil$ under the induction functor, which compactly generates $F\cD(X)$ by the argument in \cite[\textsection 5.1.17]{QCA}.
\end{defn}

\begin{rmk}
A filtered $\cD$-module $\cM \in \wt{\cD}(X)$ is coherent if and only if $\un(\cM) \in \cD(X)$ and $\gr(\cM) \in \QCoh(\bT^*_X)$ are coherent, if and only if $\gr(\cM)$ is coherent and its weights are bounded below.  In particular, applying the Rees construction to a $\cD_X$-module equipped with a compatible filtration gives rise to a coherent $\wt{\cD}_X$-module if and only if the filtration was good.
\end{rmk}

\subsubsection{}  The structure and canonical sheaves are naturally equipped with the structure of a filtered left (resp. right) $\cD$-module via
$$\wt{\cO}_X := \cO_X \otimes k[t] \in F{\cD}^\ell(X), \;\;\;\;\;\;\;\;\;\; \wt{\omega}_X := \Omega^{\dim X}_X \otimes k[t]t^{-\dim X}[-\dim X] \in F{\cD}^r(X).$$
These objects and their iterated convolutions with $\wt{\cD}_X$ are in shifts of the hearts of standard t-structures, thus their 1-categorical definition lifts to an $\infty$-categorical one up to contractible homotopy.  The \emph{standard t-structure} on $F\cD^\ell(X)$ is uniquely characterized by the property that the forgetful functor to $\QCoh(X)^\fil$ is left t-exact.  We define a \emph{standard t-structure} on $F\cD^r(X)$ by shifting the standard t-structure on $\QCoh(X)^\fil$ so that $\wt{\omega}_X$ is in the heart.  There are side-changing equivalences, which are t-exact for these t-structures:
$$\begin{tikzcd}[column sep=25ex]
F{\cD}^\ell(X) \arrow[r, "\simeq"', "{\wt{\omega}_X \otimes_{\wt{\cO}_X} -}", shift left=1ex] & F{\cD}^r(X). \arrow[l, shift left=1ex, "{\intHom_{\wt{\cO}_X}(\wt{\omega}_X, -)}"]
\end{tikzcd}$$
To formulate this in the $\infty$-categorical setting, we note the sheaf 
$$(\wt{\omega}_X \boxtimes \wt{\cO}_X) \otimes_{\wt{\cO}_{X \times X}}  \wt\cD_X \in \QCoh(X \times X)^\fil$$ 
has two commuting right $\wt{\cD}_X$-module structures defined by the usual 1-categorical formulas, which give rise to $\infty$-categorical structures since these are module objects in the heart.  Likewise, the sheaf
$$\wt\cD_X \otimes_{\wt\cO_{X \times X}} (\wt\cO_X \boxtimes \wt\omega_X^{-1}) \in \QCoh(X \times X)^\fil$$ 
has two commuting left $\wt\cD_X$-module structures.

\subsubsection{} We define the pullback functor on $\cD$-modules in the usual way via a bimodule object.  This bimodule object lives in the heart, so by the usual methods we can define it using 1-categorical formulas; however, it is only t-exact for smooth morphisms.
\begin{defn}
Let $f: X \rightarrow Y$ be map of smooth schemes.  We define the (naive) pullback functor $f^\dagger: F\cD^\ell(Y) \rightarrow F\cD^\ell(X)$ by the $(\wt{\cD}_X, \wt{\cD}_Y)$-bimodule object in $\QCoh(X \times Y)^{\fil}$ given by
$$(f \times \mathrm{id}_Y)^*\wt{\cD}_Y \in \QCoh(X \times Y)^\fil$$
where we view $\wt{\cD}_Y \in \QCoh(Y \times Y)^\fil$.  This is a module object in the heart for iterated convolutions with both $\wt\cD_X$ and $\wt\cD_Y$, thus its bimodule structure may be defined 1-categorically.  The right $\wt\cD_Y$-module structure is tautological, and the left $\wt\cD_X$-module structure is defined by the usual formulas \cite[\textsection 8.6]{MHM}, i.e.
$$(f \times \mathrm{id}_Y)^*\wt{\cD}_Y \simeq \wt{\cO}_X \otimes_{f^{-1}\wt{\cO}_Y} f^{-1}\wt{\cD}_Y, \;\;\;\;\;\;\;\;\;\;\; \theta \cdot (x \otimes D) = \theta(x) \otimes D + x \otimes (f_*\theta \, D) \text{ for } \theta\in \wt{\Theta}_X.$$
The functor $f^\dagger$ is defined on right $\cD$-modules by side-changing.  We define
$$f^! := f^\dagger[\dim X - \dim Y].$$
When $f$ is smooth, the functor $f^!$ is t-exact when viewed as a functor on right $\cD$-modules by side-changing.  It is straightforward to verify that the pullback restricts to coherent complexes and commutes with the restriction functor.
\end{defn}

\subsubsection{} We now establish a basic functoriality and descent properties for the category of filtered $\cD$-modules.  We begin with the following elementary identification of $F\cD(X)$ as a derived category.
\begin{prop}\label{dmod heart}
Let $X$ be a smooth scheme.  The canonical map $\begin{tikzcd} D(F\cD(X)^\heartsuit) \arrow[r, "\simeq"] &  F\cD(X)\end{tikzcd}$ is an equivalence.
\end{prop}
\begin{proof}
The t-structure on $F\cD(X)$ is left t-complete, since the t-structure on $\QCoh(X)^{\gr}$ is t-complete and the forgetful functor is conservative, left t-exact, and commutes with limits.  By \cite[I.3 Lem. 2.4.5]{GR} we have an equivalence $D(F\cD(X)^\heartsuit)^+ \rightarrow F\cD(X)^+$, so it suffices to show that $D(F\cD(X)^\heartsuit)$ is left t-complete.  This assertion is Zariski local on $X$, so we may assume that $X$ is affine, where the claim follows since $F\cD(X)^\heartsuit$ has finite global dimension when $X$ is affine (e.g. using the Spencer complex \cite[Def. 8.4.3]{MHM}).
\end{proof}

Next, we define the family of categories $F\cD(X)$ with the $!$-pullback functorially in an $\infty$-categorical sense.
\begin{prop}\label{dmod pullback prop}
The assignment $X \mapsto F\cD(X)$ and $f \mapsto f^!$ gives rise to a $\QCoh(\bA^1/\bG_m)$-linear $\infty$-functor
$$F\cD: \cat{smSch}_k^{\opp} \rightarrow \cat{Pr}^L_k$$
and the restriction functor give rise to a map of functors $\res: F\cD \rightarrow \QCoh^{\fil}$.
\end{prop}
\begin{proof}
Let $f: W \rightarrow X$ and $g: X \rightarrow Y$ be maps of smooth schemes.  We need to produce an equivalence of functors $f^! g^! \rightarrow (g \circ f)^!$.  These functors are given by $(\wt\cD_W, \wt{\cD}_Y)$-bimodule objects, so we may instead produce an equivalence
$$(g f \times \mathrm{id}_Y)^*\wt\cD_Y \longrightarrow (f \times \mathrm{id}_X)^*\wt\cD_X  \otimes_{\wt{\cD}_X} (g \times \mathrm{id}_Y)^*\wt\cD_Y.$$
The map is given first by the identification $(g f \times \mathrm{id}_Y)^*\wt\cD_Y = (f \times \mathrm{id}_X)^*\wt\cO_X  \otimes_{\wt\cO_X} (g \times \mathrm{id}_Y)^*\wt\cD_Y$, then by the map to $(f \times \mathrm{id}_X)^*\wt\cD_X  \otimes_{\wt\cO_X} (g \times \mathrm{id}_Y)^*\wt\cD_Y$ induced by the inclusion $\wt\cO_X \hookrightarrow \wt\cD_X$, and then by the insertion map $ (f \times \mathrm{id}_X)^*\wt\cD_X  \otimes_{\wt\cO_X} (g \times \mathrm{id}_Y)^*\wt\cD_Y \rightarrow (f \times \mathrm{id}_X)^*\wt\cD_X  \otimes_{\wt{\cD}_X} (g \times \mathrm{id}_Y)^*\wt\cD_Y.$  It is a standard verification that this map is an equivalence.  To lift the functor to an $\infty$-functor, i.e. that composition $f_1^! f_2^! \cdots f_r^!$ is independent of grouping in a homotopy coherent way, it suffices to check the case where $r=3$, since the bimodule objects live in the heart the standard t-structure (see the proof of \cite[Prop. A.2.7]{CD}), which is a standard direct verification.   Compatibility with restriction follows by a similar 1-categorical verification in the heart of the t-structure.  To see that these functors are $\QCoh(\bA^1/\bG_m)$-linear, we may restrict the $\QCoh(\bA^1/\bG_m)$-action to the full subcategory of locally free sheaves on $\bA^1/\bG_m$ by \cite[Prop. A.1.13]{CD}, i.e. check linearity on $\Hom$s, which is by construction.
\end{proof}

\subsubsection{} We wish to now define the pushforward functor for filtered $\cD$-modules and establish a base-change result.  We have thus far labored to keep definitions as concrete as possible, but at this juncture we find it necessary to switch gears and make definitions in terms of the \emph{Hodge stack}.  We will show that the category of ind-coherent sheaves on the Hodge stack are equivalent to filtered $\cD$-modules as we have defined them; then, the results we need for pushforward and base-change will automatically follow from results in \cite{GR}.

\medskip

We briefly recall the main results of \cite{GR}.  Let $Y$ be a formal moduli problem under $X$, i.e. a nil-isomorphism $X \rightarrow Y$ where $Y$ is locally almost-finite-type and admits deformation theory \cite[\textsection II 5.1.3]{GR}.  In all our examples, $X$ will be locally almost-finite-type and admitting deformation theory with a map to $\pt = \Spec k$, and we take $X \rightarrow X^{\dR} := \wh{X}_{\pt}$ to be the map to the de Rham stack.  To such a set-up, one can define a \emph{formal deformation to the normal cone} construction \cite[\textsection II 9.2]{GR}, which is a prestack $\wh{\bN}_{X/Y}$ over $\bA^1/\bG_m$ whose generic fiber is $Y$, and whose special fiber is a formal moduli problem under $X$ realized as a colimit of split square-zero extensions by $\Symp^n \cT_{X/Y}$ \cite[\textsection II 7.1.4]{GR}.  Furthermore, $\bN_{X/Y}$ is a formal moduli problem under $X \times \bA^1/\bG_m$.  Sometimes, we write for any map of prestacks $f: X \rightarrow Y$ such that $Y$ is locally almost-finite-type and admits deformation theory $\wh{\bN}_{X/Y} := \wh{\bN}_{X/\wh{X}_Y}$.

\medskip

We define the Hodge stack, and filtered $\cD$-modules as follows.  This definition, and many subsequent notions, can be done more generally; we develop the bare minimum needed for our application.
\begin{defn}
Let $X$ be a scheme.  We define the \emph{Hodge stack}
$$X^{\Hod} = \wh{\bN}_{X/X^{\dR}}.$$
It is a prestack over $\bA^1/\bG_m$ with generic fiber $X^{\dR}$, and a formal moduli problem under $X \times \bA^1/\bG_m$ with structure map $i: X \times \bA^1/\bG_m \rightarrow X^{\Hod}$.  We call its special fiber the \emph{Dolbeault stack} $X^{\Dol}$, a stack over $B\bG_m$.  This association is functorial, i.e. if $f: X \rightarrow Y$ is a morphism of schemes, we have a morphism $f: X^{\Hod} \rightarrow Y^{\Hod}$ over $\bA^1/\bG_m$.

We define the category of \emph{right filtered $\cD$-modules} to be the category $\IndCoh(X^{\Hod})$ of ind-coherent sheaves on the Hodge stack, and the category of \emph{left filtered $\cD$-modules} to be the category $\QCoh(X^{\Hod})$ on the Hodge stack.  We have compatible \emph{forgetful functors} $i^!, i^*$ and \emph{side-changing functors} $\Upsilon$:
$$\begin{tikzcd}
\QCoh(X^{\Hod}) \arrow[r, "i^*"] \arrow[d, "\Upsilon", "\simeq"']  & \QCoh(X \times \bA^1/\bG_m) \arrow[d, "\Upsilon", "\simeq"']\\
\IndCoh(X^{\Hod}) \arrow[r, "i^!"] & \IndCoh(X \times \bA^1/\bG_m).
\end{tikzcd}$$
The functor $\Upsilon$ is the tautological functor which arises from the $\QCoh(Y)$-action on $\omega_Y \in \IndCoh(Y)$ for any prestack $Y$; in particular, it sends $\Upsilon(\cO_Y) = \omega_Y$.  The functor $i^!: \IndCoh(X^{\Hod}) \rightarrow \IndCoh(X)$ admits a left adjoint, $i_*$, the \emph{induction functor}.  The corresponding functor on left $\cD$-modules is given by $\Upsilon^{-1} \circ i_* \circ \Upsilon.$

Given a map $f: X \rightarrow Y$, we tautologically have \emph{pullback functors}
$$\begin{tikzcd}
\QCoh(X^{\Hod}) \arrow[r, "i_X^*"] & \QCoh(X) & \IndCoh(X^{\Hod}) \arrow[r, "i_X^!"] & \IndCoh(X)  \\
\QCoh(Y^{\Hod})\arrow[u, "f^*"]  \arrow[r, "i_Y^*"] & \QCoh(Y) \arrow[u, "f^*"] & \IndCoh(Y^{\Hod})\arrow[u, "f^!"]  \arrow[r, "i_Y^!"] & \QCoh(Y) \arrow[u, "f^!"] 
\end{tikzcd}$$  
compatible with side-changing.  We also have a \emph{pushforward functor} on right $\cD$-modules compatibly with induction
$$\begin{tikzcd}
\IndCoh(X) \arrow[r, "i_{X*}"] \arrow[d, "f_*"] & \IndCoh(X^{\Hod}) \arrow[d, "f_*"] \\
\IndCoh(Y) \arrow[r, "i_{Y*}"] & \IndCoh(Y^{\Hod})
\end{tikzcd}$$
which is defined in \cite[\textsection II 3.4.3]{GR} for $f$ any map between inf-schemes.  For left $\cD$-modules, we define $f_* = \Upsilon^{-1}_Y \circ f_* \circ \Upsilon_X$ by side-changing.
\end{defn}

We have the following, which is a combination of results from \cite[\textsection II 3.4.4]{GR}; for the notion of right adjointability see \cite[Def. 3.4.5]{german}.
\begin{prop}\label{dmod base change}
Suppose that $f: X \rightarrow Y$ is proper.  Then the functors 
$$\begin{tikzcd} \IndCoh(X^{\Hod}) \arrow[r, "f_*", shift left] & \arrow[l, "f^!", shift left]  \IndCoh(Y^{\Hod})\end{tikzcd}$$ 
are adjoint, and the square
$$\begin{tikzcd}
X' \arrow[r, "g'"]  \arrow["f'"', d] & \arrow[d, "f"] X \\
Y' \arrow[r,  "g"] & Y
\end{tikzcd}$$
satisfies the left Beck-Chevalley condition, i.e. is right adjointable, i.e. the natural map
$$f'_*g'^! \rightarrow f'_*g'^!f^!f_* \simeq f'_*f'^!g^!f_* \rightarrow g^!f_*$$
is an equivalence. 
\end{prop}
\begin{proof}
If $f: X \rightarrow Y$ is proper, then $f: X^{\Hod} \rightarrow Y^{\Hod}$ is proper.  The adjunction follows by \cite[II Cor. 5.2.3]{GR}, and base change by \cite[II Prop. 2.3.2]{GR}.
\end{proof}

Finally, we wish to identify the category of ind-coherent sheaves on $X^{\Hod}$ with filtered $\cD$-modules defined in the ``classical'' sense.  In particular, this means that we have well-defined pushforward functors and base change.
\begin{prop}\label{hodge id}
Let $X$ be a smooth scheme.  Under the equivalence by integral transforms
$$\cat{Fun}^L_{\QCoh(\bA^1/\bG_m)}(\QCoh(X), \QCoh(X)) \simeq \QCoh(X \times X)^{\fil}$$
the monad $i^!i_*$ is identified with the algebra object $\wt{\cD}_X$.  For $f: X \rightarrow Y$ a map of smooth schemes, the functor $f^!$ is identified with the $(\wt{\cD}_Y, \wt\cD_X)$-bimodule object $(f \times \mathrm{id}_Y)^! \wt{\cD}_Y  \in \QCoh(X \times Y)^{\fil}$.  In particular, there is a natural equivalence of functors
$$F\cD^r(-) \simeq \IndCoh((-)^{\Hod}): \cat{smSch}_k^{\opp} \rightarrow \cat{Pr}^L_k.$$
\end{prop}
\begin{proof}
We first pass to the formal groupoid attached to the formal moduli problem $i: X \times \bA^1/\bG_m \rightarrow X^{\Hod}$, i.e. we take the Cech nerve of $i$; in fact this is how the deformation to the normal cone is defined in the first place \cite[\textsection II 9.2.4]{GR}.  By construction the deformation to the normal cone comumutes with fiber products; thus
$$\mathrm{Cech}(\wh{\bN}_{X/X} \rightarrow \wh{\bN}_{X/X^{\dR}}) = \left( \begin{tikzcd} \cdots \arrow[r] \arrow[r, shift left] \arrow[r, shift right] & \wh{\bN}_{X/\wh{X \times X \times X}_{\Delta}} \arrow[r, shift left] \arrow[r, shift right] & \wh{\bN}_{X/\wh{X \times X}_\Delta} \arrow[r] & \wh{\bN}_{X/X} \end{tikzcd} \right).$$
On the other hand, for a regular closed embedding $Z \subset Y$, we have that $\wh{\bN}_{Z/Y}$ is a formal completion of the classical deformation to the normal cone, i.e. the Rees construction applied to the (decreasing, so negatively weight-graded) $\cI_Z$-adic filtration of $\cO_Y$.  In our setting, the functor $i^!$ is monadic since it is a nil-isomorphism, and the monad is identified with the integral kernel $\omega_{\wh{\bN}_\Delta}$.

We write $k^{\fil} = \bA^1/\bG_m$ the filtered point, and $Y^{\fil} = Y \times k^{\fil}$ for any stack $Y$; we have $\omega_{Y^\fil/k^{\fil}} = \omega_Y \otimes k[t]$ (with no weight shift).  Next, we show that there is a natural equivalence
$$\omega_{\wh{\bN}_\Delta/k^{\fil}} \simeq \wt{\cD}_X \otimes_{\cO_X} \omega_X.$$
The dualizing sheaf $\omega_{(X \times X)^{\fil}/k^{\fil}}$ is concentrated in degree $-2n$, so one can check locally that its $!$-restriction to the formal neighborhood, i.e. local cohomology along a regular embedding of codimension $n$, is in degree $-n$.  The iterated convolutions are also in degree $-n$, i.e.
$$(\wt{\cD}_X \otimes_{\cO_X} \omega_X) \otimes^!_{\wt\cO_X} (\wt{\cD}_X \otimes_{\cO_X} \omega_X) \simeq \wt\cD_X \otimes_{\cO_X} \wt{\cD}_X \otimes_{\cO_X} \omega_X$$
so we may understand the algebra structure 1-categorically.  By the cohomological calculation, we can write the dualizing sheaf as a colimit of dualizing sheaves over infinitesimal neighborhoods:
$$\omega_{\wh{\bN}_\Delta/k^\fil} = \colim H^{-n}(\omega_{\wh{\bN}_\Delta^{(n)}/k^\fil})[n], \;\;\;\;\;\;\;\;\;\; \wh{\bN}_\Delta^{(n)} := \Spec_{X \times X} \bigoplus_{-n \leq k \in \bZ^{\leq 0}} \cI_{\Delta}^k t^{-k}.$$
We choose a projection $\pi: \wh{\bN}_\Delta^{(n)} \rightarrow X^{\fil}$ and view $\wh{\bN}_\Delta^{(n)}$ as a finite flat (thus affine) scheme over $X^{\fil}$, so that 
$$H^{-n}(\omega_{\wh{\bN}_\Delta}^{(n)})[n] \simeq H^{-n}(\omega_{X^{\fil}/k^{\fil}} \otimes_{\wt\cO_X} \omega_{\wh{\bN}_\Delta^{(n)}/X^\fil})[n] \simeq \omega_X \otimes_{\cO_X} \intHom_{\wt\cO_X}(\cO_{\wh{\bN}_\Delta^{(n)}}, \wt\cO_X).$$
It is then a standard calculation \cite{lurie crystals} that $\intHom_{\wt\cO_X}(\cO_{\wh{\bN}_\Delta^{(n)}}, \wt\cO_X) \simeq \wt{\cD}^{\leq n}_X$ by a pairing $\wt{\cD}^{\leq n}_X \otimes_{\wt\cO_X} \cO_{\wh{\bN}_\Delta^{(n)}} \rightarrow \wt\cO_X$, and since all objects are in the heart the algebra structure is a straightforward 1-categorical verification by computing the dual to pullback maps on structure sheaves.

For compatibility with pullback, using Proposition \ref{bimod prop} it suffices to identify the corresponding $(\wt{\cD}_Y, \wt{\cD}_X)$-bimodule object in $\QCoh(X \times Y)^{\fil}$.  Applying the functor to $\wt{\cD}_Y$ we see that the bimodule must be in the heart of the standard t-structure.  We may then verify Zariski locally on $Y$, where the bimodule is determined by its value on the generator $\wt{\cD}_Y$.
\end{proof}

\subsubsection{} 

Our aim is now to establish smooth descent for filtered $\cD$-modules; our strategy will be to deduce it from the usual descent for $\cD$-modules and descent for sheaves on the cotangent bundle.  The following is standard; see Proposition 3.5.1 in \cite{ho li}.
\begin{prop}\label{modcat hom}
Let $\cat{A}$ be a compactly generated symmetric monoidal $\infty$-category, and $\cat{M}$ a module category.  Then, the $\Hom$-spaces of $\cat{M}$ are naturally enriched as objects of $\cat{A}$.  Furthermore, if $\Phi: \cat{A} \rightarrow \cat{B}$ be a monoidal functor of rigid compactly-generated symmetric monoidal $\infty$-categories, then 
$$\Hom_{\cat{M} \otimes_{\cat{A}} \cat{B}}(X, Y)) = \Phi(\uHom_{\cat{M}}(X, Y)).$$
\end{prop}

In view of Proposition \ref{modcat hom}, the computation of internal $\Hom$s allows us to deduce results on categories over $\bA^1/\bG_m$ from results at its special and generic fiber.
\begin{prop}\label{prop A1 reduce}
Let $F: \cat{C} \rightarrow \cat{D}$ be a colimit-preserving functor of presentable $\QCoh(\bA^1/\bG_m)$-module categories such that $F \otimes_{\QCoh(\bA^1/\bG_m)} \QCoh(\bG_m/\bG_m)$ and $F \otimes_{\QCoh(\bA^1/\bG_m)} \QCoh(\{0\}/\bG_m)$ are equivalences.  Then, $F$ is an equivalence.
\end{prop}
\begin{proof}
Since the $\Hom$-spaces of $\cat{C}, \cat{D}$ are naturally enriched as objects of $\QCoh(\bA^1/\bG_m)$, by Proposition \ref{modcat hom}, for full faithfulness it suffices to check the claim for objects in $\QCoh(\bA^1/\bG_m)$.  Every object in $M \in \QCoh(\bA^1/\bG_m)$ sits in an exact triangle $R\Gamma_{\{0\}/\bG_m}(M) \rightarrow M \rightarrow M|_U$, so it suffices to show that the local cohomology of $M$ at 0 vanishes if and only if the derived stalk at 0 does, which is a standard argument by spectral sequences (see, for example, the proof of Proposition 2.2.13 in \cite{CD}).  For essential surjectivity, it suffices to show that the quotient $\cat{C}/\cat{D} \not\simeq 0$, i.e. it suffices to show that if an object restricts to zero on both fibers, then it is zero.  An object $X$ is zero if and only if $\End(X) \simeq 0$, and the claim follows by the same argument as full faithfulness.
\end{proof}

Using the above, we can deduce the desired descent statement.
\begin{prop}\label{d mod descent}
The functor in Proposition \ref{dmod pullback prop} is a sheaf in the smooth topology.
\end{prop}
\begin{proof}
For the sheaf property, by Proposition \ref{prop A1 reduce} it suffices to check at the generic and special points.  The generic point, which we identify via Proposition \ref{dmod heart}, is the claim for $\cD$-modules \cite[Prop. 3.2.2]{crystals and dmod}. The special fiber is $\QCoh(\bT^*_X/\bG_m)$, where the claim follows by Theorem \ref{graded functoriality thm} and Proposition \ref{odd tangent indcoh descent}.
\end{proof}

\subsubsection{} Using descent, we can define filtered $\cD$-modules on stacks.
\begin{defn}\label{def filt dmod stack}
Let $X$ be a smooth Artin 1-stack.  We define the category of \emph{filtered $\cD$-modules} on $X$ to be
$$F\cD(X) := \lim_{\substack{U \rightarrow X \\ \mathrm{smooth}}} F\cD(U).$$
By Proposition \ref{dmod pullback prop} the smooth pullback is $\QCoh(\bA^1/\bG_m)$-linear and compatible with restriction functors.  Thus, we have associated graded and forgetful functors
$$\begin{tikzcd}
\cD(X) & & F{\cD}(X) \arrow[ll, "\un"'] \arrow[rr, "\gr"] & & \QCoh(\mathbb{T}^*_X/\bG_m)
\end{tikzcd}$$
and a restriction functor $\res: F\cD(X) \rightarrow \QCoh(X)^{\fil}$ which commutes with colimits and limits by construction, thus has a left adjoint, which we define to be the induction functor $\ind: \QCoh(X)^{\fil} \rightarrow F\cD(X)$.  Furthermore, for a smooth representable map $f: X \rightarrow Y$ of Artin 1-stacks, we can define the pullback $f^!: F\cD(Y) \rightarrow F\cD(X)$ by descent.  For a proper representable map $f: X \rightarrow Y$ of Artin 1-stacks, we can define the pushforward $f_*: F\cD(X) \rightarrow F\cD(Y)$ by passing to left adjoints, descent, and base change, following Propositions \ref{hodge id} and \ref{dmod base change}.
\end{defn}

\medskip

We now define small subcategories of $F\cD(X)$ when $X$ is a stack.  Unlike the case of schemes, there are two different small subcategories we may consider.
\begin{defn}\label{def coh d mod}
We define the following small full subcategories of $F\cD(X)$ when $X$ is a stack.
\begin{enumerate}
\item We define the full subcategory of \emph{safe $\cD$-modules} $F\cD_s(X) \subset F\cD(X)$ to be the idempotent-completion of the essential image of $\Coh(X)^{\fil}$ under the induction functor (similar to the definition in Definition \ref{filt dmod schemes}); by the argument in \cite[Thm. 8.1.1]{QCA} these are exactly the compact objects of $F\cD(X)$, which is compactly generated.  
\item We define the full subcategory of \emph{coherent $\cD$-modules} $F\cD_c(X)$ to consist of those objects $\cF \in F\cD(X)$ such that for any smooth map $p: U \rightarrow X$ from a smooth affine scheme $U$, we have $p^! \cF \in F\cD_c(U)$.
\end{enumerate}
\end{defn}
%

\subsection{Koszul duality}\label{kd sec}

In this section we prove the Koszul duality equivalence.  We first discuss the case of schemes; this is a calculation which is well known (e.g. \cite[\textsection 1.8]{kapranov}\cite[\textsection B.4]{coderived}) but we need to take care to produce a statement at the level of $\infty$-categories.  We then deduce the statement for stacks.

\subsubsection{} \label{kd schemes sec} We define certain well-known resolutions and explicit formulas that appear in Koszul duality; see  \cite[Lem. 1.5.27]{HTT} for the analogous resolutions in the non-filtered setting, and \cite[\textsection 8.4.13]{MHM} for filtered versions.

\begin{defn}\label{spencer}
Recall our weight conventions: $|\delta| = |\Omega^1_X| = -1$ and $|t| = 1$.  We denote by $\dddot{\Omega}^k_X$ the weight $-k$ (degree $k$) part of the $\cO_X$-bimodule algebra $\dddot{\Omega}^\bullet_X$.  Explicitly, we have isomorphisms
$$\dddot{\Omega}^k_X \simeq \Omega^k_X \oplus \delta \Omega^{k-1}_X  \simeq \Omega^k_X \oplus  \Omega^{k-1}_X \delta \;\;\;\;\;\;\;\;\;\; \omega + \delta \omega' = \omega +  \omega' \delta + (-1)^{|\omega|}d\omega'.$$
Likewise, we let $\wt{\cD}_X^k = \cD_X^{\leq k} t^k$ denote the weight $k$ part of the Rees algebra.  Let 
$$\mathrm{coev}: \cO_X \longrightarrow \dddot\Omega^1 \otimes_{\cO_X} \wt\cD^1_X$$
be the coevaluation map (note the tensor over $\cO_X$ and not $\wt\cO_X$).  We define the \emph{deformed de Rham complex} $d\mathcal{R}_X$ to be the weight-graded $(\dddot{\Omega}^\bullet_X, \wt{\cD}_X)$-bimodule with underlying bigraded complex of sheaves $\dddot{\Omega}^\bullet_X \otimes_{\OO_X} \wt{\cD}_X$ with internal differential $d: \dddot{\Omega}^k_X  \otimes_{\cO_X} \wtD_X \rightarrow \dddot{\Omega}^{k+1}_X \otimes_{\cO_X} \wtD_X$ given by $\omega \otimes u \mapsto \omega \, \mathrm{coev}(1) u$, i.e. in local coordinates
\begin{equation}
\omega \otimes u \mapsto  \omega\delta  \otimes u t + \sum_{i=1}^n \omega \, dx_i \otimes \frac{\partial}{\partial x_i}u t.
\end{equation}
We note that for a vector field $\theta$ and top-dimensional differential form $\omega$, the differential sends
\begin{equation}\label{lie derivative}
\iota_\theta(\omega) \otimes 1 \mapsto \iota_\theta(\omega) \delta  \otimes t + \omega \otimes \theta t = (\delta \iota_\theta(\omega) + (-1)^{|\omega|}\cL_\theta(\omega)) \otimes t + \omega \otimes \theta t
\end{equation}
where $\iota_\theta$ is contraction with $\theta$ and $\cL_\theta$ is the Lie derivative along $\theta$.  We define two $(\wt{\cD}_X, \dddot{\Omega}^\bullet_X)$-bimodules: the locally $\dddot\Omega^\bullet_X$-semifree \emph{deformed Koszul complex} $\cK_X$ is the $\dddot{\Omega}^\bullet_X$-linear graded dual of $d\cR_X$, and the locally $\wt\cD_X$-free \emph{deformed Spencer complex} $\cS_X$ is the $\wt{\cD}_X$-linear graded dual of $d\cR_X$.    We can write $\cS_X = \wt\cD_X \otimes_{\cO_X} \dddot\Omega_X^\vee$, where $\dddot\Omega_X^\vee = \intHom_{\cO_X}(\dddot\Omega_X, \cO_X)$ is the right dual, i.e. as right $\cO_X$-modules, has differential
$$u \otimes \eta \mapsto ut \otimes \eta(- \cdot \delta) + \sum u \displaystyle \frac{\partial}{\partial x_i} t \otimes \eta(- \cdot dx_i).$$
For a vector field $\theta$ and dual-to-top-dimensional $\eta$, the differential sends
\begin{equation}\label{lie derivative dual}
1 \otimes (\eta \circ \iota_\theta) \mapsto t \otimes \eta(\iota_\theta(-))\cdot \delta + t \otimes  (-1)^{|\omega|}\eta(\cL_\theta(-)) + \theta t \otimes \eta.
\end{equation}
\end{defn}
%



The following result is standard (e.g. in \cite{kapranov}); one writes down the natural filtration on the resolutions defined above, and checks that their associated gradeds are sums of Koszul resolutions. 
\begin{lemma}\label{koszul acyclic}
The complexes $\cS_X$ and $\cK_X$ are both quasi-isomorphic to $\OO_X$ as $(\wtD_X, \dddot{\Omega}^{\bullet}_X)$-modules.  The complex $d\cR_X$ is quasi-isomorphic to $\wt{\omega}_X \otimes k[\delta] t^{-1}$.
\end{lemma}
\begin{proof}
Let $d = \dim(X)$.  We first consider for $d\cR_X$; in weights $-1, 0, 1$ the complex looks like:
$$\begin{tikzcd}[row sep=tiny]
 & \dddot\Omega_X^1 \otimes \wt\cD^0_X \arrow[r] & \cdots  \arrow[r] & \dddot\Omega^d_X \otimes \wt\cD_X^{d-1} \arrow[r] & \dddot\Omega^{d+1} \otimes \wt{\cD}_X^{d} \\
\cO_X \arrow[r] & \dddot\Omega_X^1 \otimes \wt\cD_X^1 \arrow[r]&  \cdots \arrow[r] &  \arrow[r] \dddot\Omega^{d}_X \otimes \wt\cD_X^{d} \arrow[r] & \dddot\Omega^{d+1} \otimes \wt{\cD}_X^{d+1} \\
\wt{\cD}_X^1 \arrow[r] & \dddot\Omega_X^1 \otimes \wt\cD_X^2 \arrow[r] & \cdots \arrow[r] & \dddot\Omega^{d}_X \otimes \wt\cD_X^{d+1} \arrow[r] & \dddot\Omega^{d+1}_X \otimes \wt{\cD}_X^{d+2}.
\end{tikzcd}$$
Each row is a fixed internal weight $r$, and each column a fixed cohomological degree.  The multiplication by $t$ map induces a injective map from one weight to the next, thus to compute the cohomology of the complex we can compute the cohomology of the associated graded pieces:
$$\Theta_X^r  \rightarrow \dddot\Omega^1_X \otimes \Theta_X^{r+1} \rightarrow \cdots \rightarrow \dddot\Omega_X^{d} \otimes \Theta_X^{r+d} \rightarrow \dddot\Omega^{d+1}_X \otimes \Theta^{r+d+1}$$
which are acyclic except when $r=-d-1$, where the complex is precisely $\dddot\Omega^{d+1}_X$.  Note that $t$ acts freely and $\delta$ commutes with $\dddot\Omega^{d+1}_X$, giving the claimed description.  We leave the argument for $\cS_X$ and $\cK_X$ to the reader.
\end{proof}

\subsubsection{} We now establish Koszul duality for schemes.  Similar statements have appeared in many forms in the literature \cite{kapranov, coderived,  loops and conns, foliations 2}.
\begin{thm}\label{filtered kd schemes}\label{kd schemes}
Let $X$ be a smooth scheme.  We have compatible equivalences of $\cat{Vect}^{\omega B\bG_a^{\gr}} \simeq \QCoh(\bA^1/\bG_m)^{\shear}$-module categories
$$\begin{tikzcd}
\IndCoh(\bT_X[\ng1])^{\omega B\bG_a} \arrow[r, "\simeq"', "\kappa"] \arrow[d, shift left, "z^!"]  & F^r\cD(X)^{\shear}  \arrow[d, "\res", shift left] \\
\QCoh(X)^{\omega B\bG_a} \arrow[u, shift left, "z_*"] \arrow[r, "\simeq"] & \QCoh(X)^\fil \arrow[u, shift left, "\ind"]
\end{tikzcd}$$
Furthermore, the equivalence $\kappa$, and the restriction functors $z^!$ and $\res$, are compatible with pullback.  When $f$ is proper, by passing to left adjoints, the equivalence is compatible with pushforward.
\end{thm}
\begin{proof}
Consider the diagram of adjoint functors
$$\begin{tikzcd} 
 \arrow[r, shift left, "\eta"] \IndCoh(X) \otimes \cat{Vect}^{\bG_m} & \arrow[r, shift left, "z_*"] \IndCoh(X) \otimes \cat{Vect}^{\omega B\bG_a^{\gr}} \arrow[l, "\eta^R", shift left] & \Ind(\Coh(\bT_X[\ng1])^{B\bG_a^{\gr}}) \arrow[l, "z^!", shift left]
 \end{tikzcd}$$
where $\eta$ is induced by the quotient map on group stacks $B\bG_a^{\gr} \rightarrow \bG_m$ and $\eta^R$ is its right adjoint, guaranteed to exist since $\eta$ is continuous; $\eta^R$ is furthermore continuous since $\eta$ preserves compact objects by definition of the renormalization in Definition \ref{ren G inv}.  We claim both $\eta^R$ and $z^!$ (thus their composition) are monadic; they preserve colimits since their left adjoints preserve compactness, so it remains to argue that they are conservative.  For $\eta^R$, we may identify it with the pushforward functor along $\bA^1/\bG_m \rightarrow B\bG_m$ by shearing by Proposition \ref{ren bga}, which is evidently conservative since $\bA^1$ is affine.  For $z^!$, the claim follows by first reducing to the claim for compact objects, then noting that the forgetful functor for $B\bG_a^{\gr}$-equivariance is conservative, and the claim that the underlying non-equivariant functor $z^!$ is conservative since $z: X \rightarrow \bT_X[\ng1]$ is surjective on geometric points when $X$ is a scheme \cite[Prop. 8.1.2]{indcoh}.

Thus we may apply Barr-Beck-Lurie to the map of monads $\eta^R \eta \rightarrow \eta^R z^! z_* \eta$ on $\IndCoh(X \times B\bG_m)$ to deduce a diagram
$$\begin{tikzcd}
\IndCoh(\bT_X[\ng1])^{\omega B\bG_a} \arrow[r, "\simeq"', "\kappa"] \arrow[d, shift left, "z^!"]  & \Mod_{\QCoh(X)^{\bG_m}}(\eta^R z^! z_* \eta) \arrow[d, "\res", shift left] \\
\QCoh(X)^{\omega B\bG_a} \arrow[u, shift left, "z_*"] \arrow[r, "\simeq"]  \arrow[d, shift left, "\eta^R"] & \Mod_{\QCoh(X)^{\bG_m}}(\eta^R\eta) \arrow[u, shift left, "\ind"] \arrow[ld, shift left] \\
\QCoh(X)^{\bG_m}. \arrow[u, shift left, "\eta"] \arrow[ur, shift left]
\end{tikzcd}$$
It remains to compute the monads explicitly, and the map between them.  The monad $p^* p_*$ is simply tensoring by $\cO(\bA^1) \simeq k[t]$ by Proposition \ref{ren bga}.  Under the identification in Proposition \ref{identify omega}, we identify the monad (using the fact that $\cO_X$ is a compact object in the renormalized category):
$$\eta^R z^! z_* \eta(-) \simeq R\intHom_{\dddot{\Omega}^{-\bullet}_X}(\cO_X, -) \simeq R\intHom_{\dddot{\Omega}^{-\bullet}_X}(\cO_X, \cO_X) \otimes_{\cO_X} -.$$
We apply the Tate unshearing and use the Koszul resolution in Lemma \ref{koszul acyclic} to compute
$$R\intHom_{\dddot\Omega^\bullet_X}(\cO_X, \cO_X) \simeq \intHom_{\dddot\Omega^\bullet_X}(\cK_X, \cO_X) \simeq \cO_X \otimes_{\dddot\Omega_X} d\cR_X \simeq \wt{\cD}_X.$$
We require this identification as algebra objects in the category $\QCoh(X \times X)^{\fil}$.  It is clear that they agree at the level of homotopy categories.  Since $\wt{\cD}_X$ is an algebra object in the heart, by \cite[Prop. A.2.8]{CD} this suffices, and we deduce the main claim of the theorem, as well as compatibility with induction and restriction.

We now deduce compatibility with pullback.  Consider the diagram
$$\begin{tikzcd}[column sep=15ex]
\Mod_{\QCoh(X)}(\dddot\Omega^\bullet_X) \arrow[r, "- \otimes_{\dddot{\Omega}_X} d\cR_X"', "\kappa_X"] & F^r\cD(X)  \\
\Mod_{\QCoh(Y)}(\dddot\Omega^\bullet_Y) \arrow[u, "(df^!)^{\shear}"] & F^r\cD(Y). \arrow[u, "f^!"'] \arrow[l, "- \otimes_{\wt{\cD}_Y} \cS_Y", "\kappa^{-1}_Y"'] 
\end{tikzcd}$$
We produce an equivalence between the $\infty$-functors $f^!$ and $\kappa_X \circ (df^!)^\shear \circ \kappa_Y^{-1}: F^r\cD(Y) \rightarrow F^r\cD(X)$  
by producing a natural equivalence on the corresponding $(f^{-1}\wt\cD_Y, \wt{\cD}_X)$-modules on $\QCoh(X \times Y)^{\fil}$:
\begin{equation}\label{kd funct eq}
f^{-1}\cS_Y \tens{f^{-1}\dddot\Omega_Y} d\cR_X \longrightarrow f^{-1}\wt{\cD}_Y \tens{f^{-1} \wt{\cO}_Y}   f^{-1}\wt\omega_Y^{-1} \tens{\wt\cO_X} \wt{\omega}_{X}
\end{equation}
where the left $f^{-1}\wt\cD_Y$-actions and the right $\wt{\cD}_X$-action on the source are the obvious ones, and the right $\wt{\cD}_X$-action on the target is given by $(m \otimes \eta \otimes \omega) \cdot \theta = (m \cdot f_* \theta) \otimes \omega \otimes \eta + m \otimes (- f_*\theta \cdot \eta \otimes \omega + \eta \otimes \omega \cdot \theta)$.  As a right $\dddot\Omega_Y$-module, $\cS_Y$ is quasi-isomorphic to $\wt\cD_Y \otimes_{\cO_Y} (\Omega^d_Y)^{\vee}[\dim Y]$; in particular, both sides are concentrated in the same cohomological degree, which gives the map.  The relations on the left are imposed by the differential; using the relations (\ref{lie derivative}) and (\ref{lie derivative dual}) the right $\wt\cD_X$-actions agree.  It is clear that the left actions agree, and that the map is an equivalence.  Naturality may be verified 1-categorically since all objects live in the heart of the standard t-structure.
\end{proof}

\begin{exmp}
In the case where $Y = \Spec k$ is a point, (\ref{kd funct eq}) reduces to the claim
$$k \tens{k\langle \delta \rangle} d\cR_X \overset{\simeq}{\longrightarrow} k \tens{k[\delta]} \wt{\omega}_X \overset{\simeq}{\longrightarrow} \wt{\omega}_X$$
which follows from the calculation in Lemma \ref{koszul acyclic}.  Namely, the underlying chain complex on the left is $\Omega^\bullet_X \otimes_{\cO_X} \wt\cD_X$ with differential
$$\omega \otimes u \mapsto \omega \delta \otimes ut + \sum \omega dx_i \otimes \displaystyle\frac{\partial}{\partial x_i} u = (-1)^{|\omega| + 1}d\omega + \sum \omega dx_i \otimes \displaystyle\frac{\partial}{\partial x_i} u.$$
One can perform a similar calculation to check that the underlying complex is $\wt\omega_X$.  The relations are given by the differential from $\Omega^{\dim X - 1}_X \otimes_{\cO_X} \wt\cD_X$, and by (\ref{lie derivative}) the right action of $\wt\cD_X$ is by the Lie derivative.
\end{exmp}

\begin{exmp}
Filtered Koszul duality exchanges the following objects.
\bgroup
\def\arraystretch{1.5}
\begin{center}\begin{tabular}{|c||c|c|c|c|} \hline
$\Omega$ side & $\cO_X$ & $\Omega^\bullet_X$ & $\dddot\Omega^\bullet_X$ & $\Omega^\bullet_{X/Y}$    \\ \hline 
$\cD$ side & $\wt{\cD}_X$ & $\wt{\omega}_X$ & $\wt{\omega}_Xt^{-1}[\ng1]$ & $f^{-1} \wt\cD_Y \otimes  f^{-1}\wt{\omega}_Y^{-1} \otimes \wt{\omega}_X$ \\ \hline
\end{tabular}\end{center}
\egroup
\noindent In particular, this gives an alternative calculation of the cyclic homology of $\QCoh(X)$ as a sum of the Hodge truncated de Rham cohomologies:
$$\cO(\cL X)^{S^1} \simeq \Hom_{\cL X}(\cO_{\cL X}, \cO_{\cL X})^{S^1} \simeq \uHom_{F\cD(X^{\fil})}(\wt{\omega}_X, \wt{\omega}_X)^{\shear} \simeq \left( \bigoplus_{k \in \bZ} \Gamma(X, \Omega^{\geq k}_{X, d})t^{-k}\right)^{\shear}$$
where the $\uHom$ indicates the internal graded $k[t]$-linear $\Hom$. 
\end{exmp}

\subsubsection{} \label{sec kd stacks} We now deduce Koszul duality for stacks via descent.
\begin{thm}[Coherent Koszul duality for stacks]\label{kd kperf}
Let $X$ be a smooth QCA stack.  We have compatible equivalences
$$\begin{tikzcd}
\hCoh(\bT_X[\ng1])^{\Tate^{\mathrm{gr}}} \arrow[d, "\simeq"'] & \hCoh(\bT_X[\ng1])^{\BGA} \arrow[r] \arrow[d, "\simeq"'] \arrow[l] & \hCoh(\bT_X[\ng1]/\bG_m) \arrow[d,"\simeq"'] \\
\cD_c(X) & F\cD_c(X)  \arrow[r, "\gr"'] \arrow[l, "\un"] & \Coh(\mathbb{T}_{X}^*/\bG_m)
\end{tikzcd}$$
functorial with respect to schematic pullback, induction and restriction functors.  The equivalences exchange the notions of singular support.  When $f$ is proper, by passing to left adjoints, the equivalence is compatible with pushforward.
\end{thm}
\begin{proof}  
First, we define for any smooth scheme $U$ the functor
$$\Psi_U: \IndCoh(\bT_U[\ng1])^{\omega B\bG_a} \rightarrow \QCoh(\bT_U[\ng1])^{B\bG_a^{\gr}}$$ 
to be the unique colimit-preserving functor whose restriction to $\Coh(\bT_U)^{B\bG_a^{\gr}}$ is the fully faithful inclusion 
$$\Coh(\bT_U[\ng1])^{B\bG_a^{\gr}} \hookrightarrow \QCoh(\bT_U[\ng1])^{B\bG_a^{\gr}}.$$ 
The forgetful functor intertwines this functor with the canonical functor (abusively denoted):
$$\Psi_U: \IndCoh(\bT_U[\ng1]) \rightarrow \QCoh(\bT_U[\ng1])$$
defined in \cite{indcoh}.  We now establish the middle equivalence, with the other two following by passing to the generic and special fibers.  Choose a smooth atlas $p: U \rightarrow X$.  First, recall the descent statement from  \cite[Thm. 6.6]{loops and conns} for quasi-coherent sheaves on odd tangent bundles:
\begin{equation}\label{bn descent}
\QCoh(\bT_X[\ng1])^{B\bG_a^{\gr}} \simeq \Tot(\QCoh(\bT_{U_\bullet}[\ng1])^{B\bG_a^{\gr}}).
\end{equation}
The claim follows from the diagram:
$$\begin{tikzcd}[column sep=12ex]
\QCoh(\wh{\bT}_X[\ng1])^{\BGA} \arrow[r, "\simeq", "{(\ref{bn descent})}"'] & \Tot(\QCoh(\wh{\bT}_{U_\bullet}[\ng1])^{\BGA})  & \arrow[l, "\Tot(\Psi_{U_\bullet} \circ \kappa^{\ng1})"']  \Tot(F\cD(U_\bullet)) \arrow[r, "\simeq", "{\mathrm{Def.} \ref{def filt dmod stack}}"']& F\cD(X)\\
\hCoh(\wh{\bT}_X[\ng1])^{\BGA} \arrow[r, "\simeq", "{\substack{\mathrm{Prop.} \ref{ff limit} \\ \mathrm{Prop.} \ref{graded functoriality thm}}}"'] \arrow[u, hook] & \Tot(\Coh(\bT_{U_\bullet}[\ng1])^{\BGA}) \arrow[u, hook]& \arrow[l, "\simeq"', "{\mathrm{Thm.} \ref{kd schemes}}"] \Tot(F\cD_c(U_\bullet)) \arrow[r, "\simeq","\substack{\mathrm{Def.} \ref{def coh d mod} \\ {\mathrm{Prop.} \ref{ff limit}}}"']\arrow[u, hook]& F\cD_c(X).\arrow[u, hook]
\end{tikzcd}$$
That is, we use the characterization of $\wh{\Coh}$ in Proposition \ref{graded functoriality thm} and the fact that the equivalences $\kappa'_{U_\bullet}$ from Theorem \ref{kd schemes}  restrict coherent objects.  The compatibility for pullback and restriction follows by descent, induction by passing to left adjoints, and the claim for singular supports follows since singular supports can be computed on atlases.
\end{proof}

\subsubsection{} There are two large-category versions of the above theorem.  The first one is ``renormalized'', i.e. we simply ind-complete the categories appearing in the above theorem.  We denote the category of \emph{ind-coherent filtered $\cD$-modules} by $F\rnD(X) := \Ind(F\cD_c(X))$.
\begin{cor}[Ind-coherent Koszul duality for stacks]
Let $X$ be a smooth QCA stack.  We have compatible equivalences
$$\begin{tikzcd}
\Ind\!\hCoh(\wh{\bT}_X[\ng1])^{\grTate} \arrow[d, "\simeq"'] & \Ind\!\hCoh(\wh\bT_X[\ng1])^{\omega\BGA} \arrow[r] \arrow[d, "\simeq"'] \arrow[l] & \Ind\!\hCoh(\wh\bT_X[\ng1]/\bG_m) \arrow[d, "\kappa", "\simeq"'] \\
\rnD(X) & F\rnD(X)  \arrow[r, "\gr"'] \arrow[l, "\un"] & \IndCoh(\mathbb{T}_{X}^*/\bG_m)
\end{tikzcd}$$
compatibly with schematic pullback, induction and restriction, and exchanging the notions of singular support.  When $f$ is proper, by passing to left adjoints, the equivalence is compatible with pushforward.
\end{cor}

We also have the following unrenormalized version. 
\begin{thm}[Koszul duality for stacks]
Let $X$ be a smooth QCA stack.  We have compatible equivalences
$$\begin{tikzcd}
\IndCoh(\wh{\bT}_X[\ng1])^{\grTate} \arrow[d, "\simeq"'] & \IndCoh(\wh{\bT}_X[\ng1])^{\omega\BGA} \arrow[r] \arrow[d, "\simeq"'] \arrow[l] & \IndCoh(\wh{\bT}_X[\ng1])^{\G_m} \arrow[d, "\kappa", "\simeq"'] \\
\cD(X) & F\cD(X)  \arrow[r, "\gr"'] \arrow[l, "\un"] & \QCoh(\mathbb{T}_{X}^*)^{\G_m} 
\end{tikzcd}$$
compatibly with schematic pullback, induction and restriction, and exchanging the notions of singular support.  When $f$ is proper, by passing to left adjoints, the equivalence is compatible with pushforward.
\end{thm} 
\begin{proof}
The result follows from Theorem \ref{kd kperf}, and the observation that the compact objects in $F\cD(X)$ are generated by inductions from $\Coh(X)$, while $\Coh(\wh{\bT}_X[\ng1])$ is generated by objects pushed forward from $\Coh(X)$.
\end{proof}

\begin{rmk}\label{kd weak}
There is a category of weakly $G$-equivariant $\cD$-modules that sits between strongly $G$-equivariant $\cD$-modules and non-equivariant $\cD$-modules:
$$F\cD_c(X/G) \longrightarrow F\cD_c(X\brslash G) \longrightarrow F\cD_c(X).$$
On the Koszul dual side, we have a corresponding sequence of maps and functors
$$\bT_X[\ng1] \longrightarrow \bT_X[\ng1]/G \longrightarrow \wh{\bT}_{X/G}[\ng1],$$
$$\Coh(\bT_X[\ng1])^{B\bG_a^{\gr}} \longrightarrow \Coh(\bT_X[\ng1]/G)^{B\bG_a^{\gr}} \longrightarrow \hCoh(\wh{\bT}_{X/G}[\ng1])^{B\bG_a^{\gr}}.$$
\end{rmk}

\subsection{Example: classifying stacks.}\label{dmod BG exmp} We describe the main theorem in the case of classifying stacks $X = BG$, where $G$ is an affine algebraic group.  In the following we will always use \emph{degree} to refer to cohomological degree, and \emph{weight} to refer to internal $\bG_m$ weight.  We will use the Tate (un)shearings from Definition \ref{tate shear}:
$$M^{\unshear} = \bigoplus_{n \in \bZ} M_n[2n], \;\;\;\;\;\;\;\;\;\; N^{\shear} = \bigoplus_{n \in \bZ} N_n[-2n].$$
We wish to explain the categories and functors that appear in the following diagram.
$$\begin{tikzcd}[column sep=15ex]
\Coh(\asterisk)^{B\bG_a^{\gr}} = \Mod^{\gr}_{\mathrm{fd}}(C_\bullet(S^1; k)) \arrow[r, shift left=1ex, "{\Hom_{C_\bullet(S^1; k)}(k, -)}", "\simeq"'] & \Mod^{\gr}_{\mathrm{fg}}(C^\bullet(BS^1; k)) \arrow[l, shift left=1ex, "{- \otimes_{C^\bullet(BS^1; k)} k}"] \arrow[r, shift left=1ex, "\unshear"] & \arrow[l, shift left=1ex, "\simeq"', "\shear"] F\cD_c(\asterisk) = \Mod_{\mathrm{fg}}(\Rees(k)) \\
\wh{\Coh}(\wh{\mf{g}}/G)^{B\bG_a^{\gr}} \arrow[rr, shift left=1ex, "\simeq"', "\unshear \circ \kappa"] \arrow[u]  & & F\cD_c(BG). \arrow[u] \arrow[ll, shift left=1ex, "\kappa^{-1} \circ\, \shear"] 
\end{tikzcd}$$
The top row is the usual Koszul duality between $C_\bullet(S^1;k)$-modules and $C^\bullet(BS^1; k)$-modules, followed by a shearing identifying graded $C^\bullet(BS^1; k)$-modules with the Rees algebra for the trivial filtration on $k$.  We fix quasi-isomorphisms and coordinates with degree-weight:
$$C_\bullet(S^1; k) \simeq k[\lambda] \text{ where } |\lambda| = (-1, -1),$$ 
$$C^\bullet(BS^1; k) \simeq k[u] \text{ where } |u| = |t^{\shear}| = (2, 1),\;\;\;\;\;\;\;\;\;\;\;\Rees(k) \simeq k[t] \text{ where } |t| = |u^{\unshear}| = (0, 1).$$
%

\subsubsection{} Next, we give an explicit description of the category $\hCoh(\wh{\mf{g}}/G)^{B\bG_a^{\gr}}$ as follows.  Its objects consist of weight-graded bounded coherent complexes $(M^\bullet, d_M)$ with the following structure:
\begin{enumerate}[(a)]
\item a $\mf{g}$-representation structure, i.e. assigning $\mf{g}$ weight 1, a weight $1$ action map $\alpha_M: \mf{g} \otimes M^\bullet \rightarrow M^\bullet$  commuting with the differentials,
\item a locally nilpotent $\cO(\mf{g})$-module structure, i.e. assigning $\mf{g}^*$ weight -1, a weight 0 action map $a_M: \mf{g}^* \otimes M^\bullet \rightarrow M^\bullet$ commuting with the differentials, and
\item a square-zero nullhomotopy of the universal endomorphism\footnote{Morally, this universal endomorphism of sheaves on $\mf{g}/G$ is the derivative of the universal automorphism on sheaves on the loop space $\cL(BG) = G/G$ given by the $S^1$-action, i.e. $g \in G$ fixes $g$ under the adjoint action, and for $\cF \in \QCoh(G/G)$ the automorphism on the fiber $\cF_g$ is by $g$-equviariance.} of $\Coh(\mf{g}/G)$, i.e. a weight $-1$ degree $-1$ map 
$$\delta_M \in \Hom^{-1}_k(M^\bullet, M^\bullet)  \hspace{3ex} \text{ such that } \hspace{3ex}  \delta_M^2 = 0 \;\;\text{ and }\;\; [d_M, \delta_M] := d_M \delta_M - (-1)^{|\bullet|} \delta_M d_M = a_M \circ \gamma_M$$
where $\gamma_M: \gamma_V: V \rightarrow V \otimes \mf{g}^*$ is the representation $\mf{g}^*$-coaction, i.e. the weight -1 dual to $\alpha_M$.
\end{enumerate}
Furthermore, these structures are required to be compatible: $a_M$ is required to be a map of $\mf{g}$-representations where $\mf{g}^*$ is viewed as the dual adjoint representation, and $\delta_M$ is required to be a map of $\mf{g}$-representations and $\cO(\mf{g})$-modules.

\subsubsection{} The pullback functor on odd tangent bundles forgets the $G$-structure and $!$-restricts to $0 \in \mf{g}$.  To compute this $!$-restriction we use the following \emph{curved} Koszul resolution of the augmentation $\cO(\mf{g})$-module:
$$\cK_{\mf{g}}^\bullet := \left(\begin{tikzcd}
\cO(\mf{g}) \otimes \det(\mf{g}^*) \arrow[r, shift left] & \arrow[l, shift left] \cdots \arrow[r, shift left]  & \arrow[l, shift left] \cO(\mf{g}) \otimes \Wdg^2 \mf{g}^* \arrow[r, shift left] & \cO(\mf{g}) \otimes \mf{g}^* \arrow[r, shift left, "a_{\cO(\mf{g})}"]  \arrow[l, shift left] & \cO(\mf{g}) \arrow[l, shift left, "\gamma_{\cO(\mf{g})}"] 
\end{tikzcd}\right) \in \Coh(\wh{\mf{g}}/G)$$
where the module $\mf{g}$-action map $a_{\cO(\mf{g})}$ is extended via the Leibniz rule, and the representation $\mf{g}^*$-coaction is extended via the dual Leibniz rule (and squares to zero due to associativity of the action).  The $!$-restriction, i.e. $\Hom_{\cO(\mf{g})}(\cK^\bullet_{\mf{g}}, \cM^\bullet)$, is the complex
$$(M^\bullet \otimes \Symp(\mf{g}[-1]), d := d_M \otimes \mathrm{id} + \wt{c}_M, \delta := \delta_M \otimes \mathrm{id} + \mathrm{id} \otimes \alpha_M)$$ 
where $\wt{c}_M: M^\bullet \rightarrow M^\bullet \otimes \mf{g}$ is the action by the Euler vector field, i.e, $\wt{c}_M(m) = \sum x_i \cdot m \otimes y_i$ where $x_i$ and $y_i$ are dual bases of $\mf{g}^*$ and $\mf{g}$.  One can verify that $[d, \delta] = 0$.    Finally, one can apply the usual Koszul duality functor and Tate unshearing to obtain:
$$(M^{\bullet \unshear} \otimes \Symp(\mf{g}[1]) \otimes k[t], d := d_M \otimes \mathrm{id} + \wt{c}_M + \delta_M t \otimes \mathrm{id}  + \mathrm{id} \otimes \alpha_M t).$$

\subsubsection{} The objects of the category $F\cD_c(BG)$ of coherent filtered $\cD$-modules are bounded coherent $k[t]$-complexes $(N^\bullet, d_N)$ of $G$-representations along with a weight 0 degree -1 morphism
$$h_N: \mf{g} \otimes N^\bullet \rightarrow N^\bullet \hspace{3ex} \text{ such that } \hspace{3ex} h_N^2 = 0 \;\;\text{ and }\;\; [d_N, h_N] := d_Nh_N - (-1)^{|\bullet|} h_Nd_N = \alpha_N t$$
witnessing the strong equivariance condition homotopically.  The pullback functor on filtered $\cD$-modules simply forgets both the $G$-representation structure and the map $h$, i.e. takes the underlying $k[t]$-module $N^\bullet$.  We let the weight 1 degree -1 map $h_y: N^\bullet \rightarrow N^\bullet$ denote the evaluation at $y \in \mf{g}$.

\subsubsection{} Then, to describe the Koszul duality functor $\wh{\Coh}(\wh{\mf{g}}/G)^{\BGA} \rightarrow F\cD_c(BG)$, we may consider the functor
$$\wh{\Coh}(\wh{\mf{g}}/G)^{\BGA} \rightarrow \Coh(\asterisk)^{\BGA} \rightarrow F\cD_c(\asterisk)$$
described above, but instead we do not forget the $G$-structure and recover the map $h_y$ as multiplication by $y \in \mf{g}[1]$.

\subsubsection{} We now turn our attention to the case $X = BG$ where $G$ is commutative, where there is a simple geometric description of the categories on the $\Omega$-side of the equivalence in terms of matrix factorizations; for details, the reader may consult \cite[\textsection 6]{MF} or \cite{ihes}.  The derived loop space, i.e. the adjoint quotient decomposes into a product $\cL(BG) \simeq G \times BG$.  We assume the existence of a Cartier dual group, i.e. an algebraic group $\wc{G}$ and a monodial equivalence
$$(\QCoh(BG), \otimes) \simeq (\QCoh(\wc{G}), \otimes).$$
In particular, this determines equivalences
$$\Coh(\cL(BG)) \simeq \Coh(G \times \wc{G}), \;\;\;\;\;\;\;\;\;\; \Coh(\wh{\bT}_{BG}[\ng1]) \simeq \Coh(\wh{\mf{g}} \times \wc{G}).$$
Note that the $S^1 = B\bZ$, and that the Cartier dual of $\bZ$ is $\bG_m$, while the Cartier dual of $B\bG_a$ is $\wh{\bG}_a$.  In particular any $S^1$ action or $B\bG_a^{\gr}$-action has a Cartier dual description as an action of the monoidal category $\QCoh(\bG_m)$ or $\QCoh(\wh{\bG}_a)$.  In this setting, by passing through Cartier duality the actions have geometric origin via pullback along the maps
$$\mathrm{ev}: G \times \wc{G} \rightarrow \bG_m, \;\;\;\;\; (g, \chi) \mapsto \chi(g); \;\;\;\;\;\;\;\;\;\; \mathrm{ev}: \wh{\mf{g}} \times \wc{G} \rightarrow \wh{\bG}_a, \;\;\;\;\; (x, \chi) \mapsto \chi(\exp(x)).$$
This leads to explicit descriptions of the $B\bG_a$-equivariant categories:
$$\Coh(\bT_{BG}[\ng1])^{B\bG_a^{\gr}} \simeq \Coh(\{0\} \times_{\wh{\bA}^1} (\wh{\mf{g}} \times \wc{G})), \;\;\;\;\;\;\;\;\;\;\; \Coh(\bT_{BG}[\ng1])^{\Tate} \simeq \mathrm{MF}(\wh{\mf{g}} \times \wc{G}, \mathrm{ev}).$$
We refer the reader to \cite{MF, toby} for a discussion of matrix factorizations.   We describe the Koszul duality equivalence explicitly in two opposite examples: commutative reductive groups, and commutative unipotent groups.

\subsubsection{}  When $G = \bG_m$, we identify $X^\bullet(G) \simeq \bZ$ the via the Cartier duality equivalence $\IndCoh(\wh{\mf{g}}/G) \simeq \IndCoh(\wh{\bG}_a \times \bZ)$ where the $B\bG_a$-action is given by the map
$$\mf{g} \times \widecheck{G} = \wh{\bG}_a \times \bZ \rightarrow \wh{\bA}^1, \;\;\;\;\;\;\;\;\;\;\; (x, n) \mapsto nx.$$
Then we have
$$\wh{\Coh}(\wh{\mf{g}}/G)^{\BGA} \simeq \bigoplus_{n \in \bZ} \Mod_{x\mh\mathrm{nil}, \mathrm{f.g.}}(k[x, \delta], d(\delta) = nx), \;\;\;\;\;\;\;\;\;\; |x| = (0, -1), \;|\delta| = (-1, -1)$$
$$F\cD_c(X) \simeq \bigoplus_{n \in \bZ} \Mod_{\mathrm{f.g.}}(k[t, h_x], d(h_x) = nt), \;\;\;\;\;\;\;\;\;\; |t| = (0, 1), \; |h_x| = (-1, 1).$$
We view $(x, t)$ and $(\delta, h_x)$ as dual variables; in particular, inverting $t$ kills all blocks except the trivial block.  Koszul duality exchanges the following objects, where $k(n)$ denotes the weight $n$ character for $G = \bG_m$, $\iota^{(m)}: \mf{g}^{(m)} \hookrightarrow \mf{g}$ is the closed embedding of the $m$th infinitesimal neighborhood of $\{0\} \subset \mf{g}$ and $\iota: \wh{\mf{g}} \hookrightarrow \mf{g}$ is the inclusion of the formal neighborhood:
\begin{table}[H]
\centering
\begin{tabular}{c|c|c|c|c}
& $\wh{\Coh}(\wh{\mf{g}}/G)$ & $\wh{\Coh}(\wh{\mf{g}}/G)^{\BGA}$ & $F\cD_c(X)$  & $\cD_c(X) \simeq\Mod_{\mathrm{fg}}(C_\bullet(S^1))$ \\ \hline 
& & & & \\[-1em]
$n \ne 0$ & $\iota^{(0)}_* \cO_{\{0\}/G}(n)$ & $k(n)$ & $k[t]$  & $0$\\
& & & & \\[-1em]
\tablefootnote{Note that $\iota^{(m)}_* \cO_{\mf{g}^{(m)}/G}(n)$ does not admit a $B\bG_a$-equivariant structure when $n \ne 0$, since the module does not live on the derived zero fiber of the function $nx$.  Also note that the $C_\bullet(S^1)$-module $C_\bullet(S^{2m+1})$ is not formal, i.e. isomorphic to its cohomology, since the $(m+2)$-ary $A_\infty$-action of $H^1(S^1)$ on $H^0(S^{2m+1})$ is non-trivial.} & $\iota^{(m)}_* \cO_{\mf{g}^{(m)}/G}$ & $k[x]x^{-m}/k[x]x$ & $C_\bullet(S^{2m+1})[t]$ & $C_\bullet(S^{2m+1})$ \\
& & & & \\[-1em]
& $\iota^! \omega_{\mf{g}/G} \simeq \omega_{\wh{\mf{g}}/G}$ & $k[x,x^{-1}]/k[x]x$ & $C_\bullet(ES^1)[t] \simeq k[t]$ & $C_\bullet(ES^1) \simeq k$ \\ 
& & & & \\[-1em]
& $\iota^{(m)}_* \cO_{\mf{g}^{(n)}/G} \otimes C_\bullet(S^1)$  & $k[x, \delta]x^{-m}/k[x]x $ & $C_\bullet(S^{2m+1})[t]/t^{m+1}$ & $0$ \\
& & & & \\[-1em]
& $\iota^! \cO_{\mf{g}/G} \otimes C_\bullet(S^1)$ & $k[x, x^{-1}, \delta]/x$ & $C_\bullet(ES^1) \simeq k$ & $0$
\end{tabular}
\end{table}
\noindent On the $\Omega$-side, we have free resolutions and nullhomotopies (indicated by dashed arrows) of objects:
$$\iota^{(0)}_*\cO_{\{0\}/G}(n) \simeq \left(\begin{tikzcd} k(n)[x] \arrow[r, "x", shift left] & k(n)[x]  \arrow[l, dashed, "n", shift left] \end{tikzcd}\right), \;\;\;\;\;\;\; \iota^{(0)}_* \cO_{\{0\}/G} \otimes C_\bullet(S^1) \simeq \left(\begin{tikzcd} k[x] \arrow[r, "0", shift left] & k[x]  \arrow[l, dashed, "1", shift left] \end{tikzcd}\right).$$
Passing to the Tate categories, these define matrix factorizations which are isomorphic to zero, except in the case $\iota^{(0)}_*\cO_{\{0\}/G}(n)$ when $n=0$.

\subsubsection{}\label{BGa exmp} When $G = \bG_a$, Cartier duality gives an equivalence $\IndCoh(\wh{\mf{g}}/G) \simeq \IndCoh(\wh{\bG}_a \times \wh{\bG}_a)$ where the $B\bG_a$-action is given by the map
$$\mf{g} \times \widecheck{G} = \wh{\bG}_a \times \wh{\bG}_a \rightarrow \wh{\bA}^1, \;\;\;\;\;\;\;\;\;\; (x, y) \mapsto xy.$$
Here, we view $x \in \mf{g}^*$ as a coordinate for the Lie algebra of $\bG_a$, and $y$ as the nilpotent operator determining the $G = \bG_a$-action via the exponential map.  We have
$$\wh{\Coh}(\wh{\mf{g}}/G)^{\BGA} \simeq \Mod_{x,y\mh\mathrm{nil}, \mathrm{f.g.}}(k[x, y, \delta], d(\delta) = xy)$$
$$F\cD_c(X) \simeq \Mod_{y\mh\mathrm{nil}, \mathrm{f.g.}}(k[t, y, h_x], d(h_x) = y t).$$
In particular, inverting $t$ gives $\cat{Vect}_{k, \mathrm{f.g.}}$ on both sides.  This matches the expectation the $\cD(B\bG_a) \simeq \cat{Vect}_k$ since $\bG_a$ is contractible.  We note, however, that \emph{filtered} $\cD$-modules on $B\bG_a$ may still be nontrivial.  Koszul duality exchanges the following objects, where we preserve the notation from the previous example:
\begin{table}[H]
\centering
\begin{tabular}{c|c|c|c|c}
& $\wh{\Coh}(\wh{\mf{g}}/G)$ & $\wh{\Coh}(\wh{\mf{g}}/G)^{\BGA}$ & $F\cD_c(X)$  & $\cD_c(X) \simeq \cat{Vect}_k$ \\  \hline
& & & & \\[-1em]
\tablefootnote{Letting $V_n$ denote the unique indecomposible $(n+1)$-dimensional representation of $\bG_a$, the objects $\iota^{(m)}_*\cO_{\{0\}/G} \otimes V_n$ admit $S^1$-equivariant structures if and only if $n = 0$ and $m = 0$, i.e. we require $xy \in (x^m, y^n)$.  However, there are ``diagonal'' objects where the extension as a $\bG_a$-representation and the extension as a $k[x]$-module interact nontrivially.} & $\iota^{(0)}_* \cO_{\{0\}/G}$ & $k[x,y]/(x, y)$ & $(\Symp^\bullet \mf{g}[1]) \otimes k[t]$  & $k \oplus k[1]$\\
& & & & \\[-1em]
\tablefootnote{This object corresponds to letting $m \rightarrow \infty$, i.e. the dualizing sheaf on the $x$-axis.  We do not allow the corresponding object where we let $n \rightarrow \infty$, i.e. the dualizing sheaf on the $y$-axis, corresponding to the ind-regular representation of $\bG_a$.  However, the resulting matrix factorization is isomorphic to that coming from $\omega_{\wh{\mf{g}/G}}[1]$.} & $\iota^!\omega_{\mf{g}/G} \simeq \omega_{\wh{\mf{g}}/G}$ & $k[x,y]/y$ & $k[t]$ & $k$ \\
\end{tabular}
\end{table}
\noindent The object $\iota^{(0)}_* \cO_{\{0\}/G}$  may be presented via the free resolution with nullhomotopies:
$$\iota^{(0)}_* \cO_{\{0\}/G} \simeq \left(\begin{tikzcd}[ampersand replacement=\&, column sep=huge, every label/.append style = {font = \scriptsize}] k[x, y] \arrow[r, shift left, "{\frac 1 2 \begin{pmatrix} y \\ -x \end{pmatrix}}"] \& \arrow[l, shift left, dashed, "{\frac 1 2 \begin{pmatrix} x & -y \end{pmatrix}}"] k[x, y]^2 \arrow[r, "{\frac 1 2 \begin{pmatrix} x & y \end{pmatrix}}", shift left] \& k[x, y]  \arrow[l, dashed, "{\frac 1 2 \begin{pmatrix} y \\ x \end{pmatrix}}", shift left] \end{tikzcd}\right)$$
which gives rise to the matrix factorization associated to the skyscraper sheaf at the singular point of the affine nodal curve $\Spec k[x,y]/xy$, while $\omega_{\wh{\mf{g}}/G}$ gives rise to the matrix factorization associated to the skyscraper sheaf along the $x$-axis. In particular, we note that $\iota^{(0)}_* \cO_{\{0\}/G}$ and $\omega_{\wh{\mf{g}}/G} \oplus \omega_{\wh{\mf{g}}/G}[1]$ are non-isomorphic as filtered $\cD$-modules, but give rise to isomorphic $\cD$-modules.


\end{document}